\documentclass[11pt,reqno,a4paper, english]{amsart}
\usepackage[top=3.5cm,bottom=3cm,left=2.5cm,right=2.5cm]{geometry}
\usepackage[mathscr]{euscript}

\usepackage{amsmath, mathtools, amsthm, amsfonts, amssymb,enumerate,xcolor}
\usepackage{microtype}
\usepackage{hyperref}
\hypersetup{colorlinks={true},linkcolor={blue},citecolor=blue}
\usepackage{comment}

\usepackage[style=nature,sorting=nyt,doi=false,isbn=false,url=false,eprint=false,intitle=true]{biblatex}
\DeclareNameAlias{author}{family-given}
\numberwithin{equation}{section}

\theoremstyle{plain}
\newtheorem{theorem}{Theorem}[section]
\newtheorem{proposition}[theorem]{Proposition}
\newtheorem{corollary}[theorem]{Corollary}
\newtheorem{lemma}[theorem]{Lemma}
\newtheorem{observation}[theorem]{Observation}

\theoremstyle{definition}
\newtheorem{definition}[theorem]{Definition}

\theoremstyle{remark}

\DeclareMathOperator{\supp}{supp}

\DeclareMathOperator{\e}{e}

\newcommand{\N}{\mathbb{N}}
\newcommand{\Z}{\mathbb{Z}}

\newcommand{\R}{\mathbb{R}}
\newcommand{\C}{\mathbb{C}}

\newcommand{\im}{\mbox{Im}}
\newcommand{\re}{\mbox{Re}}

\newcommand{\abs}[1]{\vert #1 \vert}
\newcommand{\absb}[1]{\bigl\vert #1 \bigr\vert}
\newcommand{\absbig}[1]{\biggl\vert #1 \biggr\vert}
\newcommand{\norm}[1]{\| #1 \|}

\newcommand{\normbig}[1]{\biggl\| #1 \biggr\|}
\newcommand{\parent}[1]{\bigl(#1\bigr)}
\newcommand{\parentbig}[1]{\biggl(#1\biggr)}
\newcommand{\set}[1]{\bigl\{#1\mathclose{}\bigr\}}

\newcommand{\bigo}[1]{O\mathopen{}\left(#1\right)}
\newcommand{\japbrak}[1]{\langle#1\rangle}

\newcommand{\interval}[4]{\mathopen{#1}#2\mathclose{}\mathpunct{},#3\mathclose{#4}}
\newcommand{\intervaloo}[2]{\interval{(}{#1}{#2}{)}}
\newcommand{\intervaloc}[2]{\interval{(}{#1}{#2}{]}}
\newcommand{\intervalco}[2]{\interval{[}{#1}{#2}{)}}
\newcommand{\intervalcc}[2]{\interval{[}{#1}{#2}{]}}

\newcommand{\disp}{\displaystyle}

\begin{document}


\title[Scattering for the 3D NLS with random initial data]{Scattering for the cubic Schrödinger equation in 3D with randomized radial initial data}

\author{Nicolas Camps}
\address{Universit\'e Paris-Saclay, Laboratoire de mathématiques d'Orsay, UMR 8628 du CNRS, B\^atiment 307, 91405 Orsay Cedex,}
\email{nicolas.camps@universite-paris-saclay.fr}



\date{\today}

\begin{abstract}
We obtain almost-sure scattering for the cubic defocusing Schr{\"o}dinger equation in the Euclidean space $\mathbb{R}^3$, with randomized radially-symmetric initial data at some supercritical regularity scales. Since we make no smallness assumption, our result generalizes the work of B{\'e}nyi, Oh and Pocovnicu~\cite{benyi2015}. It also extends the results of Dodson, L{\"u}hrmann and Mendelson~\cite{dodson-luhrmann-mendelson-19} on the energy-critical equation in $\mathbb{R}^4$, to the energy-subcritical equation in $\mathbb{R}^3$. In this latter setting, even if the nonlinear Duhamel term enjoys a stochastic smoothing effect that makes it subcritical, it still has infinite energy. In the present work, we first develop a stability theory from the deterministic scattering results below the energy space, due to Colliander, Keel, Staffilani, Takaoka and Tao. Then, we propose a globalization argument in which we set up the I-method with a Morawetz bootstrap in a stochastic setting. To our knowledge, this is the first almost-sure scattering result for an energy-subcritical Schrödinger equation outside the small data regime. 
\end{abstract} 
\ \vskip -1cm  \hrule \vskip 1cm \vspace{-8pt}
 \maketitle 
{ \textwidth=4cm \hrule}

\vspace{-6pt}
\tableofcontents
\pagebreak{}
\section{Introduction}
We consider the initial value problem for the cubic defocusing Schrödinger equation
\begin{equation}
\label{eq:nls}
\tag{NLS}
\begin{cases}
i\partial_t u(t,x) +\Delta u(t,x) = \absb{u(t,x)}^2u(t,x),\quad (t,x)\in\R\times\R^3\,.\\
u(0,x) = u_0(x)\,,
\end{cases}
\end{equation}
The equation has a Hamiltonian structure associated with the energy
\begin{equation}
    \label{eq:energy}
    E(u)= \frac{1}{2}\int_{\R^3}\abs{\nabla u(\cdot,x)}^2dx+\frac{1}{4}\int_{\R^3}\abs{u(\cdot,x)}^4dx\,. 
\end{equation}
For smooth solutions, the above energy is conserved as time evolves. The mass
\[
M(u)= \int_{\R^3}\abs{u(\cdot,x)}^2dx\,
\]
is also formally conserved. In addition, the equation enjoys the scaling symmetry. Namely, given $\lambda>0$ and $u$ a solution to the equation,
\[
u_\lambda(t,x) = \lambda^{-1}u(\lambda^{-2}t,\lambda^{-1}x)\,,\quad (t,x)\in\R\times\R^3\,,
\]
is also a solution. Since the scaling symmetry leaves invariant the homogeneous Sobolev space $\dot{H}^{1/2}$, the scale $s_c=1/2$ is \textit{critical}. In the \textit{scaling-subcritical} regime $s>1/2$, we deduce from a contraction mapping argument that the initial value problem is \textit{locally well-posed} in the following sense:
\begin{itemize}
    \item (Uniqueness) Given two solutions $u_1, u_2$ in $H^s$ and a time $t$, if $u_1(t)=u_2(t)$ then $u_1=u_2$.
    \item (Existence) For all $u_0\in H^s$, there exists $T(\norm{u_0}_{H^s})>0$ and a solution to~\eqref{eq:nls}
    \[
    u(t) = \e^{it\Delta}u_0 - i \int_0^t \e^{i(t-\tau)\Delta}\abs{u(\tau)}^2u(\tau)d\tau\,,\quad u\in C(\intervalcc{-T}{T};H^s)\,.
    \]
    \item (Continuity) The solution map $u_0\in H^s\mapsto u \in C(\intervalcc{-T}{T},H^s)$ is continuous.
\end{itemize}
    
\subsection{Known results}

Note that in the \textit{scaling-subcritical} regime, the maximal lifespan of the solution is bounded from below by a quantity that only depends on the $H^s$ norm of the initial data. Therefore, one can use a continuity argument and deduce global existence from a priori uniform bound on the $H^s$ norm of the solution.

In the present case, $s_c<1$ and the equation is \textit{energy-subcritical}. Note that an initial data in $H^1$ have a finite energy. Besides, it follows from the defocusing nature of the equation that the energy is coercive, and yields a uniform control on the $H^1$ norm of the solution so that~\eqref{eq:nls} is globally well-posed in the energy space $H^1$. In addition, the solution display \textit{linear scattering} in large time. Namely, for all $u_0\in H^1$, there exist some unique final states $u_\pm\in H^1$ such that 
\[
\underset{t\to\pm\infty}{\lim} \norm{\e^{-it\Delta}u(t)-u_\pm}_{H^1}=0\,.
\]
It is conjectured that~\eqref{eq:nls} is globally well-posed in $H^s$ for $1/2\leq s<1$, and that the solution scatters at infinity. Still, we do not have a coercive conservation law at these regularities and, in general, we cannot always preclude the $H^s$ norm from blowing up in finite time. In the \textit{scaling-critical} case, the time of existence of the solution does not only depend on its norm $\dot{H}^{1/2}$, but also on its profile. Nevertheless, by using a concentration-compactness method, Kenig and Merle proved in~\cite{kenig-merle-10} some \textit{conditional} scattering.

At \textit{scaling-supercritical} scales $s<1/2$, the equation is ill-posed (see for instance~\cite{christ-colliander-tao}). Nevertheless, following the pioneering work of Bourgain, a \textit{probabilistic Cauchy theory} has emerged. It displays the existence of a local flow defined on \textit{generic} sets of initial data, in some regimes where a deterministic Cauchy theory does not hold. In the present work, we address the globalization of such a probabilistic flow for~\eqref{eq:nls} in the Euclidean space, with randomized initial data at supercritical regularities $3/7<s$.

\subsubsection*{Global well-posedness \& Scattering below the energy space}
\label{paragraph-intro}
Let us present the general framework used to study~\eqref{eq:nls} at scaling-subcritical regularity, below the energy space, where we still have a local flow. At regularities $1/2\leq \sigma <1$, there is no coercive conservation law that prevents the Sobolev norm to grow up. Hence, to obtain global a priori estimates for a general initial data~\footnote{\ Note that one can easily prove scattering under smallness assumption on the initial data at regularity $1/2\leq \sigma <1$, by directly performing a contraction argument and obtaining global spacetime bounds.}, one has to make use of some modified energies. The idea is to consider a frequency truncated version of the solution, and to control the increments of its energy along the evolution. Bourgain implemented this idea in~\cite{bourgain-98} to prove global existence for the cubic Schrödinger equation in $H^\sigma(\R^2)$, for $2/3<\sigma$, by using the \textit{high-low} argument together with some bilinear estimates. The observation is that the nonlinear evolution of the high frequencies is well approximated by the linear evolution. Namely, if we evolve on a short time step the low-frequency modes by the nonlinear flow, and the high-frequency modes by the linear flow, the error we make turns out to be in the energy space $H^1$. Furthermore, the methods displays a \textit{nonlinear smoothing effect}: for all $t\in\R$,
\begin{equation}
u(t) - \e^{it\Delta}u_0 \in H^1\,. \label{eq:smoothing-effect}
\end{equation}
This globalization method is now standard, and was used in many contexts (see, e.g.,~\cite{kenig-ponce-vega-00} for the wave equation). Bourgain extended the method to the 3D case in~\cite{bourgain-98-3D}, where he proved global existence for $11/13<\sigma$. Then, Colliander, Keel, Staffilani, Takaoka and Tao refined this frequency-cutoff in~\cite{ckstt-2004} and introduced the I-operator, which goes from $H^\sigma$ to $H^1$. By exploiting the structure of the nonlinearity, the authors provide some frequency cancellations that come with a nonlinear smoothing effect, and that loosen the regularity threshold at which the local flow extends globally. 
\begin{definition}[I operator]
Given $1/2<\sigma<1$ and $N\gg1$, the I-operator is the Fourier multiplier 
\begin{equation}
    \label{eq:def:I}
    I = \mathcal{F}^{-1}(m\mathcal{F}),\quad m(\xi)=\begin{cases}
    1\quad &\text{if}\ \abs{\xi}\leq N,\\
    \parent{\frac{N}{\abs{\xi}}}^{1-\sigma}\quad &\text{if}\ \abs{\xi}\geq 2N\,.
    \end{cases}
\end{equation}
It holds that 
\begin{equation}
    \label{eq:norm:I}
    \norm{\nabla Iu}_{L^2}\leq\norm{u}_{H^\sigma}\,,\quad \norm{u}_{H^\sigma}\leq N^{1-\sigma}\norm{Iu}_{H^1}\,.
\end{equation}
\end{definition}
As $N$ becomes larger, the energy is supposed to grow to infinity, while we expect it's time derivative to decrease. Specifically, if we assume that $J\times\R^3$ is a spacetime slab where we have a smallness assumption on the $L_{t,x}^4$ norm of $u$, then 
\begin{equation}
\label{eq:almost-conservation-ckstt}
\int_J\abs{\frac{d}{dt}E(Iu)(\tau)}d\tau \lesssim \frac{1}{N} + \bigo{\frac{1}{N^2}}\,. 
\end{equation}
The method usually comes with a rescaling argument, that allows one to assume that the energy of $Iu$ is less than $1$ initially. In addition, the \textit{interaction Morawetz inequality} gives a bound on the $L_{t,x}^4$ norm, depending on the energy of $Iu$. Hence, we have a control on the number of spacetime slabs where the increment of $\mathcal{E}(v)$ satisfies~\eqref{eq:almost-conservation-ckstt}. By finely tuning the parameters $N$ and $s$, we would  be able to sum~\eqref{eq:almost-conservation-ckstt} over finitely many spacetime slabs, and to obtain a uniform bound on both the energy of $Iu$ and on the $L_{t,x}^4$ norm of $u$, from which scattering follows. The interaction Morawetz estimate reads 
\[
\int_J\int_{\R^3}\abs{u(t,x)}^4dxdt \lesssim \norm{u_0}_{L^2}^2\parent{\underset{t\in J}{\sup}\norm{u(t)}_{\dot{H}^{1/2}}}^2\,.
\]
It is due to~\cite{ckstt-2004}, and its proof consists in averaging over $y$ the recentered Morawetz action against the mass density, and in tracking some monotonicity. The I-method with a Morawetz bootstrap has been refined, and adapted in different contexts. Let us cite some works that are relevant to our problem. By the use of linear-nonlinear decomposition of the initial data,~\cite{dodson-13} reached the exponent $5/7<s$, further improved to $2/3<s$ in~\cite{qingtang-11} by combining the arguments of Dodson with a second generation I-method. In the present work, we implement the following key idea. To improve the results of~\cite{ckstt-2004}, Dodson provided a refined version of the energy increment estimate~\eqref{eq:almost-conservation-ckstt}, by keeping some sub-additive quantities on the right-hand-side: for some constant $c$ independent of $N$, 
\[
\int_J\abs{\frac{d}{dt}E(Iu)(\tau)}d\tau\lesssim \frac{1}{N^{1-}}\norm{\nabla P_{>cN}Iu}_{L_t^2(J;L_x^6(\R^3))}^2+\bigo{\frac{1}{N^2}}\,.
\] 
For radial initial data, this comes together with the so-called \textit{long-time Strichartz estimate}, that yields a uniform a priori estimate on $\norm{\nabla P_{>cN}Iu}_{L_t^2L_x^6}^2$. This strategy avoids loosing too many powers of $N$ when brutally multiplying the energy increments by the number of intervals where we have a control. Dodson implemented this idea in~\cite{dodson-13,dodson19} to prove scattering for $s>5/7$, and for $s>1/2$ in the radial case. However, the conjecture is still open for general data. Before ending this paragraph, we formulate the optimal result known at this date. We will largely use it throughout our analysis.~\footnote{\ Some stronger restrictions will appear in the analysis, and we shall only use~\eqref{eq:uniform-l5-bound} when $\sigma>6/7$. The high-low method from Bourgain~\cite{bourgain-98-3D} already settle global existence at these range of regularities, and the original version of the I-method from~\cite{ckstt-2004} yields the global spacetime estimates~\eqref{eq:uniform-l5-bound} and scattering.}
\begin{theorem}[Global well-posedness below the energy space,~\cite{ckstt-2004,dodson-13,qingtang-11}]
\label{theorem:deterministic-gwp}
Let $2/3<\sigma<1$ and $u_0\in H_x^\sigma(\R^3)$. The Cauchy problem~\eqref{eq:nls} with $u(t_0)=u_0$ is globally well-posed in and the solution scatters at infinity in $H_x^\sigma(\R^3)$. Moreover, there exists a constant $C(\norm{u_0}_{H_x^\sigma})>0$ such that 
\begin{equation}
\label{eq:uniform-l5-bound}
    \norm{u}_{L_{t,x}^5(\R\times\R^3)}\leq C(\norm{u_0}_{H_x^\sigma})\,.
\end{equation}
\end{theorem}
\subsubsection{Probabilistic Cauchy theory }
Given a randomized initial data $f_0^\omega$, which is basically the superposition of modes decoupled by multiplying each of them with 
a Gaussian variable, the probabilistic approach consists in performing a linear-nonlinear decomposition in the Duhamel integral representation, and in studying the nonlinear term
\[
v(t) = u(t)-\e^{it\Delta}f_0^\omega = \int_0^t\e^{i(t-t')\Delta}\parent{\abs{\e^{it'\Delta}f_0^\omega}^2\e^{it'\Delta}f_0^\omega}dt' + \cdots\,.
\]
The remaining terms correspond to higher order Duhamel iterations. It turns out that the probabilistic decoupling enhances the nonlinear smoothing effect ~\eqref{eq:smoothing-effect}, even in scaling-supercritical regimes and, in the favorable cases, we expect $v$ to be at critical, or subcritical regularity, where we have some local existence result. Besides, $v$ is solution to ~\eqref{eq:nls} with a stochastic forcing term. 
\begin{equation}
    \tag{NLS$_f$}
    \label{eq:nls:f}
    \begin{cases}
    i\partial_tv+\Delta v = \abs{v}^2v + \parent{\abs{v+f}^2(v+f)-\abs{v^2}v}\,,\quad (t,x)\in\R\times\R^3\,. \\
    v(t_0) = v_0\in H^\sigma\,,
    \end{cases}
\end{equation}
At time $t=0$, we have $v(0)=0$. Hence, we can exploit the deterministic Cauchy theory at scaling-subcritical regularities, at least locally in time, and prove almost sure local well-posedness by some perturbative arguments. This strategy corresponds to the Da Prato Debussche trick~\cite{da-prato02}. In the context of dispersive equations, it was first implemented by Bourgain for the 2D Schrödinger equation on the torus~\cite{bourgain94}, and then by Burq and Tzvetkov for the wave equation on compact manifolds without boundaries~\cite{burq-tzvetkov-2008I,burq-tzvetkov-2008II}. In such contexts, the randomization procedure is based on the spectral resolution of the Laplacian. However, on the Euclidean space, we do not have such natural decompositions for the initial data. It turns out that the so-called \textit{Wiener randomization} proved to be suitable, in the sense that linear evolution of Wiener randomized initial data displays improved spacetime integrability, and the nonlinear Duhamel term benefits from a nonlinear probabilistic smoothing that makes it subcritical.
\begin{definition}[Wiener randomization]
Let $f_0\in L_x^2(\R^3)$ be radially-symmetric. We define the unit-scale Wiener decomposition of $f_0$ by 
\[
f_0\sim \sum_{k\in\Z^3}Q_kf_0\,,
\]
where $Q_k$ is the Fourier projector associated with a smooth multiplier localized on the unit cube centered around $k\in\Z^3$. Next, given a sequence of real valued mean zero and independent complex Gaussian variables $\{g_k\}_{k\in\Z^3}$, we define a random variable 
\begin{equation}
\label{eq:def:fomega}
\omega\in\Omega\mapsto f_0^\omega = \sum_{k\in\Z^3}g_k(\omega)Q_kf_0\,.
\end{equation}
\end{definition}
Let us emphasize that the procedure does not gain regularity in the sense that if $f_0\in H^s\setminus \bigcap_{\delta>s}H^{\delta}$, then $f_0^\omega\in H^s\setminus \bigcap_{\delta>s}H^{\delta}$ for almost every $\omega$. We refer to the works~\cite{luhrmann-mendelson-14,benyi2015-local,benyi2015,benyi2019, benyi-oh-pocovnicu-2019,dodson-luhrmann-mendelson-20,dodson-luhrmann-mendelson-19} and the references therein for further details on the Wiener randomization. 
\subsubsection{Globalization arguments in the stochastic setting}
Besides the construction of local solutions in supercritical regimes, one may ask about the global existence and asymptotic behavior of these solutions. This question is out of reach in general, and the probabilistic method is mainly concerned with perturbations of the zero solution. Still, there are different globalization procedures in such a stochastic context outside the small data regime that make use, in a mild sense, of the conservation laws.
In the compact setting, we can use some invariant, or quasi-invariant measure as a substitute for the conservation of energy. Once again, this was initiated by Bourgain, who settled in~\cite{bourgain94,bourgain96-gibbs} the invariance property of the Gibbs measure for~\eqref{eq:nls} on the torus, in dimension 1 and 2. In~\cite{burq-tzvetkov-2008II}, Burq and Tzvetkov proved almost sure global well-posedness by using the energy increment for $v$ and a Gronwall's argument.

In the Euclidean space, we do not have invariant measures and the globalization arguments are mostly based on deterministic methods. Sill, some works concern the harmonic oscillator. Thanks to the so-called \textit{lens transform} that intertwine the harmonic oscillator and~\eqref{eq:nls}, a randomization procedure emerges from the decomposition on the Hermite functions, that forms an orthonormal basis of eigenfunctions for the harmonic oscillator. On the support of the induced measure, that lies in a space just below $L^2$, we have global existence and scattering under smallness assumptions (see~\cite{poiret2012,prt2014} in dimension $d\geq2$). Recently, Burq and Thomann evidenced in~\cite{burq-thomann2020} that this randomization procedure provides quasi-invariant measures for Schrödinger equations in 1D, and proved almost-sure scattering. However, note that in higher dimensions $d\geq2$, the randomization on the Hermite functions induces a gain of regularity in $L^p$ for $p\neq2$, as well as a gain of decay in weighted $L^2$-spaces. To address energy-critical equations, for which the conservation law is not enough, the natural strategy in the random data setting consists in using a stability theory for~\eqref{eq:nls:f} in the energy space, and to infer a priori space-time bounds, from which we deduce scattering. This method was used by~\cite{dodson-luhrmann-mendelson-20} to prove almost sure scattering for the 4D energy critical wave equation. Note that the case of the wave equation is more favorable, since the energy controls $\partial_tv$. Nevertheless, Killip, Murphy and Visan proved the counterpart of~\cite{dodson-luhrmann-mendelson-20} for the energy critical Schrödinger equation in 4D via a double bootstrap argument, that combines a modified energy and a Morawetz like estimate. In addition, the result of~\cite{killip-murphy-visan-2019} was improved in~\cite{dodson-luhrmann-mendelson-19}, where the authors finely exploited the local smoothing effect. Namely, they refined the probabilistic local well-posedness theory in a new functional framework, inspired by some spaces used to study Schrödinger maps, together with an improved Sobolev embedding for randomized radial initial data. We stress out that the stability-perturbation strategy is more likely to work under some radial assumption on the data, in order to use local energy decay estimate and to gain smoothness when estimating the modified energy increments, with rough data. Specifically, one needs to use a radial Sobolev embedding at some point in the analysis, in order to absorb a weight that comes from the need to use local energy decay, or from the Lin-Strauss estimate. However, Oh, Okamoto and Pocovnicu~\cite{oh-okamoto-19} proved almost sure global existence for the energy critical Schrödinger in higher dimension $d\geq5$, where the potential energy controls more terms. To our knowledge, stability-perturbation arguments for NLS were only used in the energy critical case. In the energy-subcritical case, one may be lead to consider modified energy arguments as well. Colliander and Oh already used the high-low method in a stochastic context, to prove almost-sure global well-posedness for the cubic periodic Schrödinger equation below $L^2(\mathbb{T})$ in ~\cite{colliander-oh-2012},  using the mass as a conservation law. Then, Lührmann and Mendelson~\cite{luhrmann-mendelson-14} applied the high-low method to prove global well-posedness for the NLW. Poiret Robert and Thomann also used this method to prove scattering for supercritical small initial data for the harmonic oscillator in~\cite{prt2014}. More recently, the authors of~\cite{gubinelli-koch-oh2021} used the I-method in the context of the stochastic nonlinear wave equation. In~\cite{fan-mendelson-20}, the I-method was used to address the log-log blow up profile for the mass critical NLS in $L^2(\R^2)$, with a randomization procedure based on wiener cubes.

In the present work, we consider the 3D case, where the equation is energy-subcritical. Hence, there is a gap to bridge between the scaling-critical space $\dot{H}^{1/2}$ and the energy space $\dot{H}^1$. The known results for this equation with supercritical initial data are the followings. On the one hand, we have the almost-sure scattering result for small-data that comes from~\cite{poiret2012} by the use of the lens transform. On the other hand, the works of~\cite{benyi-oh-pocovnicu-2019} based on the wiener randomization address scattering for small data, and conditional global existence for general data.
Let us also mention the work ~\cite{benyi-oh-pocovnicu-2019}, where the authors used higher order expansion in the Duhamel formula. The idea is to refine Da Prato Debussche's trick and include higher order terms in the Ansatz. By doing so, they prove that the remainder lies in $H^1$, and has finite energy. Yet, it is not clear how to deduce global well-posedness from this observation.
\vspace{-5pt}
\subsection{Main result and outline of the proof}
Our approach is a bit different. Since $v$ is not in the energy space but just below, we use the scattering theory in $H^\sigma$ for $1/2<\sigma<1$ detailed in paragraph~\ref{paragraph-intro}, and we establish a stability theory for the perturbed equation~\eqref{eq:nls:f} to obtain a conditional scattering result. Then, we obtain some global a priori estimates and prove the following result.
\vspace{-8pt}
\begin{theorem}
\label{theorem:main}
Let $s,\sigma$ be such that $3/7<s\leq1$, and $6/7<\sigma<2s$. Given a radially-symmetric function $f_0$ in $H^s(\R^3)$, we define the randomized initial data $f_0^\omega$ as in~\eqref{eq:def:fomega}. Then, for almost every $\omega$, there exists a global solution $u$ to the defocusing cubic Schrödinger equation~\eqref{eq:nls},
with
\[u-\e^{it\Delta}f_0^\omega\in  C\parent{\R;H_x^\sigma(\R^3)}\,.\]
Moreover, there exist $u_\pm\in H_x^\sigma(\R^3)$ such that 
\[\underset{t\to\pm\infty}{\lim}\norm{u(t)-\e^{it\Delta}\parent{f_0^\omega+u_\pm}}_{H_x^\sigma(\R^3)}=0\,.\]
\end{theorem}
Uniqueness holds for $v = u-\e^{it\Delta}f_0^\omega\in X^\sigma(\R)\subset C(\R;H_x^\sigma(\R^3))$, solution to the nonlinear Schrödinger equation with some stochastic forcing terms~\eqref{eq:nls:f}, in the critical space $X^\sigma$ defined in~\eqref{def:bourgain-spaces} and with zero initial data. Note that we have scattering at regularity closed $H^{2s}$, which has to be compared to Theorem~\ref{theorem:deterministic-gwp}, where we do not directly have scattering in $H^1$. The threshold $3/7<s$ is certainly not optimal. Even if the local probabilistic flow exists up to $1/4<s$, it would be more realistic to reach $1/3<s$, since the deterministic theory is known for $2/3<\sigma$. One way of loosening the threshold on $s$ would be to have a local-smoothing estimate as~\eqref{eq:local:smoothing1} independent of $N$, and then to perform a scaling argument to have an energy smaller than $1$. In addition, one can also hope to prove a long-time Strichartz estimate in a probabilistic setting, and to refine the local existence theory by the use of lateral spaces as in~\cite{dodson-luhrmann-mendelson-19}.

We stress out that randomized initial data $f_0^\omega$ are not radial.~\footnote{One can also perform an extra microlocal randomization to decouple the orientation of each block that appears in~\eqref{eq:def:fomega}. For instance, one can multiply in the physical space each block by some independent Gaussian vectors on the sphere, and make therefore the initial data as generic as possible.} However, we follow~\cite{dodson-luhrmann-mendelson-20,dodson-luhrmann-mendelson-19} and we use the radial assumption on each frequency localized pieces $Q_kf_0$, for $k\in\Z^3$. Then, we deduce from probabilistic decoupling to deduce an estimate on the linear Schrödinger evolution of the  whole function $f_0^\omega\sim\sum g_n(\omega)Q_kf_0$, which is not radially-symmetric. The radial-symmetry assumption turns out to be crucial in our analysis, since it gives access to the energy decay estimates. It might be loosened, but probably not entirely removed. Indeed, when we estimate the modified energy increment, a lot of derivatives fall on the rough term $f$, and we need at least to gain of $1/2$ derivative when taking the $L^2$-in-time spacetime norms of $f$. For this purpose, we use the local-energy decay. Subsequently, since we avoid using some decay assumptions, we need to use a radial Sobolev embedding to absorb the weights that comes from the local smoothing estimates. This is also the reason Dodson settled the deterministic subcritical theory for radial data only, since the long time Strichartz estimates used in~\cite{dodson19} exploits the local energy decay effect. 
\subsubsection{Outline of the proof}
First, we fix $\alpha>0$, and $N_0=N_0(\alpha,s,\sigma,\norm{f_0}_{H^s})\gg1$, to be chosen at the end of the analysis. From the large deviation bounds presented in section~\ref{section:random}, we can find a set $\widetilde{\Omega}_\alpha$, with $\mathbb{P}(\Omega\setminus\widetilde{\Omega}_\alpha)\leq2\alpha$, such that for every $\omega$, the linear evolution of the random initial data $f_0^\omega$ satisfies refined global spacetime estimates. Specifically, for every $\omega\in\widetilde{\Omega}_\alpha$ and every $N\geq N_0$, we have the finiteness of $F^\omega(\R)$, $F_\infty^\omega(\R)$ and $F_2^\omega(\R)$ defined in~\eqref{eq:ass:data:def}. The large deviation estimate~\eqref{eq:large-deviation-2} for this latter quantity $F_2^\omega$ follows from an improved radial Sobolev embedding estimate, which is a weakened version in dimension 3 of the estimate (1.10) from~\cite{dodson-luhrmann-mendelson-20}. In some sense, it reflects the local energy decay for the linear Schrödinger evolution. Then, we first prove in section~\ref{sec:stability} a \textit{conditional} scattering result, written in Proposition~\ref{prop:uniform:scattering}. It states that the nonlinear Duhamel term $v$ scatters in $H^\sigma$ provided that we have a uniform bound for its $H^s$ norm, on its maximal lifespan. This follows from the stability theory for the perturbed equation~\eqref{eq:nls:f}, at the subcritical energy scales $H^s$ for $2/3<s<1$ where one can use the global spacetime bounds provided by the recent progresses we recalled in Theorem~\ref{theorem:deterministic-gwp}. We avoid the use of~\eqref{eq:large-deviation-2} and of the radial assumption in this section, and we restrict the analysis to some $\omega\in\widetilde{\Omega}_\alpha$, where we only assume~\eqref{eq:ass:data}.

The rest of the paper is devoted to the proof of such an a priori bound for $v$. For this purpose, we fix $6/7<\sigma<1$ and $v\in C(J^*;H^\sigma)$, the maximal lifespan solution of~\eqref{eq:nls:f}. Then, we perform a double bootstrap argument that involves a modified energy and a modified interaction Morawetz inequality, together with large deviation estimates~\eqref{eq:large-deviation-2} on $F_2$, that comes from the local energy decay. In section~\ref{sec:conservation}, we set up the I-method in a stochastic setting. First, we define the modified energy, and compute its increment on some intervals where we assume a smallness condition for the $L_{t,x}^4$ norm of $v$. In addition, we establish a modified interaction Morawetz estimate for the frequency truncated solution $Iu$, where $I:H^\sigma\to H^1$ is the I-operator at regularity $H^\sigma$. Indeed, the usual Morawetz cannot hold for the whole solution, since it requires $H^{1/2}$ regularity. Yet, the general strategy of the globalization argument is to consider $Iu$ instead, solution to the cubic Schrödinger equation with the forcing term $H=I\mathcal{N}(u)-\mathcal{N}(Iu)$. Namely, $H$ is the commutator between the nonlinearity and $I$. By doing so, we can exploit the nonlinear smoothing effect yielded by the frequency cancellations in the commutator $H$. Basically, we can extract a power $N^{-1}$ from spacetime quantities where $H$ appears, and we write a general key estimate in Lemma~\ref{lemma:H} that encapsulate this smoothing effect. Then,  we specify this estimate to handle the modified energy increment and the modified Morawetz interaction. Let us now comment a bit on the modified energy we use. It contains two terms, and writes
\[
\mathcal{E}(v)(t)\coloneqq \frac{1}{2}\int_{\R^3}\abs{\nabla
Iv(t,x)}^2dx+\frac{1}{4}\int_{\R^3}\abs{Iv+If}^4dx\,.
\]
The kinetic energy only depends on $v$, truncated by the I-operator. However, the potential energy depends on $f$, as in~\cite{sun-2015} and ~\cite{oh-okamoto-19}. The motivation for this is to preserve the structure of the unperturbed equation~\eqref{eq:nls}, and to benefit from the key frequency cancellations. Formally, we have
\[
\frac{d}{dt}\mathcal{E}(v)=\re\int_{\R^3} \partial_t(Iv)\parent{\mathcal{N}(Iu)-I\mathcal{N}(u)}dx + \re-i\int_{\R^3} \nabla(If)\mathcal{N}(\nabla Iu, Iu, Iu)dx\,.
\]
To control the first term, we exploit the frequency cancellations in the commutator, which do not appear in the second term. Indeed, this latter term is purely perturbative, and we handle it thanks to local smoothing ~\eqref{eq:ass:data:2}.
Similarly, we write the interaction Morawetz estimate for $Iu$, and not only $Iv$, in order to have some perturbation terms that contain $H$. Note that in the modified energy increment~\eqref{eq:E:loc-incr}, we use the aforementioned idea from Dodson~\cite{dodson-13,dodson19}, and we keep track of some sub-additive quantities on the right-hand side as in, not to lose too many powers of $N$ when summing over the spacetime slabs. Finally, we perform the double bootstrap in section~\ref{sec:conservation} as follows. From the almost conservation of energy, we gain a negative power of $N$, multiplied by some polynomial powers of $N$ on intervals where the $L_{t,x}^4$ norm of $Iv$ is small. However, the interaction Morawetz estimate yields a control on such a spacetime norm, that depends on the energy of $Iv$. Hence, by doing a continuity argument, and after tuning up the parameter $s,\sigma$, we may be able to sum the energy increments over a partition of $J^*$ made of spacetime slabs where the $L_{t,x}^4$ norm of $Iv$ is small, and to prove that the energy cannot exceed $N^{2(1-\sigma)}$ on the lifespan of $v$. At the difference of the standard framework, we cannot use the scaling argument in this perturbed setting. Indeed, some perturbation terms coming from $\norm{\nabla I f^\omega}_{L_t^2L_x^6}$ of size $N^\frac{1-\sigma}{2}$ appear on the right-hand side of the energy increment estimate, so that the energy cannot remain less than $1$ as time evolves. Furthermore, we will use the almost-conserved mass, which is harmless, although it is supercritical with respect to scaling. 
\subsubsection{Organization of the paper}
The analysis is divided into two parts. In the first one, we prove the conditional scattering result for $v$, solution to the perturbed equation~\eqref{eq:nls:f}. In section~\ref{section:random}, we recall standard large deviation estimates, as well as an improved bound for radially-symmetric randomized initial data. Then, we address the stability theory in section~\ref{sec:stability}. In the second part of the analysis, we prove the global a priori estimate for $v$. Namely, we obtain almost conservation laws in section~\ref{sec:conservation}, used in section~\ref{sec:double-bootstrap} to perform the double bootstrap argument. 
\subsection*{Notations}
$\sigma$ is a scaling-subcritical exponent, and $s$ a scaling supercritical one. They are linked together by the relation 
\[
s < 1/2 <\sigma < 2s\,.
\]
We shall keep in mind that the nonlinear Duhamel term is in $H^\sigma$, whereas the initial data is in $H^s$. $I$ is the I-operator defined in~\eqref{eq:def:I}, and $N$ is a large dyadic integer. Throughout the analysis, one can keep in mind that the energy is a priori of order 
\[
\mathcal{E}(v)\lesssim N^{2(1-s)}\,.
\]
We denote the Duhamel integral by 
\begin{equation}
    \label{eq:def-duhamel}
\mathcal{I}(\intervalcc{t_0}{t},F) \coloneqq -i\int_{t_0}^t\e^{i(t-t')\Delta}F(t')dt'
\,.
\end{equation}
The nonlinearity is seen as a trilinear operator, and writes 
\begin{equation}
\label{eq:def-nonlinearity}
    \mathcal{N}(u^{(1)},u^{(2)},u^{(3)}) \coloneqq u^{(1)}\overline{u^{(2)}}u^{(3)},\quad \mathcal{N}(u) \coloneqq \mathcal{N}(u,u,u)\,.
\end{equation}
Given a time-interval $J\subseteq\R$, a randomized initial data $f_0^\omega$ corresponding to some $\omega\in\Omega$, and its linear evolution at time $t$ denoted $f^\omega(t)=\e^{it\Delta}f_0^\omega$, we define
\begin{equation}
\begin{split}
\label{eq:ass:data:def}
F^\omega(J)&\coloneqq\norm{\japbrak{\nabla}^sf^\omega}_{L_{t,x}^{10}(J)}+\norm{\japbrak{\nabla}^sf^\omega}_{L_{t,x}^{4}(J)}+\norm{\japbrak{\nabla}^sf^\omega}_{L_{t,x}^{5}(J)}+\norm{\japbrak{\nabla}^sf^\omega}_{L_t^4L_x^{12}(J)}\,,\\
F_\infty^\omega(J)&\coloneqq\norm{f^\omega}_{L_t^\infty L_x^4(J)}+\norm{f^\omega}_{L_t^\infty L_x^{6}(J)}\,,\\
F_2^\omega(J)&\coloneqq \norm{\japbrak{\nabla}If^\omega}_{L_t^2L_x^\infty(J)}+\norm{\japbrak{\nabla}If^\omega}_{L_t^2L_x^6(J)}\,.
\end{split}
\end{equation}
We will often drop the dependence on $\omega$ from the notations. Some large deviation estimates on this terms are proved in section~\ref{section:random}.~\footnote{We might also incorporate the $L_{t,x}^{10/3}$ norm  in the definition of $F^\omega(J)$, which corresponds to the Strichartz admissible pair $(10/3,10/3)$. Indeed, at some point in the analysis, we will need to consider some spacetime slabs $J\times\R^3$ for which this quantity is small.}
In addition, $C$ is a constant that is irrelevant, and that may change from line to line. We may also write $\lesssim$. Finally, we shall denote by $P_K$ the Littlewood Paley multiplier around the dyadic frequency of size $K=2^k$:
\[
\supp \widehat{P_Kf}\subset\set{\xi\in\R^3\mid 1/2K\leq\abs{\xi}\leq 2K}\,.
\]
The counterpart in the physical space is $\chi_j$:
\[
\supp \chi_jf\subset\set{x\in\R^3\mid 2^{j-1}\leq\abs{x}\leq 2^{j+1}}\,.
\]
The Wiener multiplier around $k\in\Z^3$ is denoted by $Q_k$:
\[
\supp \widehat{Q_kf}\subset k+\intervalcc{-1}{1}^3\,.
\]
\begin{center}
\textbf{Acknowledgements}
\end{center}
The author would like to thank Nicolas Burq, Frédéric Rousset, Chenmin Sun and Nikolay Tzvetkov for enlightening discussions.
\section{Preliminaries and probabilistic estimates}
\label{section:random}
\subsection{The free Schrödinger evolution and critical spaces}
\begin{proposition}[Strichartz estimates in dimension $d=3$]
\label{proposition:strichartz}
Let $u_0\in L_x^2(\R^3)$. We have
\[
\norm{\e^{it\Delta}u_0}_{L_t^q(\R;L_x^r(\R^3))}\lesssim\norm{u_0}_{L_x^2(\R^3)}\,,
\]
for any Strichartz admissible pair 
\[
(q,r)\in\intervalcc{2}{+\infty},\quad  s(q,r)=\frac{2}{q}+d\parent{\frac{1}{r}-\frac{1}{2}}=0\,.
\]
\end{proposition}

\begin{definition}[Functions of bounded 2-variation $V^2$ and atomic space $U^2$]
\label{def-Vp}
Let $J\subseteq\R$ be a time-interval, and $\mathcal{Z}$ be the collection of every finite partitions of $J$. The set of functions with bounded $2$-variation $V^2(J)$ is the set of functions $v:J\to L_x^2(\R^3)$, endowed with the norm
\begin{equation*}
 \norm{v}_{V^2(J)} \coloneqq \underset{\{t_k\}_{k=0}^{K-1}\in\mathcal{Z}}{\sup} \parentbig{\sum_{k=1}^K\norm{v(t_{k})-v(t_{k-1})}_{L_x^2}^2}^\frac{1}{2}.
\end{equation*}
A function $a:J\to L_x^2(\R^3)$ is an atom if there exists a partition $\{t_k\}_{k=0\dots K}$ in $\mathcal{Z}$ and $\{\phi_k\}_{k=0,\dots, K-1}$ some elements in $L_x^2$ such that
\begin{equation*}
a (t) = \sum_{k=1}^K\mathbf{1}_{\intervalco{t_{k-1}}{t_k}}(t)\phi_{k-1},\quad \sum_{k=0}^{K-1} \norm{\phi_k}_{L_x^2}^2 \leq 1.
\end{equation*}
The atomic space $U^2(J)$ is the set of functions $u:J\to L_x^2(\R^3)$ endowed with the norm
\begin{equation}
\label{def-atomic-space} 
\norm{u}_{U^2(J)}\coloneqq \inf\set{\norm{(\lambda_j)}_{\ell^1}\mid u = \sum_{j\geq1}\lambda_j a_j\quad \text{for some $U^2$-atoms $(a_j)$}}.
\end{equation}
\end{definition}
\begin{definition}[Function spaces adapted to the free Schrödinger evolution]
Spaces adapted to the linear propagator $\e^{it\Delta}$ are the Banach spaces endowed with the norms 
\begin{equation*}
\norm{u}_{U_\Delta^2(J)}\coloneqq\norm{\e^{it\Delta}u}_{U^2(J)},\quad \norm{v}_{V_\Delta^2(J)}\coloneqq\norm{\e^{it\Delta}v}_{V^2(J)}.
\end{equation*}
Then, given $\sigma\in\R$, the critical spaces $X^\sigma(J)$ and $Y^\sigma(J)$ are the Banach spaces endowed with the norms
\begin{equation}
\label{def:bourgain-spaces}
\norm{u}_{X^\sigma(J)}^2=\sum_{K\in2^\N}K^{2\sigma}\norm{P_Ku}_{U_\Delta^2(J)}^2,\quad \norm{v}_{Y^\sigma(J)}^2=\sum_{K\in2^\N}K^{2\sigma}\norm{P_Kv}_{V_\Delta^2(J)}^2\,.
\end{equation}
\end{definition}
\begin{proposition}[Embeddings,~\cite{hadac2009} Proposition 2.2 and Corollary 2.6~\footnote{\ The spaces $U_\Delta^2H^\sigma$ and $V_\Delta^2H^\sigma$ are generalizations of the spaces $U_\Delta^2$ and $V_\Delta^2H^\sigma$, with functions $u,v: J\to H^\sigma(\R^3)$ instead of $L_x^2(\R^3)$.}]
\begin{equation}
\label{eq:embedding:U2} 
U_\Delta^2H_x^\sigma\hookrightarrow X^\sigma(J)\hookrightarrow Y^\sigma(J)\hookrightarrow V_\Delta^2H_x^\sigma\hookrightarrow L_t^\infty(J,H_x^\sigma)\,.
\end{equation}
\end{proposition}

\begin{proposition}[Duality (see~\cite{hadac2009} and~\cite{herr2011})]
\label{proposition:dual:U2}
There exists a unique bilinear map $\operatorname{B}:U^2\times V^2\to\C$ such that
\begin{equation*}
v\in V^2 \mapsto \operatorname{B}(\cdot,v)\in(U^2)^*
\end{equation*}
is a surjective isometry, and
\begin{equation*}
\norm{v}_{V^2}=\underset{\norm{u}_{U^2}\leq1}{\sup} \abs{\operatorname{B}(u,v)},\quad \norm{u}_{U^2}=\underset{\norm{v}_{V^2}\leq1}{\sup} \abs{\operatorname{B}(u,v)}.
\end{equation*}
If $\partial_tu\in L_t^1(J;L_x^2(\R^3))$ we have the explicit formula
\begin{equation*}
\operatorname{B}(u,v)=\int_I (\partial_t u\mid v)_{L_x^2(\R^3)}dt.
\end{equation*}
In particular, for the Duhamel term~\eqref{eq:def-duhamel} and when $F\in L_t^1(J;L_x^2(\R^3))$, we have
\begin{equation}
\label{eq:duality:duhamel}
    \norm{\mathcal{I}(J,F)}_{U_\Delta^2}\lesssim\underset{\norm{v}_{V_\Delta^2}\leq1}{\sup}\ \absbig{\iint_{J\times\R}F\overline{v}dxdt}\,.
\end{equation}
\end{proposition}
\begin{proposition}[Transferred linear and bilinear estimates, Lemma 3.3 in~\cite{benyi2015}]
\label{prop:transfer}
Let $(q,r)$ be an admissible pair as in Proposition~\ref{proposition:strichartz}. We have
\begin{equation}
\label{eq:strichartz:embedding}
\norm{u}_{L_t^q(\R ; L_x^r(\R^3))}\lesssim\norm{u}_{U_\Delta^2}\,.
\end{equation}
Let $K,M$ be two dyadic integers. The  bilinear estimate from Bourgain 
\begin{equation}
\label{eq:bili}
\norm{\parent{\e^{it\Delta}P_Ku_0}\parent{ \e^{it\Delta}P_Mu_0}}_{L_{t,x}^2(\R\times\R^3)}\lesssim KM^{-\frac{1}{2}}\norm{P_{\leq K}u_0}_{L_x^2}\norm{P_Mv_0}_{L_x^2}
\end{equation}
has transferred versions to the context of the spaces $U^2$ and $V^2$, that read 
\begin{equation}
\label{eq:bili:transf}
    \begin{split}
\norm{P_KuP_Mv}_{L_{t,x}^2(\R\times\R^3)}&\lesssim KM^{-\frac{1}{2}+}\norm{P_Kv}_{V_\Delta^2}\norm{P_Mv}_{V_\Delta^2}\,,\\ \norm{P_KuP_Mv}_{L_{t,x}^2(\R\times\R^3)}&\lesssim KM^{-\frac{1}{2}}\norm{P_Kv}_{U_\Delta^2}\norm{P_Mv}_{U_\Delta^2}
\,.
\end{split}
\end{equation}
\end{proposition}

\begin{lemma}[Time continuity and sub-additivity, see Lemma A.4 and A.8 in~\cite{benyi2015}]
\label{lemma:U2-continuity}
Let $J=\intervalco{a}{b}$ and $u\in X^\sigma(J)\cap C(J;H_x^\sigma(\R^3))$. The mapping
\[
t\in J\to \norm{u}_{X^\sigma\intervalco{a}{t}}
\]
is continuous. Moreover, given a partition $\disp J=\bigcup_{k=1}^L J_k$, we have
\[
\norm{u}_{X^\sigma(J)}\leq \sum_{k=1}^L\norm{u}_{X^\sigma(J_k)}\,.
\]
\end{lemma}
\subsection{Large deviation estimates for the free Schrödinger evolution}
First, we recall the following probabilistic decoupling estimate. It expresses the gain of integrability that emerge in probabilistic averaging effects, due to the cancelation of interferences.
\begin{lemma}[Lemma 3.1 in~\cite{burq-tzvetkov-2008I}]
\label{lemma:Khinchin}
Let $\{g_n\}$ be a sequence of real valued, zero-mean and independent random variables with distribution $\{\mu_n\}$ on a probability space $\parent{\Omega,\mathcal{A},\mathbb{P}}$. Assume that there exists $c>0$ such that for any $\gamma\in\R$ and any $n\geq0$ we have
 \[\absbig{\int_\R\e^{\gamma x}d\mu_n(x)}\leq \e^{c\gamma^2}\,.\]
Then, there exists $C>0$ such that for every $2\leq p\leq \infty$ and for every sequence of complex numbers $\{c_n\}$ in $\ell^2(\N;\C)$ it holds that
\begin{equation}
    \label{eq:Khinchin}
    \normbig{\sum_{n\in\N} c_ng_n(\omega)}_{L_\omega^p(\Omega)}\leq C\sqrt{p}\parentbig{\sum_{n\in\N}\abs{c_n}^2}^{1/2}\,.
\end{equation}
\end{lemma}
From this, we obtain the general large deviation estimate, on which rely all the estimates stated in this section. 
\begin{lemma}[Large deviation estimate]
\label{lemma:large:deviation}
Let $F$ be a real valued measurable function on a probability space $\parent{\Omega,\mathcal{A},\mathbb{P}}$. Assume that there exists some constants $C_0>0$, $M>0$ and $p_0\geq1$ such that for all $p\geq p_0$ we have
\begin{equation}
\label{eq:moment}
\norm{F}_{L_\omega^p(\Omega)}\leq C_0\sqrt{p}M\,.
\end{equation}
Then, there exists $C=C(C_0)$ and $c=c(C_0)>0$ independent of $M$, such that for all $\lambda>0$,
\begin{equation}
    \label{eq:large-deviation}
    \mathbb{P}\parent{\set{\omega\in\Omega\mid \abs{F(\omega)}>\lambda}}\leq C\exp\parent{-c\lambda^2M^{-2}}\,.
\end{equation}
\end{lemma}
In the following lemma, we recall the ~\textit{probabilistic} Strichartz estimates, that provide improved global spacetime bounds by loosening the Strichartz admissibility condition.
\begin{lemma}[Improved Strichartz estimates]
\label{lemma:rand-zero}
Let $(q,r)$ be a Strichartz admissible pair, i.e $s(q,r)=0$. For all $\widetilde{r}\geq r$, there exists $C=C(q,\widetilde{r})>0$ and $c=c(q,r)>0$ such that for all $f_0\in H^s(\R^3)$ and $\lambda>0$, we have
\begin{equation}
\mathbb{P}\parent{\set{\omega\in\Omega\mid\norm{\japbrak{\nabla}^s\e^{it\Delta}f_0^\omega}_{L_t^q(\R;L_x^{\widetilde{r}}(\R^3)}>\lambda}}\leq C\exp\parent{-c\lambda^2\norm{f_0}_{H^s}^{-2}}\,.
\end{equation}
\end{lemma}
\begin{proof}
The proof is standard, and we refer to~\cite{benyi2015,benyi2015-local}. To prove~\eqref{eq:moment} for $p>\max(q,\widetilde{r})$, we apply Minkowski's inequality and then the decoupling estimate~\eqref{eq:Khinchin}. Then, we apply the unit-scale Bernstein estimate on each frequency bloc, that yields a bound on the Fourier multiplier $Q_k$ from $L_x^r$ to $L_x^{\widetilde{r}}$, uniformly in $k$. We conclude by applying the Strichartz estimate with $(q,\widetilde{r})$.
\end{proof}
Next, we state a uniform in time large deviation estimate for the linear evolution of randomized initial data.  
\begin{lemma}
\label{lemma:large-deviation-infty}
Let $f_0$ in $H^s$ for some $s>0$. For all $q>2$, there exists $C,c>0$ such that for all $\lambda>0$,
\[
\mathbb{P}\parent{\set{\omega\in\Omega\mid \norm{\e^{it\Delta}f_0^\omega}_{L_t^\infty(\R;L_x^q(\R^3))}>\lambda}}\leq C\exp\parent{-c\lambda^2\norm{f_0}_{H_x^s}^{-2}}\,.
\]
\end{lemma}

\begin{proof}
See Lemma 5.15 in~\cite{dodson-luhrmann-mendelson-19}. This follows from a Sobolev embedding in time, and from the improved Strichartz estimate stated in the previous Lemma.
\end{proof}
We deduce from these large deviation bounds, together with lemma~\ref{lemma:large:deviation}, the following spacetime estimates on $f^\omega$, for some $\omega$ in a large measure set. 
\begin{lemma}
\label{lemma:rand-infty}
For every $\alpha>0$ there exists a set $\Omega_\alpha$ and a constant $C_\alpha$ depending on $\norm{f_0}_{H^s}$ such that $\mathbb{P}(\Omega\setminus\Omega_\alpha)\leq \alpha$, and for every $\omega\in\Omega_\alpha$ we have
\begin{equation}
\label{eq:ass:data}
F^\omega(\R)+F_\infty^\omega(\R)\leq C_\alpha\,.
\end{equation}
\end{lemma}
In what follows, $\omega\in\Omega_\alpha$ is fixed such that~\eqref{eq:ass:data} holds. At this stage, we do not need the radial assumption on $f_0$. 
\subsection{Additional estimate for randomized radial initial data}
The following improved Sobolev embedding for radial functions was introduced in~\cite{dodson-luhrmann-mendelson-20}, and used in~\cite{dodson-luhrmann-mendelson-19} in order to prove almost-sure scattering for the energy critical Schrödinger and wave equation in $\R^4$. Here, we present an analog of this lemma in dimension 3, whose proof needs some slight modifications.
\begin{lemma}[Improved Sobolev embedding for radial functions in $\R^3$]
For all $\delta>0$, there exists $C_\delta>0$ such that for all radially-symmetric function $f_0$ in $H^\delta(\R^3)$ we have 
\label{lemma:radialish:embedding} 
\begin{equation}
\label{eq:radialish}
\normbig{\japbrak{x}\parentbig{\sum_{k\in\Z^3}\abs{Q_kf_0}^2}^\frac{1}{2}}_{L_x^\infty(\R^3)}\leq C_\delta \norm{f_0}_{H^\delta(\R^3)}\,.
\end{equation}

\end{lemma}
\begin{proof}
Assume that $x=\abs{x}e_3$, where $e_3=(0,0,1)$ and $\abs{x}\gg1$. We denote by $(\rho,\theta,\varphi)$ the spherical coordinates in the Fourier space $\R_\xi^3$:
\[\xi = \xi(\rho,\theta,\varphi) =  \parent{\rho\cos(\varphi)\sin(\theta),\rho\sin(\varphi)\sin(\theta),\rho\sin(\varphi)\cos(\theta)}\,. \]
with $\rho>0$, $\theta\in\intervaloo{0}{\pi}$ and $\varphi\in\intervalcc{0}{2\pi}$. Given some $k\in\Z^3$ such that $1\ll\abs{k}$, we write
\begin{equation}
\label{eq:Q_k:spherical}
    \begin{split}
    Q_kf(x)&=\int_{\R^3}\e^{i\xi\cdot x}\psi_k(\xi)\widehat{f}(\xi)d\xi\\
    &=\int_0^{+\infty}\int_0^{2\pi}\parentbig{\int_0^\pi \e^{i\abs{x}\rho\cos{\theta}}\psi_k(\xi(\rho,\theta,\varphi))\sin{\theta}d\theta} d\varphi\rho^2\widehat{f_0}(\rho)d\rho\,.
    \end{split}
\end{equation}
Hence, we need to study the following oscillatory integral
\begin{equation}
\label{eq:oscillatory}
    I(x,\rho,\varphi) = \int_0^\pi\e^{i\abs{x}\rho\cos{\theta}}\psi_k(\xi(\rho,\theta,\varphi))\sin{\theta}d\theta\,.
\end{equation}
The function $\psi_k$ and its derivatives provide a localization on the unit cube centered around $k$. More precisely, this localization reads in spherical coordinates as follows:   if $\xi(\rho,\theta,\varphi)$ is in $\supp(\psi_k)$, then 
\begin{itemize}
    \item $\rho$ is in an interval of size $\sim1$ around $\abs{k}$.
    \item $\theta$ in an interval of size $\sim\frac{1}{\abs{k}}$ around the angle $\theta_k$, where $\sin(\theta_k)\sim\sqrt{1-\parent{\frac{k}{\abs{k}}e_3}^2}$.
    \item $\varphi$ is in an interval $J_k$ of size $\sim\min(1,\frac{1}{\abs{k}\sin{(\theta_k)}})$.
\end{itemize}
To gain some decay with respect to $x$, we integrate by parts in the variable $\theta$. It holds
\begin{multline}
\label{eq:I:xrhovarphi}
        I(x,\rho,\varphi) = -\frac{1}{i\abs{x}\rho}\int_0^\pi\partial_\theta\parent{\e^{i\rho\cos\theta\abs{x}}}\psi_k(\xi(\rho,\theta,\varphi))d\theta\\
        =-\frac{1}{i\abs{x}\rho}\left[\e^{i\rho\cos\theta\abs{x}}\psi_k(\xi(\rho,\theta,\varphi))\right]_0^\pi+\frac{1}{i\abs{x}\rho}\int_0^\pi \e^{i\rho\cos\theta\abs{x}}\partial_\theta\varphi_k(\xi(\rho,\theta,\varphi))d\theta\,.
\end{multline}
Moreover, we have for all $\xi(\rho,\theta,\varphi)\in\R_\xi^3$ and all $x\in\R^3$ and $k\in\Z^3$ that
\begin{equation}
\label{eq:dtheta}
    \abs{\frac{\partial}{\partial_\theta} \psi_k(\xi)}\leq \rho\abs{\nabla_\xi\psi_k(\xi)}\,.
\end{equation}
Keeping in mind the localization in the variables $\rho,\theta$ yielded by the amplitude $\psi_k$ and its derivatives, we deduce from~~\eqref{eq:I:xrhovarphi} and~\eqref{eq:dtheta} that
\begin{equation}
\label{eq:I:xrhovarphi2}  
\abs{I(x,\rho,\varphi)}\lesssim\frac{1}{\abs{x}\rho}\mathbf{1}_{\intervaloo{\abs{k}-1}{\abs{k}+1}}(\rho)\mathbf{1}_{J_k}(\varphi)\,.
\end{equation}
Using that $\varphi$ is localized on the interval $J_k$ of size $\min(1,\frac{1}{\abs{k}\sin{\theta_k}})$, integrating over the domain$ \intervaloo{\abs{k}-1}{\abs{k}+1}\times J_k$ and using Cauchy-Schwarz in $\rho$, we get
\begin{multline*}
    \abs{Q_k(f_0)(x)}\lesssim\frac{1}{\abs{k}\abs{x}}\min(1,\frac{1}{\abs{k}\sin\theta_k})\norm{\mathbf{1}_{\intervaloo{\abs{k}-1}{\abs{k}+1}}(\rho)\widehat{f_0}(\rho)\rho^2}_{L_\rho^1(\R)}
    \\
    \lesssim\min(1,\frac{1}{\abs{k}\sin\theta_k})\norm{\mathbf{1}_{\intervaloo{\abs{k}-1}{\abs{k}+1}}(\rho)\widehat{f_0}(\rho)\rho}_{L_\rho^2(\R)}
\end{multline*}
Summing the above estimate over $\Z^3$ gives
\begin{equation*}
    \begin{split}
    \abs{x}^2\sum_{k\in\Z^3}\abs{Q_k(f_0)(x)}^2&\lesssim\sum_{k\in\Z^3}\min(1,\frac{1}{\abs{k}\sin\theta_k})^2\norm{\mathbf{1}_{\intervaloo{\abs{k}-1}{\abs{k}+1}}(\rho)\widehat{f_0}(\rho)\rho}_{L_\rho^2(\R)}^2\\
    &\lesssim\sum_{j\in\N}\sum_{l=0}^j\frac{\#\set{k\in\Z^3\mid\abs{k}=j,\ \abs{j\sin(\theta_k)-l}\leq1}}{(1+l)^2}\norm{\mathbf{1}_{\intervaloo{j-1}{j+1}}(\rho)\widehat{f_0}(\rho)\rho}_{L_\rho^2(\R)}^2\\
    &\lesssim \sum_{j\in\N}\sum_{l=0 }^j\frac{1}{1+l}\norm{\mathbf{1}_{\intervaloo{j-1}{j+1}}(\rho)\widehat{f_0}(\rho)\rho}_{L_\rho^2(\R)}^2\\
    &\lesssim \sum_{j\geq1}\log(1+j)\norm{\mathbf{1}_{\intervaloo{j-1}{j+1}}(\rho)\widehat{f_0}(\rho)\rho}_{L_\rho^2(\R)}^2\lesssim_\delta \norm{f_0}_{H^\delta(\R^3)}^2\,.
\end{split}
\end{equation*}
This ends the proof of Lemma~\ref{lemma:radialish:embedding}.
\end{proof}
\begin{corollary}
Let $2\leq r\leq\infty$, $K\in2^\N$ and $\delta>0$. There exists a constant $C_\delta>0$ such that for any radial function $f_0\in L^2(\R^3)$ we have 
\begin{equation}
    \label{eq:radialish:r}
    \normbig{\japbrak{x}^{1-\frac{2}{r}}\parentbig{\sum_{k\in\Z^3}\abs{Q_kP_Kf_0}^2}^\frac{1}{2}}_{L_x^r(\R^3)}\leq C_\delta K^\delta\norm{P_Kf_0}_{L_x^2(\R^3)}\,.
\end{equation}
\end{corollary}
\begin{proof}
The estimate~\eqref{eq:radialish:r} is obtained by interpolation between~\eqref{eq:radialish} and the trivial estimate
\[\norm{\parentbig{\sum_{k\in\Z^3}\abs{Q_kP_Kf_0}^2}^\frac{1}{2}}_{L_x^2(\R^3)}\leq \norm{P_Kf_0}_{L_x^2(\R^3)}\,. \]
It holds
\begin{equation*}
    \begin{split}
        \normbig{\japbrak{x}^{1-\frac{2}{r}}\parentbig{\sum_{k\in\Z^3}\abs{Q_kP_Kf_0}^2}^\frac{1}{2}}_{L_x^r}
        &\leq \normbig{\japbrak{x}\parentbig{\sum_{k\in\Z^3}\abs{Q_kP_Kf_0}^2}^\frac{1}{2}}_{L_x^\infty}^{1-\frac{2}{r}}\normbig{\parentbig{\sum_{k\in\Z^3}\abs{Q_kP_Kf_0}^2}^\frac{1}{2}}_{L_x^2}^\frac{2}{r}\\
        &\leq C_\delta\norm{P_Kf_0}_{H^\delta}^{1-\frac{2}{r}}\norm{P_Kf_0}_{L_x^2}^\frac{2}{r}\\
        &\lesssim C_\delta K^\delta\norm{P_Kf_0}_{L_x^2}\,.
    \end{split}
\end{equation*}
\end{proof}
We will use the local smoothing estimate under the following form.
\begin{lemma}[Local smoothing,~\cite{local-smoothing}]
\label{lemma:local:smooting}
For any $d\geq3$ and $\alpha>0$, there exists a constant $C$ such that for all $f_0\in L^2(\R^d)$,
\begin{equation}
\label{eq:local:smoothing1}
\underset{R>0}{\sup}\norm{\e^{it\Delta}f_0}_{L_{t,x}^2(\R\times\{\vert x\vert<R\})}\leq CR^{1/2} \norm{\abs{\nabla}^{-1/2}f_0}_{L_x^2(\R^d)}\,.
\end{equation}
\end{lemma}
Recall that the Fourier multiplier $I:H_x^\sigma\to H_x^1$ was defined in~\eqref{eq:def:I}
\begin{proposition}
\label{prop:ass:data:radial}
Let $s>1/4$ and $\sigma>1/2$ such that $\sigma<2s$. Then, for all $r\in\intervaloc{4}{+\infty}$ and $\delta>0$, there exist $N_0$, $C>0$ and $c>0$ depending on $r,\delta,s,\sigma$ such that for all $\lambda>0$ and $N\geq N_0$, we have the large deviation estimates

\begin{equation}
    \label{eq:large-deviation-2}
\mathbb{P}\parent{\set{\omega\in\Omega\mid \norm{\nabla I\e^{it\Delta}f_0^\omega}_{L_t^2(\R;L_x^r(\R^3))}>\lambda}}\leq C\exp\parent{-c\lambda^2( N^{-\delta}N^\frac{1-\sigma}{2}\norm{f_0}_{H_x^s})^{-2}}\,.
\end{equation}
\end{proposition}
\begin{proof}
It suffices to prove that for any $\delta>0$ small enough (depending on $s$ and $\sigma$), there exists a constant $C_{\delta}$ such that for any $p>r$ one has
\begin{equation}
    \label{eq:large-deviation-F2}
    \norm{\nabla I\e^{it\Delta}f_0^\omega}_{L_\omega^pL_t^2L_x^r}\leq C_{\delta}\sqrt{p}N^{-\delta}N^\frac{1-\sigma}{2}\norm{f_0}_{H_x^s(\R^3)}\,.
\end{equation}
Subsequently,~\eqref{eq:large-deviation-2} follows from Lemma~\ref{lemma:large:deviation}, with $F \coloneqq \norm{\nabla I\e^{it\Delta}f_0^\omega}_{L_t^2(\R;L_x^r(\R^3))}$ and $M=N^{0-}N^\frac{1-\sigma}{2}\norm{f_0}_{H_x^s}$. To prove~\eqref{eq:large-deviation-F2}, we dyadically decompose $f_0$ in the frequency space. We reduce the case when $r=\infty$ to the case where $r$ is finite by applying the Bernstein estimate. For any $3<r<\infty$, we get
\begin{equation}
\label{eq:large-deviation-bernstein}
    \norm{\nabla\e^{it\Delta}If_0^\omega}_{L_\omega^pL_t^2L_x^r}\lesssim\norm{\nabla P_{\leq1}\e^{it\Delta}f_0^\omega}_{L_\omega^pL_t^2L_x^r}+\sum_{K\in2^{\N^*}}K^\frac{3}{r}\norm{\nabla P_KI\e^{it\Delta}f_0^\omega}_{L_\omega^pL_t^2L_x^r}\,.
\end{equation}
Estimating the low-frequency term is standard, and we refer to~\cite{benyi2015-local}. However, to handle the high-frequency terms, we used the radial Sobolev embedding~\eqref{eq:radialish:r} and the local smoothing estimate. Let us fix $K\geq 2$ a dyadic integer. Then, we dyadically decompose the term on right-hand side of~\eqref{eq:large-deviation-bernstein} in the physical space, with some cutoff $\chi_j$ introduced in the notations paragraph. By the triangle inequality, we have
\[
\norm{\nabla P_KI\e^{it\Delta}f_0^\omega}_{L_\omega^pL_t^2L_x^r}\leq \sum_{j\geq0}\norm{\chi_j\nabla P_KI\e^{it\Delta}f_0^\omega}_{L_\omega^pL_t^2L_x^r}.
\]
Next, for each $j\geq0$, and $p\geq r$, we use Minkowski's inequality and the decoupling estimate~\eqref{eq:Khinchin} to get 
\[
    \norm{\chi_j\nabla P_KI\e^{it\Delta}f_0^\omega}_{L_\omega^pL_t^2L_x^r}\lesssim\sqrt{p}\normbig{\parentbig{\sum_{\substack{k\in\Z^3\\\vert k\vert\sim K}}\abs{\chi_jQ_k\nabla P_KI\e^{it\Delta}f_0}^2}^\frac{1}{2}}_{L_t^2L_x^r}\,.
\]
We split the sum into a localized term, and a remainder.
\begin{equation*}
\begin{split}
\norm{\chi_j\nabla P_KI\e^{it\Delta}f_0^\omega}_{L_\omega^pL_t^2L_x^r}\lesssim&\sqrt{p} \normbig{\parentbig{\sum_{\substack{k\in\Z^3\\\vert k\vert\sim K}}\abs{\chi_jQ_k\chi_{\leq j+5}\nabla P_KI\e^{it\Delta}f_0}^2}^\frac{1}{2}}_{L_t^2L_x^r}\\
+&\sqrt{p}\normbig{\parentbig{\sum_{\substack{k\in\Z^3\\\vert k\vert\sim K}}\abs{\chi_jQ_k\chi_{>j+5}\nabla P_KI\e^{it\Delta}f_0}^2}^\frac{1}{2}}_{L_t^2L_x^r}\,.
\end{split}
\end{equation*}
To estimate the remainder, we refer to~\cite{dodson-luhrmann-mendelson-19}. The proof is a bit technical but only uses Young inequality the local smoothing estimate~\eqref{eq:local:smoothing1} and some operator bounds collected in the following lemma that holds for dimension 3 without changing the proof.  
\begin{lemma}[Lemma 5.10 and Lemma 5.11 from~\cite{dodson-luhrmann-mendelson-19}]
\label{lemma:commute}
Let $2\leq r\leq \infty$. For any $k\in\Z^3$, any dyadic $J,L\in2 ^\N$ with $2^5J<L$, and any $\nu>0$, there exists $C_\nu>0$ such that
\begin{equation}
    \label{eq:commute:PQP}
\norm{\chi_jQ_k\chi_l}_{L_x^2(\R^3)\to L_x^r(\R^3)}\leq C_\nu L^{-\nu}\,.
\end{equation}
Similarly, given $L\in2^\N$ and $k,m\in\Z^3$ such that $\abs{k-m}\geq100$, it holds that
\begin{equation}
    \label{eq:commute:QPQ}
    \norm{Q_k\chi_lQ_m}_{L_x^2(\R^3)\to L_x^r(\R^3)}\leq C_\nu L^{-\alpha}\abs{k-m}^{-\nu}\,.
\end{equation}
\end{lemma}
However, we detail how to handle the main term, since this requires to combine the local smoothing estimate~\eqref{eq:local:smoothing1} with the improved radial Sobolev embedding~\eqref{eq:radialish}.~\footnote{\ The case when $j=0$ is actually easier and does not require to use the improved radial Sobolev embedding. We refer to the original proof in~\cite{dodson-luhrmann-mendelson-20} where this term is handled separately.} Given $j\geq0$, we have
\begin{multline*}
\normbig{\parentbig{\sum_{\substack{k\in\Z^3\\\vert k\vert\sim K}}\abs{\chi_jQ_k\chi_{\leq j+5}\nabla P_KI\e^{it\Delta}f_0}^2}^\frac{1}{2}}_{L_t^2L_x^r}\\
\lesssim2^{-j(1-\frac{2}{r})}\normbig{\japbrak{x}^{1-\frac{2}{r}}\parentbig{\sum_{\substack{k\in\Z^3\\\vert k\vert\sim K}}\abs{Q_k\chi_{\leq j+5}\nabla P_KI\e^{it\Delta}f_0}^2}^\frac{1}{2}}_{L_t^2L_x^r}\,.
\end{multline*}
By the improved radial Sobolev embedding~\eqref{eq:radialish:r}, we have
\[
\normbig{\japbrak{x}^{1-\frac{2}{r}}\parentbig{\sum_{\substack{k\in\Z^3\\\vert k\vert\sim K}}\abs{Q_k\chi_{\leq j+5}\nabla P_KI\e^{it\Delta}f_0}^2}^\frac{1}{2}}_{L_t^2L_x^r}
        \lesssim_\delta K^\delta\norm{\chi_{\leq j+5}\abs{\nabla}P_KI\e^{it\Delta}f_0}_{L_t^2L_x^2}\,.
\]
Next, we apply the local smoothing estimate~\eqref{eq:local:smoothing1} to see that the above line is controlled by
\[
\normbig{\parentbig{\sum_{\substack{k\in\Z^3\\\vert k\vert\sim K}}\abs{\chi_jQ_k\chi_{\leq j+5}\nabla P_KI\e^{it\Delta}f_0}^2}^\frac{1}{2}}_{L_t^2L_x^r}\lesssim_\delta 2^{-j(1-\frac{2}{r}-\frac{1}{2})}K^{\frac{1}{2}+\delta}\norm{P_KIf_0}_{L_x^2}\,.
\]
The above expression can be summed over $j$ provided that $r>4$. In addition, we observe that if $K>2N$, there exists $\gamma>0$ such that
\[K^{\frac{1}{2}+\delta}\norm{P_KIf_0}_{L_x^2}\leq N^{1-\sigma}K^{\sigma-\frac{1}{2}-s+\delta}\norm{P_Kf_0}_{H_x^s}\,.\] 
Indeed, Plancherel estimate and the definition of the $I$-multiplier~\eqref{eq:def:I} yield 
\begin{multline*}
   K^{1+2\delta}\norm{P_KIf_0}_{L_x^2}^2=K^{1+2\delta}\int_{2^{-1}K<\abs{\xi}<2K}\parent{\frac{N}{\vert\xi\vert}}^{2(1-\sigma)}\abs{\widehat{P_Kf}_0(\xi)}^2d\xi\\
   \lesssim N^{2(1-\sigma)}K^{2(\sigma-1)+1-2s+2\delta}\norm{P_Kf_0}_{H_x^s}^2\,,
\end{multline*}
When $K\leq 2N$, we directly get that
\[\norm{\abs{\nabla}^{\frac{1}{2}+\delta}P_KIf_0}_{L_x^2}\leq N^{\frac{1}{2}-s+\delta}\norm{P_Kf_0}_{H_x^s}\,.\]
We now sum over $K$ in estimate~\eqref{eq:large-deviation-bernstein} and get that for all $r>4$ and $\delta>0$,~\footnote{In the case when we look at the $L_t^2L_x^\infty$ norm, we need to add a $K$ to the power $\frac{3}{r}$, for $r$ arbitrarily large, that comes from the Sobolev embedding from $W^{\frac{3}{r},r}(\R^3)$ to $L^\infty(\R^3)$. This term is harmless, and can be absorbed in the residual power $K^\delta$.}
\begin{multline*}
\norm{\nabla\e^{it\Delta}IF}_{L_\omega^pL_t^2L_x^\infty}\lesssim_{\delta,r}\sqrt{p}\norm{P_\leq1f_0}_{L_x^2}+\sqrt{p}\sum_{\substack{K\in2^{\N^*}\\K\leq 2N}}K^{\frac{1}{2}-s+\delta}\norm{P_Kf_0}_{H_x^s}\\
+\sqrt{p}N^{1-\sigma}\sum_{\substack{K\in2^{\N^*}\\K>2N }}K^{\sigma-\frac{1}{2}-s+\delta}\norm{P_Kf_0}_{H_x^s}\,. 
\end{multline*}
Define $\gamma_0(s,\sigma)=s-\frac{\sigma}{2}>0$ such that $\frac{1}{2}-s=\frac{1-\sigma}{2}-\gamma_0\,$.~\footnote{\ Recall the assumption that $s\leq 1/2$ and $\sigma<2s$.} Next, we chose $\delta$ and $\gamma_0$ such that
\begin{equation}
\label{eq:ass-r-delta}
\sigma-\frac{1}{2}-s+\delta<0,\quad \text{and}\ \frac{1}{2}-s\leq \frac{1-\sigma}{2}-\gamma_0\,.
\end{equation}
By using Cauchy Schwarz, we conclude that 
\[
\norm{\nabla\e^{it\Delta}I\e{it\Delta}f_0^\omega}_{L_\omega^pL_t^2L_x^\infty}\lesssim_{\delta,r}\sqrt{p} N^{\frac{1}{2}-s+\delta}\norm{P_Kf_0}_{H_x^s}\,.
\]
This ends the proof of Proposition~\ref{prop:ass:data:radial}.
\end{proof}
We deduce from the above large deviation estimates the following bounds on the linear evolution of the randomized initial data.
\begin{proposition}
For all $\alpha>0$, there exists a set $\widetilde{\Omega}_\alpha\subset\Omega_\alpha$, with $\mathbb{P}(\Omega_\alpha\setminus\widetilde{\Omega}_\alpha)\leq\alpha$, a constant $C_\alpha>0$ and $N_0(\alpha)$ such that for all $\omega\in\Omega_\alpha$, and all $N\geq N_0$ we have
\begin{equation}
\label{eq:ass:data:2}
F_2^\omega(\R)\leq C_\alpha N^\frac{1-\sigma}{2}\,. 
\end{equation}
\end{proposition}
\begin{proof}
Estimate~\eqref{eq:ass:data:2} follows from the large deviation estimate~\eqref{eq:large-deviation-2} by choosing 
\[\lambda=C_\alpha N^{0-}N^\frac{1-\sigma}{2}\norm{f_0}_{H_x^s}\,,\quad C_\alpha^2\geq c^{-1}\ln(C\alpha^{-1})\,.\qedhere\]
\end{proof}
\section{Cauchy theory for the forced NLS equation}
The main step of this section is to come with a conditional scattering result for the solutions $v$ to the perturbed Schrödinger equation with a random forcing term~\eqref{eq:nls:f}. More precisely, we prove that an a priori uniform estimate of the $H^\sigma$ norm of $v$ on its maximal lifespan yields global existence and scattering, when $2/3<\sigma\leq1$. To prove such a result, we essentially follow the same lines as in the proof of Proposition 3.1 from~\cite{killip-murphy-visan-2019}. However, since $v$ lies below the energy space, we need to develop a stability theory in $H^\sigma$. For this, given a solution $v$ to~\eqref{eq:nls}, a time $t_0$ and a solution $u$ to~\eqref{eq:nls} with $u(t_0)=v(t_0)$,  we appeal to the global Cauchy theory in such subcritical regimes provided by~\cite{ckstt-2004,dodson-13}, that claims that $u$ enjoys the spacetime global estimate 
\[
\norm{u}_{L_{t,x}^5}\leq C(\norm{u_0}_{H^\sigma})\,.
\]
With this global spacetime bound at hand, we can use sub-additivity and reduce the analysis to intervals $J$ where we have a smallness assumption on the $L_{t,x}^5(J)$ norm of $u$, which is scaling-critical, and hence on $v$ which is expected to stay close to $u$. For this reason, we are led to refine the trilinear stochastic estimates from~\cite{benyi-oh-pocovnicu-2019}, and to estimate the Duhamel nonlinear term not only by the critical norm $X^\sigma(J)$, which is not small, but also by some powers of the norm  $L_{t,x}^5(J)$ of $v$ to gain smallness. This is the matter of the nonlinear estimates from Propositions~\ref{prop:trilinear} and~\ref{prop:trilinear:5}. Throughout this section, we fix $\alpha>0$ and $\omega\in\Omega_\alpha$ such that 
\[
F^\omega(\R)\leq C_\alpha\,.
\]
In particular, we do not use the quantity $F_2^\omega(\R)$, so that we can release the radial assumption on $f_0$.
\subsection{Nonlinear estimates}
Before diving into the nonlinear analysis, let us detail how to address the Littlewood-Paley summation of $L^q$-norms of dyadic blocs $u_N$, for a given function $u$ that comes with a gain of regularity materialized by the presence of a negative power of $N$ in front of each $u_N$.   
\begin{observation}
\label{observation:sum}
Let $q>2$ and $u$ in $L_t^q(J;W_x^{\gamma,q}(\R^3))$ for some $\gamma\in\R$. We denote by $u\sim \sum_{N\in2^\N}u_N$ the Littlewood-Paley decomposition of $u$. We have
\begin{equation}
\label{eq:observation-q}
\sum_{N\in2^\N}N^{0-}N^\gamma\norm{u_N}_{L_{t,x}^q(J\times\R^3)}\lesssim \norm{u}_{L_t^q(J;W_x^{\gamma,q}(\R^3))}.
\end{equation}
\end{observation}
\begin{proof}
From Hölder, we have
\begin{multline*}
       \sum_{N\in2^\N}N^{0-}N^{\gamma}\norm{u_N}_{L_{t,x}^q(J\times\R^3)}\leq \parent{\sum_{N\in2^\N}N^{0-}}^\frac{1}{q'}\parent{\sum_{N\in2^\N}\parent{N^\gamma\norm{u_N}_{L_{t,x}^q(J\times\R^3)}}^q}^\frac{1}{q}
  \\ \lesssim \norm{N^\gamma u_N}_{\ell_N^q(2^\N;L_{t,x}^q(J\times\R^3))}=\norm{N^\gamma u_N}_{L_{t,x}^q(J\times\R^3;\ell_N^q(2^\N))}\,.
\end{multline*}
Since $q\geq2$, we obtain
\[
\norm{N^\gamma u_N}_{L_{t,x}^q(J\times\R^3;\ell_N^q(2^\N))}\leq \norm{N^\gamma u_N}_{L_{t,x}^q(J\times\R^3;\ell_N^2(2^\N))}\,.
\]
We conclude the proof of~\eqref{eq:observation-q} by applying the Littlewood-Paley square function theorem
\[
\norm{u_N}_{L_{t,x}^q(J\times\R^3;\ell_N^2(2^\N))}=\normbig{\parentbig{\sum_{N\in2^\N}\abs{u_N}^2}^\frac{1}{2}}_{L_{t,x}^q}\lesssim \norm{ u}_{L_{t,x}^q(J\times\R^3)}\,.
\qedhere\]
\end{proof}
Now, we establish quadrilinear estimates that involve three types of terms. First, we have the stochastic forcing terms of type $f=\e^{it\Delta}f_0^\omega$, for which we proved improved Strichartz estimates. More precisely, we have that $F^\omega(\R)<+\infty$, and using sub-additivity, we will reduce the analysis to a finite number of intervals where $F^(J)$ is small. Then, we have the term $v$ solution to the forced Schrödinger equation~\eqref{eq:nls:f} at the subcritical regularity $H^{\sigma}$, and that corresponds to the nonlinear Duhamel term for the solution $u=\e^{it\Delta}f_0^\omega+v$. We want to obtain a priori estimates for $v$ in the spacetime spaces $X^\sigma$ and $L_{t,x}^5$, at least, on some intervals where there holds a smallness assumption on the forcing term. Finally, the terms $w\in Y^0$ that appear in the analysis come from duality.
\begin{proposition}[Trilinear estimates with random terms]
\label{prop:trilinear}
Let $1/4<s\leq 1/2$ and $1/2<\sigma<2s$. Denote by $f\omega=\e^{it\Delta}f_0^\omega$ the linear evolution of the randomized initial data. There exists a constant $C(\norm{f_0}_{H^\sigma})>0$ such that for all interval $J\subseteq\R$ where $F^\omega(J)\leq1$, and all $v\in X^\sigma(J)$ with $\norm{v}_{L_{t,x}^5(J)}\leq1$, it holds
\begin{align}
    \label{eq:tri-fff}
    \norm{\mathcal{I}\parent{\cdot,\mathcal{N}(f,f,f)}}_{X^\sigma(J)}&\leq C(\norm{f_0}_{H^\sigma})F^\omega(J)\,,\\
    \label{eq:tri-ffv}
    \norm{\mathcal{I}\parent{\cdot,\mathcal{N}(f,f,v)}}_{X^\sigma(J)}&\leq C(\norm{f_0}_{H^\sigma})F^\omega(J)\norm{v}_{X^\sigma(J)}\,,\\
    \label{eq:tri-fvv}
    \norm{\mathcal{I}\parent{\cdot,\mathcal{N}(f,v,v)}}_{X^\sigma(J)}&\leq C(\norm{f_0}_{H^\sigma})\parent{F^\omega(J)+\norm{v}_{L_{t,x}^5(J)}F^\omega(J)^\frac{1}{2}}\norm{v}_{X^\sigma(J)}\,.
\end{align}
\end{proposition}
In particular, the above trilinear estimates yield the following estimate for the forcing term of~\eqref{eq:nls:f} at the regularity $H_x^\sigma$:
\begin{equation}
    \label{eq:trilinear-error}
    \norm{\mathcal{I}\parent{\cdot,\mathcal{N}(v+f)-\mathcal{N}(v)}}_{X^\sigma(J)}\lesssim C(\norm{f_0}_{H^s})F^\omega(J)+C\set{F^\omega(J)+\norm{v}_{L_{t,x}^5(J)}F^\omega(J)^\frac{1}{2}}\norm{v}_{X^\sigma(J)}\,.
\end{equation}
\begin{proof}
In the following, we take all the spacetime norms over $J\times\R^3$. Using the duality between $U^2$ and $V^2$ (see Proposition~\ref{proposition:dual:U2}), we have~\footnote{\ To apply Proposition~\ref{proposition:dual:U2} and to be in position to apply formula~\eqref{eq:duality:duhamel}, we need to make sure that $\mathcal{N}(u)$ is in $L_t^1(J;H_x^\sigma)$. We omit this short verification which is essentially contained in the analysis, and we refer to~\cite{benyi2015} where the authors prove that $P_{\leq N}\mathcal{N}(u)\in L_t^1(J;H_x^\sigma)$ for any $N\in2^\N$.} 
\[\norm{v}_{X^\sigma(J)}^2\leq \sum_N \mathcal{I}_N^2\,,\]
where, denoting $w^{(i)}\in\{v,f\}$ for $i\in\{1,2,3\}$, we have that for all $v\in X^\sigma(J)$
\[\mathcal{I}_N\coloneqq N^\sigma\norm{P_Nv}_{U_\Delta^2(J)}\leq CN^\sigma\underset{\norm{w}_{V_\Delta^2(J)}\leq1}{\sup}\absbig{\iint_{J\times\R^3}P_N\parent{w^{(1)}\overline{w}^{(2)}w^{(3)}}\overline{w}dxdt}\,.\]
Given a fixed function $w$ with $\norm{w}_{V_\Delta^2(J)}\leq1$, we perform a Littlewood-Paley decomposition of each term~\footnote{Note that the sums over $N_1,N_2,N_3$ under the integral are absolutely convergent in $\C$. Hence, we can intertwine the sum and the integral, and then apply the triangle inequality to obtain,~\eqref{eq:multi-int}.} and we are left to estimate quadrilinear spacetime integrals of the form
\begin{equation}
    \label{eq:multi-int} \mathcal{I}_N\leq\sum_{L=(N_1,N_2,N_3)}N_{(1)}^\sigma\absbig{\iint_{J\times\R^3}w_{N_1}^{(1)}\overline{w}_{N_2}^{(2)}w_{N_3}^{(3)}\overline{w_N}dxdt}\eqqcolon \sum_{L=(N_1,N_2,N_3)}\mathcal{I}_N(L)\,.
\end{equation}
For dyadic integers $N_1,N_2,N_3$, we denote by $N_{(1)} \geq N_{(2)} \geq N_{(3)}$ the non-increasing
ordering among them, and we use the shorthand notation $w_{N_i}\coloneqq P_{N_i}w^{(i)}\,, \ i\in\{1,2,3\}$. Note that it suffices to consider the nontrivial cases, where $N\lesssim N_{(1)}$. To be able to sum the terms $\mathcal{I}_N$ over $N$, we need to gain a negative power of $N$ after summing $\mathcal{I}_N^L$ over $L$. Hence, we shall bound the above quadrilinear integrals by some negative power of the highest frequency $N_{(1)}$, and by some appropriate norms of each $w_i$. Then, we proceed as in observation~\ref{observation:sum}.
\paragraph{\textbf{Three random terms}} Here we address the first Duhamel iteration, that is 
\[-i\int_0^t\e^{i(t-t')\Delta}\mathcal{N}(f,f,f)(t')dt'\,,\]
and we prove~\eqref{eq:tri-fff}. By symmetry, we can assume that $N_1\geq N_2 \geq N_3$ without loss of generality.
\begin{enumerate}[]
    \item \textbf{Case 1: High-high-high} $N_1^\frac{1}{2}\lesssim N_3$. Applying Hölder with $1=\frac{3}{10}+\frac{7}{10}$ yields
    \begin{multline*}
         \abs{\eqref{eq:multi-int}}= N_1^\sigma \absbig{\iint_{J\times\R^3}f_{N_1}\overline{f_{N_2}}f_{N_3}w_Ndxdt}\\
         \lesssim N_1^{\sigma-s} N_2^{-s}N_3^{-s} \prod_{1\leq j\leq3}\norm{\japbrak{\nabla}^sf_{N_j}}_{L_{t,x}^{30/7}}\norm{w_N}_{L_{t,x}^{10/3}}\\
         \lesssim N_1^{\sigma-2s}\prod_{1\leq j\leq3}\norm{\japbrak{\nabla}^sf_{N_j}}_{L_{t,x}^{10/3}{30}{7}}\norm{w_N}_{V_\Delta^2}\,.
    \end{multline*}
    Recall that $\sigma<2s$. Hence, we have a negative power of the highest frequency, and we can proceed as in observation ~\ref{observation:sum} to obtain
    \[
    \sum_{N_1,N_2,N_3,N}N_1^\sigma \absbig{\iint_{J\times\R^3}f_{N_1}\overline{f_{N_2}}f_{N_3}w_Ndxdt}\lesssim \sum_N N^{0-} \norm{\japbrak{\nabla}^sf}_{L_{t,x}^{10/3}}^3\norm{w_N}_{V_\Delta^2}\,.
    \]
    When we eventually sum over $N$, we apply Cauchy-Schwarz and use that $\norm{v}_{Y^0}\leq1$ to conclude that 
    \[
    \sum_N N^{0-}\norm{w_N}_{V_\Delta^2}\lesssim \parent{\sum_N\norm{w_N}_{V_\Delta^2}^2}^\frac{1}{2}\lesssim 1\,.
    \]
    We conclude similarly in the other cases.
    \item \textbf{Case 2: High-low-low} $N_2\leq N_1^{1/2}$. We apply Cauchy-Schwarz, we use interpolation and the bilinear estimate~\eqref{eq:bili} in its transferred version~\eqref{eq:bili:transf} to obtain
\begin{equation*}
\begin{split}
&\abs{\eqref{eq:multi-int}}= N_1^\sigma \absbig{\iint_{J\times\R^3}f_{N_1}\overline{f_{N_2}}f_{N_3}w_Ndxdt}
\lesssim
N_1^\sigma\norm{f_{N_1}f_{N_2}}_{L_{t,x}^2}\norm{f_{N_3}w_N}_{L_{t,x}^2}\\
&\lesssim N_1^\sigma N_1^{-\frac{1}{4}}N_2^\frac{1}{2}\norm{f_{N_1}(0)}_{L_x^2}^\frac{1}{2}\norm{f_{N_2}(0)}_{L_x^2}^\frac{1}{2}\norm{f_{N_1}}_{L_{t,x}^4}^\frac{1}{2}\norm{f_{N_2}}_{L_{t,x}^4}^\frac{1}{2}N^{+0}N^{-\frac{1}{2}}N_3\norm{f_{N_3}(0)}_{L_x^2}\norm{w_N}_{V_\Delta^2}\\
&\lesssim N_1^{\sigma-s}N^{-\frac{1}{2}+0}N_2^{\frac{1}{2}-s}N_3^{1-s}\norm{f_{N_1}(0)}_{H_x^s}^\frac{1}{2}\norm{f_{N_2}(0)}_{H_x^s}^\frac{1}{2}\norm{f_{N_3}(0)}_{H_x^s}\norm{\japbrak{\nabla}^sf_{N_1}}_{L_{t,x}^4}^\frac{1}{2}\norm{\japbrak{\nabla}^sf_{N_2}}_{L_{t,x}^4}^\frac{1}{2}\norm{w_N}_{V_\Delta^2}\\
&\lesssim N_1^{\sigma-2s+0}\norm{f_{N_1}(0)}_{H_x^s}^\frac{1}{2}\norm{f_{N_2}(0)}_{H_x^s}^\frac{1}{2}\norm{f_{N_3}(0)}_{H_x^s}\norm{\japbrak{\nabla}^sf_{N_1}}_{L_{t,x}^4}^\frac{1}{2}\norm{\japbrak{\nabla}^sf_{N_2}}_{L_{t,x}^4}^\frac{1}{2}\norm{w_N}_{V_\Delta^2}\,.
\end{split}
\end{equation*}
As in the first case, we use that $\sigma<2s$ and we sum over the $N_j$'s to get
     \[
    \sum_{N_1,N_2,N_3,N}N_1^\sigma \absbig{\iint_{J\times\R^3}f_{N_1}\overline{f_{N_2}}f_{N_3}w_Ndxdt}\lesssim  \norm{\japbrak{\nabla}^sf}_{L_{t,x}^4}\norm{f_0}_{H_x^s}^2\,.
    \]
    \item \textbf{Case 3: High-high-low} $N_3\leq N_1^\frac{1}{2}\leq N_2$. Similarly, it follows from Hölder's inequality, from the bilinear estimate~\eqref{eq:bili} and from the Strichartz embedding~\eqref{eq:strichartz:embedding} that
\begin{equation*}
\begin{split}
    &\absbig{\iint_{J\times\R^3}f_{N_1}\overline{f_{N_2}}f_{N_3}w_Ndxdt}
        \lesssim N_1^\sigma\norm{f_{N_1}f_{N_3}}_{L_{t,x}^2}\norm{f_{N_2}}_{L_{t,x}^5}\norm{w_N}_{L_{t,x}^{10/3}}\\
        &\lesssim N_1^{\sigma-\frac{1}{4}-s}N_3^{\frac{1}{2}-s}\norm{f_{N_1}(0)}_{H_x^s}^\frac{1}{2}\norm{f_{N_3}(0)}_{H_x^s}^\frac{1}{2}\norm{\japbrak{\nabla}^sf_{N_1}}_{L_{t,x}^4}^\frac{1}{2}\norm{\japbrak{\nabla}^sf_{N_3}}_{L_{t,x}^4}^\frac{1}{2}N_2^{-s}\norm{\japbrak{\nabla}^sf_{N_2}}_{L_{t,x}^5}\norm{w_N}_{V_\Delta^2}\\
        &\lesssim N_1^{\sigma-2s}\norm{f_{N_1}(0)}_{H_x^s}^\frac{1}{2}\norm{f_{N_3}(0)}_{H_x^s}^\frac{1}{2}\norm{\japbrak{\nabla}^sf_{N_1}}_{L_{t,x}^4}^\frac{1}{2}\norm{\japbrak{\nabla}^sf_{N_3}}_{L_{t,x}^4}^\frac{1}{2}N_2^{-s}\norm{\japbrak{\nabla}^sf_{N_2}}_{L_{t,x}^5}\norm{w_N}_{V_\Delta^2}\,.
\end{split}
\end{equation*}
    Summing over the different dyadic integers yields
     \[
    \sum_{N_1,N_2,N_3,N}N_1^\sigma \absbig{\iint_{J\times\R^3}f_{N_1}\overline{f_{N_2}}f_{N_3}w_Ndxdt}\lesssim \norm{\japbrak{\nabla}^sf}_{L_{t,x}^4}\norm{\japbrak{\nabla}^sf}_{L_{t,x}^5}\norm{f_0}_{H_x^s}\,.
    \]
\end{enumerate}
\paragraph{\textbf{Mixed terms}}  Here we consider the case when at least one term is random, say $w^{(1)}=f$, and at least another one is deterministic, say $w^{(2)}=v$. As for the last term $w^{(3)}\in\{f,v\}$,  we always place it in $L_{t,x}^5(J\times\R^3)$. We prove~\eqref{eq:tri-ffv} and~\eqref{eq:tri-fvv}.
\begin{enumerate}[]
\item \textbf{Case 1: One deterministic term comes with the highest frequency} Without loss of generality, we assume that $N_{(1)}=N_2$. By applying Hölder's inequality, we get
    \[
      \abs{\eqref{eq:multi-int}}=N_2^\sigma\absbig{\iint_{J\times\R^3}f_{N_1}\overline{v}_{N_2}w_{N_3}^{(3)}\overline{w}_{N}dxdt}
      \lesssim N_2^\sigma \norm{f_{N_1}v_{N_2}}_{L_{t,x}^2}\norm{w_{N_3}^{(3)}}_{L_{t,x}^5}\norm{w_N}_{L_{t,x}^{10/3}}\,.
     \]
     The transferred bilinear estimate and Strichartz embedding yield
      \begin{multline*}
      \abs{\eqref{eq:multi-int}}\lesssim N_2^\sigma \norm{f_{N_1}v_{N_2}}_{L_{t,x}^2}^\frac{1}{2}\norm{f_{N_1}}_{L_{t,x}^5}^\frac{1}{2}\norm{v_{N_2}}_{L_{t,x}^{10/3}}^\frac{1}{2} \norm{w_{N_3}^{(3)}}_{L_{t,x}^5}\norm{w_N}_{V_\Delta^2}\\
      \lesssim N_2^{-\frac{1}{4}}N_1^{\frac{1}{2}-s}\parent{N_2^\sigma\norm{v_{N_2}}_{U_\Delta^2}} \norm{f_{N_1}(0)}_{H_x^s}^\frac{1}{2}\norm{\japbrak{\nabla}^sf_{N_1}}_{L_{t,x}^5}^\frac{1}{2}\norm{w_{N_3}^{(3)}}_{L_{t,x}^5}\norm{w_N}_{V_\Delta^2}\,,
    \end{multline*}
 Finally, we sum over the $N_i$'s to obtain that the contribution of this term to~\eqref{eq:multi-int} is less than 
    \begin{multline*}
      \sum_{N_i}N_2^\sigma\absbig{\iint_{J\times\R^3}f_{N_1}\overline{v}_{N_2}w_{N_3}^{(3)}\overline{w}_{N}dxdt}\\
      \lesssim \sum_{N_i}N_2^{-\frac{1}{4}}N_1^{\frac{1}{2}-s}N_2^\sigma\norm{v_{N_2}}_{U_\Delta^2}\norm{P_{N_1}f_0}_{H_x^s}^\frac{1}{2}\norm{f_{N_1}}_{L_{t,x}^5}^\frac{1}{2}\norm{w_{N_3}^{(3)}}_{L_{t,x}^5}\norm{w_N}_{V_\Delta^2} \\
      \lesssim\sum_{N_i} N_2^{\frac{1}{4}-s}\norm{P_{N_1}f_0}_{H_x^s}^\frac{1}{2}\norm{f_{N_1}}_{L_{t,x}^5}^\frac{1}{2}\norm{w_{N_3}^{(3)}}_{L_{t,x}^5}\norm{w_N}_{V_\Delta^2}\\
      \lesssim N^{0-}\norm{f_0}_{H_x^{1/2}}^\frac{1}{2}\norm{\japbrak{\nabla}^sf}_{L_{t,x}^5}^\frac{1}{2}\norm{w^{(3)}}_{L_{t,x}^5}\norm{v}_{X^\sigma}\norm{w_N}_{V_\Delta^2}\,.
    \end{multline*}
Note that we used the assumption that $\frac{1}{4}<s$.
    \item \textbf{Case 2: One random term comes with the highest frequency:} $N_{(1)}=N_1$. We distinguish between two cases.
\begin{enumerate}[]
    \item \textbf{Case 2a: High-high interactions } $N_1^\frac{1}{2}\leq N_2$. We apply Hölder's inequality and the Strichartz embedding~\eqref{eq:strichartz:embedding}, to get  
    \begin{multline*}
        \abs{\eqref{eq:multi-int}}=N_1^\sigma\absbig{\iint_{J\times\R^3}f_{N_1}\overline{v}_{N_2}w_{N_3}^{(3)}\overline{w}_{N}dxdt}\\
        \lesssim N_1^{\sigma-s}N_2^{-\sigma}\norm{\japbrak{\nabla}^sf_{N_1}}_{L_{t,x}^5}\norm{\japbrak{\nabla}^\sigma v_{N_2}}_{L_{t,x}^{10/3}}\norm{w_{N_3}^{(3)}}_{L_{t,x}^5}\norm{w_N}_{L_{t,x}^{10/3}}\\
        \lesssim N_1^{\frac{\sigma}{2}-s}\norm{\japbrak{\nabla}^sf_{N_1}}_{L_{t,x}^5}\norm{\japbrak{\nabla}^\sigma v_{N_2}}_{U_\Delta^2}\norm{w_{N_3}^{(3)}}_{L_{t,x}^5}\norm{w_N}_{V_\Delta^2}\,.
    \end{multline*}
    Since $\sigma<2s$, we can sum over the dyadic frequencies and obtain
    \[\sum_{N_i}N_1^\sigma\absbig{\iint_{J\times\R^3}f_{N_1}\overline{v}_{N_2}w_{N_3}^{(3)}\overline{w}_{N}dxdt}\lesssim N^{0-}\norm{\japbrak{\nabla}^sf}_{L_{t,x}^5}\norm{v}_{L_{t,x}^5}\norm{v}_{X^\sigma}\norm{w_N}_{V_\Delta^2}\,.\]
    \item \textbf{Case 2b: High-low-low interactions} $N_2,N_3\ll N_1^\frac{1}{2}$, and $N_1\sim N$. We make a different analysis depending on the type of $w^{(3)}$.
    \begin{enumerate}[]
        \item \textbf{Case 2b(i)}: $w^{(3)}=f$. The idea is to use Cauchy Schwarz in order to apply the bilinear estimate twice to gain derivatives, and to use Hölder's inequality in order to gain smallness.
    \begin{multline*}
        \abs{\eqref{eq:multi-int}}=N_1^\sigma\absbig{\iint_{J\times\R^3}f_{N_1}\overline{v}_{N_2}f_{N_3}\overline{w}_{N}dxdt}\leq N_1^\sigma\norm{f_{N_1}v_{N_2}}_{L_{t,x}^2}\norm{f_{N_3}w_N}_{L_{t,x}^2}\\
        \lesssim N_1^\sigma \norm{f_{N_1}v_{N_2}}_{L_{t,x}^2}^\frac{1}{2}\norm{f_{N_1}}_{L_{t,x}^5}^\frac{1}{2}\norm{v_{N_2}}_{L_{t,x}^{10/3}}^\frac{1}{2}\norm{f_{N_3}w_N}_{L_{t,x}^2}^\frac{1}{2}\norm{f_{N_3}}_{L_{t,x}^5}^\frac{1}{2}\norm{w_N}_{L_{t,x}^{10/3}}^\frac{1}{2}
    \end{multline*}
Applying bilinear Strichartz estimate yields 
    \begin{multline*}
        \abs{\eqref{eq:multi-int}}\lesssim N_1^{\sigma-s-\frac{1}{4}} N_2^{\frac{1}{2}-\sigma}N_3^{\frac{1}{2}-\frac{s}{2}}N^{-\frac{1}{4}+0}\norm{\japbrak{\nabla}^sf_{N_1}}_{L_{t,x}^5}^\frac{1}{2}\norm{f_{N_1}(0)}_{H_x^s}^\frac{1}{2}\\
        \norm{\japbrak{\nabla}^\sigma v_{N_2}}_{U_\Delta^2}^\frac{1}{2}\norm{f_{N_3}(0)}_{H_x^s}^\frac{1}{2}\norm{\japbrak{\nabla}^\sigma v_{N_2}}_{L_{t,x}^{10/3}}^\frac{1}{2}\norm{\japbrak{\nabla}^sf_{N_3}}_{L_{t,x}^5}^\frac{1}{2}\norm{w_N}_{V_\Delta^2}\\
        \lesssim N_1^{\sigma-s-\frac{1}{2}+0}N_2^{\frac{1}{2}-\sigma}N_3^{\frac{1}{2}-\frac{s}{2}}\norm{\japbrak{\nabla}^\sigma v_{N_2}}_{U_\Delta^2}\norm{w_N}_{V_\Delta^2}\\
        \parent{\norm{f_{N_1}(0)}_{H_x^s}\norm{f_{N_3}(0)}_{H_x^s}}^\frac{1}{2}\parent{\norm{\japbrak{\nabla}^sf_{N_1}}_{L_{t,x}^5}\norm{\japbrak{\nabla}^sf_{N_3}}_{L_{t,x}^5}}^\frac{1}{2}
         \,.
        \end{multline*}
Finally, we observe that
        \[
       N_1^{\sigma-s-\frac{1}{2}+0}N_2^{\frac{1}{2}-\sigma}N_3^{\frac{1}{2}-\frac{s}{2}}\leq N_1^{\sigma-s-\frac{1}{2}+0}N_3^{1-\sigma-\frac{s}{2}}\leq N_1^{\frac{\sigma}{2}-s}  \,, 
       \]
and we can sum over the $N_i$'s to see that the contribution for this term is less than
\[\sum_{N_i}N_1^\sigma\absbig{\iint_{J\times\R^3}f_{N_1}\overline{v}_{N_2}f_{N_3}\overline{w}_{N}dxdt}\lesssim N^{0-} \norm{f_0}_{H_x^s}\norm{\japbrak{\nabla}^sf}_{L_{t,x}^5}\norm{v}_{X^\sigma}\,.\]
\item \textbf{Case 2b(ii)}: $w^{(3)}=v$. We use Hölder's inequality, interpolation, and then apply the bilinear estimate~\eqref{eq:bili:transf}
\begin{equation*}
\begin{split}
\abs{\eqref{eq:multi-int}}&=N_1^\sigma\absbig{\iint_{J\times\R^3}f_{N_1}\overline{v}_{N_2}v_{N_3}\overline{w}_{N}dxdt}\leq N_1^\sigma\norm{f_{N_1}v_{N_3}}_{L_{t,x}^2}\norm{v_{N_2}w_N}_{L_{t,x}^2}\\
&\lesssim N_1^\sigma \parent{\norm{f_{N_1}}_{L_{t,x}^{10/3}}\norm{v_{N_2}}_{L_{t,x}^5}\norm{f_{N_1}v_{N_3}}_{L_{t,x}^2}\norm{v_{N_2}w_N}_{L_{t,x}^2}}^\frac{1}{2}\parent{\norm{v_{N_2}}_{L_{t,x}^5}\norm{w_N}_{L_{t,x}^{10/3}}\norm{v_{N_2}w_N}_{L_{t,x}^2}}^\frac{1}{2} \\
&\begin{split}
\lesssim_\delta N_1^{\sigma-s-\frac{1}{2}+\delta}N_2^\frac{1-\sigma}{2}N_3^\frac{1-\sigma}{2}&\parent{N_1^s\norm{f_{N_1}}_{L_{t,x}^{10/3}}\norm{v_{N_2}}_{L_{t,x}^5}N_1^s\norm{f_{N_1}(0)}_{ L_x^2}N^\sigma\norm{v_{N_2}}_{U_\Delta^2}}^\frac{1}{2}\\
&\parent{\norm{v_{N_3}}_{L_{t,x}^5}\norm{w_N}_{L_{t,x}^{10/3}}N_3^\sigma\norm{v_{N_3}}_{U_\Delta^2}\norm{w_N}_{V_\Delta^2}}^\frac{1}{2}\,.
\end{split}
\end{split}
\end{equation*}
Since, $N_3\leq N_2 \leq N_1^\frac{1}{2}$, we have
\[
N_1^{\sigma-s-\frac{1}{2}+\delta}N_2^\frac{1-\sigma}{2}N_3^\frac{1-\sigma}{2}\leq N_1^{\frac{\sigma}{2}-s+\delta}\,.
\]
Next, we use the assumption that $\sigma<2s$, and we chose $\delta<\frac{\sigma}{2}-s$. Subsequently, we can sum over the $N_i's$ and we conclude that
\[\sum_{N_i}N_1^\sigma\absbig{\iint_{J\times\R^3}f_{N_1}\overline{v}_{N_2}v_{N_3}\overline{w}_{N}dxdt}\lesssim N^{0-} \norm{f_0}_{H_x^s}^\frac{1}{2}\norm{\japbrak{\nabla}^sf}_{L_{t,x}^{10/3}}^\frac{1}{2}\norm{v}_{L_{t,x}^5}\norm{v}_{X^\sigma}\norm{w_N}_{V_\Delta^2}\,.\]
    \end{enumerate}
    \end{enumerate}
\end{enumerate}
This finishes the proofs of estimates~\eqref{eq:tri-ffv} and~\eqref{eq:tri-fvv}, and of Proposition~\ref{prop:trilinear}.
\end{proof}
In order to perform a double bootstrap argument with $\norm{v}_{L_{t,x}^5(J)}$ and $\norm{v}_{X^\sigma(J)}$ to prove local well-posedness, we need to estimate the spacetime critical $L_{t,x}^5$ norm  of $v$.
\begin{proposition}[Additional trilinear estimates]
\label{prop:trilinear:5}
\begin{align}
\label{eq:trili-5-ffv}
\norm{\mathcal{I}\parent{\cdot,\mathcal{N}(f,f,v)}}_{L_{t,x}^5(J)}&\lesssim\norm{f_0}_{H_x^s}^\frac{1}{2} F^\omega(J)^\frac{3}{2}\norm{v}_{X^{1/2}(J)} \\
\label{eq:trili-5-fvv}
\norm{\mathcal{I}\parent{\cdot,\mathcal{N}(f,v,v)}}_{L_{t,x}^5(J)}&\lesssim \norm{f_0}_{H_x^s}^\frac{1}{2} F^\omega(J)^\frac{1}{2}\norm{v}_{L_{t,x}^5(J)}\norm{v}_{X^{1/2}(J)}\,.
\end{align}
\end{proposition}
\begin{proof}
By the $\operatorname{TT}^*$ Strichartz estimate and the Sobolev embedding, we have 
\[ 
\norm{\mathcal{I}\parent{\cdot,\mathcal{N}}}_{L_{t,x}^5(J)}\lesssim\norm{\abs{\nabla}^\frac{1}{2}\mathcal{I}\parent{\cdot,\mathcal{N}}}_{L_t^5L_x^{30/11}}\lesssim\norm{\abs{\nabla}^\frac{1}{2}\mathcal{N}}_{L_{t,x}^{10/7}}\,.
\]
Once again we perform a Littlewood-Paley decomposition of each term and conduct a case by case analysis. We write $N_{(1)},N_{(2)},N_{(3)}$ the non-increasing ordering among them.
\paragraph{\textbf{Proof of}~\eqref{eq:trili-5-ffv}}
We need to estimate 
\[N_{(1)}^\frac{1}{2}\norm{f_{N_1}f_{N_2}v_{N_3}}_{L_{t,x}^{10/7}}\,.\]
Without loss of generality, we may assume that $N_1\geq N_2$.
\begin{enumerate}[]
    \item \textbf{Case 1:} $N_{(1)}=N_1$. We use Hölder's inequality with $\frac{7}{10}=\frac{2}{5}+\frac{3}{10}$ and get
    \[
    N_1^\frac{1}{2}\norm{f_{N_1}f_{N_2}v_{N_3}}_{L_{t,x}^{10/7}}\leq  N_1^{\frac{1}{2}-2s}\norm{\japbrak{\nabla}^sf_{N_1}}_{L_{t,x}^5}\norm{\japbrak{\nabla}^sf_{N_2}}_{L_{t,x}^5}\norm{v_{N_3}}_{L_{t,x}^{10/3}}\,.
    \]
Under the assumption that $1/4<s$ we can use Observation~\ref{observation:sum} and see that in this case, 
\[
\sum_{N_i}N_1^\frac{1}{2}\norm{f_{N_1}f_{N_2}v_{N_3}}_{L_{t,x}^{10/7}}\lesssim \norm{\japbrak{\nabla}^sf}_{L_{t,x}^5}^2\norm{v}_{L_{t,x}^5}\norm{v}_{X^0}\,.
\]
    \item \textbf{Case 2:} $N_{(1)}=N_3$.
    \begin{enumerate}[]
        \item \textbf{Case 2a}: $N_1<N_3^\frac{1}{2}$. In this case we shall apply the bilinear estimate once.
        \begin{multline*}
        N_{3}^\frac{1}{2}\norm{f_{N_1}f_{N_2}v_{N_3}}_{L_{t,x}^{10/7}}\leq N_3^\frac{1}{2}\norm{f_{N_2}v_{N_3}}_{L_{t,x}^2}\norm{f_{N_1}}_{L_{t,x}^5}\\
        \lesssim N_3^{-\frac{1}{4}}N_2^{\frac{1}{2}-s}N_1^{-s}\norm{f_{N_2}(0)}_{H_x^s}^\frac{1}{2}\norm{\japbrak{\nabla}^sf_{N_2}}_{L_{t,x}^5}^\frac{1}{2}\parent{N_3\norm{v_{N_3}}_{U_\Delta^2}\norm{v_{N_3}}_{L_{t,x}^{10/3}}}^\frac{1}{2}\norm{\japbrak{\nabla}^sf_{N_1}}_{L_{t,x}^5}\\
        \lesssim N_3^{-s}\norm{f_{N_2}(0)}_{H_x^s}^\frac{1}{2}\norm{\japbrak{\nabla}^sf_{N_2}}_{L_{t,x}^5}^\frac{1}{2}\parent{N_3^\frac{1}{2}\norm{v_{N_3}}_{U_\Delta^2}}\norm{\japbrak{\nabla}^sf_{N_1}}_{L_{t,x}^5}\,.
        \end{multline*}
        Summing over the $N_i$'s yields
        \[
\sum_{N_i}N_1^\frac{1}{2}\norm{f_{N_1}f_{N_2}v_{N_3}}_{L_{t,x}^{10/7}}\lesssim \norm{f_0}_{H_x^s}^\frac{1}{2}\norm{\japbrak{\nabla}^sf}_{L_{t,x}^5}^\frac{3}{2}\norm{v}_{X^\frac{1}{2}}\,.
\]
        \item \textbf{Case 2b}: $N_3^\frac{1}{2}\leq N_1$. In this case there is no need to apply the bilinear estimate and we only use Hölder. We have
        \begin{multline*}
        N_{3}^\frac{1}{2}\norm{f_{N_1}f_{N_2}v_{N_3}}_{L_{t,x}^{10/7}}\leq N_3^\frac{1}{2}N_1^{-s}\norm{\japbrak{\nabla}^sf_{N_1}}_{L_{t,x}^5}\norm{f_{N_2}}_{L_{t,x}^5}\norm{v_{N_3}}_{L_{t,x}^{10/3}}\\\leq N_3^{-\frac{s}{2}}
        \norm{\japbrak{\nabla}^sf_{N_1}}_{L_{t,x}^5}\norm{f_{N_2}}_{L_{t,x}^5}\parent{N_3^\frac{1}{2}\norm{v_{N_3}}_{U_\Delta^2}}\,.
        \end{multline*}
        Then we sum over the $N_i$'s and get 
        \[
\sum_{N_i}N_1^\frac{1}{2}\norm{f_{N_1}f_{N_2}v_{N_3}}_{L_{t,x}^{10/7}}\lesssim \norm{\japbrak{\nabla}^sf}_{L_{t,x}^5}\norm{f}_{L_{t,x}^5}\norm{v}_{X^\frac{1}{2}}\,.
\]
     \end{enumerate}
\end{enumerate}
This concludes the proof of~\eqref{eq:trili-5-ffv}.
\paragraph{\textbf{Proof of}~\eqref{eq:trili-5-fvv}}
We need to estimate 
\[N_{(1)}^\frac{1}{2}\norm{f_{N_1}v_{N_2}v_{N_3}}_{L_{t,x}^{10/7}}\,.\]
Without loss of generality, we assume that $N_2\geq N_3$.
\begin{enumerate}[]
    \item \textbf{Case 1:} $N_{(1)}=N_1$. We use Hölder with $\frac{7}{10}=\frac{1}{2}+\frac{1}{5}$ to apply the bilinear estimate once. 
    \begin{multline*}
    N_{1}^\frac{1}{2}\norm{f_{N_1}v_{N_2}v_{N_3}}_{L_{t,x}^{10/7}}\leq \norm{f_{N_1}v_{N_3}}_{L_{t,x}^2}\norm{v_{N_2}}_{L_{t,x}^5}\\
   \lesssim N_1^{\frac{1}{4}-s}N_3^\frac{1}{2}\norm{\japbrak{\nabla}^sf_{N_1}}_{L_{t,x}^5}^\frac{1}{2}\norm{f_{N_1}(0)}_{H_x^s}^\frac{1}{2}\norm{v_{N_3}}_{U_\Delta^2}^\frac{1}{2}\norm{v_{N_3}}_{L_{t,x}^{10/3}}^\frac{1}{2}\norm{v_{N_2}}_{L_{t,x}^5}\\
   \lesssim N_1^{\frac{1}{4}-s}\norm{\japbrak{\nabla}^sf_{N_1}}_{L_{t,x}^5}^\frac{1}{2}\norm{f_{N_1}(0)}_{H_x^s}^\frac{1}{2}\norm{v_{N_3}}_{X^\frac{1}{2}}\norm{v_{N_2}}_{L_{t,x}^5}\,.
    \end{multline*}
Under the assumption that $1/4<s$ we can use Observation~\ref{observation:sum} and see that in this case,
\[
\sum_{N_i}N_1^\frac{1}{2}\norm{f_{N_1}v_{N_2}v_{N_3}}_{L_{t,x}^{10/7}}\lesssim \norm{f_0}_{H_x^s}^\frac{1}{2}\norm{\japbrak{\nabla}^sf}_{L_{t,x}^5}^\frac{1}{2}\norm{v}_{L_{t,x}^5}\norm{v}_{X^\frac{1}{2}}\,.
\]
\item \textbf{Case 2:} $N_{(1)}=N_2$.
\begin{enumerate}[]
\item \textbf{Case 2a:} $N_1< N_2^\frac{1}{2}$. We proceed similarly, and get 
\begin{multline*}
N_2^\frac{1}{2}\norm{f_{N_1}v_{N_2}v_{N_3}}_{L_{t,x}^{10/7}}\leq \norm{f_{N_1}v_{N_2}}_{L_{t,x}^2}\norm{v_{N_3}}_{L_{t,x}^5}\\
\lesssim N_2^\frac{1}{4}N_1^{\frac{1}{2}-s}\norm{\japbrak{\nabla}^sf_{N_1}}_{L_{t,x}^5}^\frac{1}{2}\norm{f_{N_1}(0)}_{H_x^s}^\frac{1}{2}\norm{v_{N_2}}_{U_\Delta^2}^\frac{1}{2}\norm{v_{N_2}}_{L_{t,x}^{10/3}}^\frac{1}{2}\norm{v_{N_3}}_{L_{t,x}^5}\\
   \lesssim N_2^{-\frac{s}{2}}\norm{\japbrak{\nabla}^sf_{N_1}}_{L_{t,x}^5}^\frac{1}{2}\norm{f_{N_1}(0)}_{H_x^s}^\frac{1}{2}\norm{v_{N_2}}_{X^\frac{1}{2}}\norm{v_{N_3}}_{L_{t,x}^5}\,.
    \end{multline*}
Hence, 
\[
\sum_{N_i}N_2^\frac{1}{2}\norm{f_{N_1}v_{N_2}v_{N_3}}_{L_{t,x}^{10/7}}\lesssim \norm{f_0}_{H_x^s}^\frac{1}{2}\norm{\japbrak{\nabla}^sf}_{L_{t,x}^5}^\frac{1}{2}\norm{v}_{L_{t,x}^5}\norm{v}_{X^\frac{1}{2}}\,.
\]
    \item \textbf{Case 2b:} $N_2^\frac{1}{2}\leq N_1$. We apply Hölder and get
    \begin{multline*}
        N_2^\frac{1}{2}\norm{f_{N_1}v_{N_2}v_{N_3}}_{L_{t,x}^{10/7}}\leq N_2^\frac{1}{2}\norm{f_{N_1}}_{L_{t,x}^5}\norm{v_{N_2}}_{L_{t,x}^{10/3}}\norm{v_{N_3}}_{L_{t,x}^5}\\
        \leq N_2^{-\frac{s}{2}}\norm{\japbrak{\nabla}^sf_{N_1}}_{L_{t,x}^5}\parent{N_2^\frac{1}{2}\norm{v_{N_2}}_{U_\Delta^2}}\norm{v_{N_3}}_{L_{t,x}^5}\,.
    \end{multline*}
    We deduce from this and from Observation~\ref{observation:sum} that in this case,
    \[
\sum_{N_i}N_2^\frac{1}{2}\norm{f_{N_1}v_{N_2}v_{N_3}}_{L_{t,x}^{10/7}}\lesssim \norm{\japbrak{\nabla}^sf}_{L_{t,x}^5}\norm{v}_{L_{t,x}^5}\norm{v}_{X^\frac{1}{2}}\,.
\]
    \end{enumerate}
\end{enumerate}
\end{proof}
\subsection{Local well-posedness}
\label{section:lwp}
In the following local well-posedness statement, the smallness assumption comes from the $L_{t,x}^5$ norm of the free evolution of the initial data. Since the problem is at a subcritical regularity scale $1/2<\sigma$, we could avoid using such a critical norm. Nevertheless, the $L_{t,x}^5$ norm  provides a blow-up and scattering criterion that can be easily deduced from such a local well-posedness result. Then, we prove in the next section that this blow-up criterion can be exploited as soon as we are able to obtain an a priori uniform bound on the $H^\sigma$ norm of the solution.
\begin{proposition}[Local well-posedness]
\label{prop:lwp}
Let $1/4<s\leq 1/2$ and $\omega\in \Omega_\alpha$. Then, for every $1/2<\sigma<2s$, the Cauchy problem~\eqref{eq:nls:f} with data $v(t_0)\in H^\sigma$ is locally well-posed. More precisely, for every $E>0$, there exists $\epsilon_0(E)$ such that for every $\epsilon\leq \epsilon_0$, every interval $J$ that contains $t_0$ on which  
\begin{equation}
\label{eq:smallness-assumption-lwp}
F^\omega(J)\leq \epsilon,\quad \norm{\e^{it\Delta}v_0}_{L_{t,x}^5(J)}\leq\epsilon\,,
\end{equation}
and for every $v_0\in H^\sigma$ with $\norm{v_0}_{H^\sigma}\leq E$, there exists a unique solution $v$ to~\eqref{eq:nls:f} in $X^\sigma(J)$, with data $v(t_0)=v_0$. Moreover, there exists $C=C(\norm{f_0}_{H^s})>0$ such that
\[\norm{v}_{X^\sigma(J)}\leq 2E,\quad \norm{v}_{L_{t,x}^5(J)}\leq C\epsilon\,,\ \text{and}\ v\in C(J;H^\sigma)\,.\]
\end{proposition}
\begin{proof}
We perform a contraction mapping argument in the Banach space, 
\[
\mathcal{B}_{E,\epsilon}=\set{v\in X^\sigma(J)\mid \norm{v}_{X^\sigma(J)}\leq 2E,\quad \norm{v}_{L_{t,x}^5(J)}\leq C\epsilon}\,,
\]
where $C=C(\norm{f_0}_{H^\sigma})>0$ is an irrelevant constant that comes from the multilinear estimates of Proposition~\ref{prop:trilinear:5}. We endow this space with the norm $X^\sigma(J)$, and we define the mapping
\[
\Phi: v\in \mathcal{B}_{E,\epsilon} \mapsto \e^{it\Delta}v_0 -i\int_{t_0}^t\e^{i(t-t')\Delta}\mathcal{N}(v+f)(t')dt'\,.
\]
From the Duhamel integral representation formula, we see that $v$ is solution to~\eqref{eq:nls:f} with $v(t_0)=v_0$ if an only if $\Phi(v)=v$. Let us prove that 
\[
\Phi(\mathcal{B}_{E,\epsilon})\subseteq \mathcal{B}_{E,\epsilon}
\]
when $\epsilon\ll E^{-1}$.
We get from Strichartz and from the trilinear estimate~\eqref{eq:trilinear-error} that for all $v\in\mathcal{B}_{E,\epsilon}$,
\begin{equation*}
\begin{split}
    \norm{\Phi(v)}_{X^\sigma(J)}&\leq \norm{v_0}_{H_x^\sigma}+\norm{\mathcal{I}\parent{\cdot,\mathcal{N}(v)}}_{X^\sigma(J)}+\norm{\mathcal{I}\parent{\cdot,\mathcal{N}(v+f)-\mathcal{N}(v)}}_{X^\sigma(J)}\\
    &\leq E+C\norm{\japbrak{\nabla}^\sigma\parent{\mathcal{N}(v)}}_{L_{t,x}^{10/7}}+C(\norm{f_0}_{H^s})F^\omega(J)\\
    &\quad \quad +C(\norm{f_0}_{H^s})\set{F^\omega(J)+\norm{v}_{L_{t,x}^5(J)}}\norm{v}_{X^\sigma(J)}\,.
\end{split}
\end{equation*}
Applying the Fractional Leibniz rule with $\frac{7}{10}=\frac{2}{5}+\frac{3}{10}$ yields
\[
\norm{\japbrak{\nabla}^\sigma\parent{\mathcal{N}(v)}}_{L_{t,x}^{10/7}}\leq\norm{v}_{L_{t,x}^5(J)}^2\norm{\japbrak{\nabla}^\sigma v}_{L_{t,x}^{10/3}}\leq \norm{v}_{L_{t,x}^5(J)}^2\norm{v}_{X^\sigma(J)}\leq C^2\epsilon^2E\,.
\]
Hence, for all $v\in\mathcal{B}_{E,\epsilon}$ and under the smallness assumption~\eqref{eq:smallness-assumption-lwp} we get 
\begin{equation}
\label{eq:lwp:1}
\norm{\Phi(v)}_{X^\sigma(J)}\leq E + C\epsilon^2E + C(1+C)\epsilon E\,.
\end{equation}
Similarly, we have
\begin{multline}
\label{eq:lwp-5}
\norm{\Phi(v)}_{L_{t,x}^5}\leq\norm{\e^{it\Delta}v_0}_{L_{t,x}^5(J)}+\norm{\mathcal{I}\parent{\cdot,\mathcal{N}(v)}}_{L_{t,x}^5(J)}
+\norm{\mathcal{I}\parent{\cdot,\mathcal{N}(f)}}_{L_{t,x}^5(J)}\\+
\norm{\mathcal{I}\parent{\cdot,\mathcal{N}(f,f,v)+\mathcal{N}(f,v,v)}}_{L_{t,x}^5(J)}\,.
\end{multline}
First, observe from the Sobolev embedding and from~\eqref{eq:tri-fff} that
\[
\norm{\mathcal{I}\parent{\cdot,\mathcal{N}(f)}}_{L_{t,x}^5(J)}\leq C\norm{\mathcal{I}\parent{\cdot,\mathcal{N}(f)}}_{X^{1/2}(J)}\leq C(\norm{f_0}_{H^\sigma})F^\omega(J)\,.
\]
By Hölder's inequality, Sobolev embedding and Strichartz embedding we have
\[
\norm{\mathcal{I}\parent{\cdot,\mathcal{N}(v)}}_{L_{t,x}^5(J)}\leq\norm{v}_{X^{1/2}(J)}\norm{v}_{L_{t,x}^5(J)}^2\,.
\]
Then, using the trilinear estimates presented in Proposition~\ref{prop:trilinear:5} we get that for any $v\in\mathcal{B}_{E,\epsilon}$,
\[
\norm{\mathcal{I}\parent{\cdot,\mathcal{N}(f,f,v)+\mathcal{N}(f,v,v)}}_{L_{t,x}^5(J)}\leq C\parent{F^\omega(J)^\frac{3}{2}+F^\omega(J)^\frac{1}{2}\norm{v}_{L_{t,x}^5(J)}}\norm{v}_{X^\sigma(J)}\,.
\]
Collecting the above estimates, we deduce from~\eqref{eq:lwp-5} that under the smallness condition~\eqref{eq:smallness-assumption-lwp}, we have for $v\in\mathcal{B}_{E,\epsilon}$
\begin{equation}
    \label{eq:lwp:2}
    \norm{\Phi(v)}_{L_{t,x}^5}\leq \epsilon + C\epsilon + 2C^3\epsilon^2E + 2CE\parent{\epsilon^\frac{3}{2}+C\epsilon^\frac{3}{2}}\,.
\end{equation}
Combining~\eqref{eq:lwp:1},~\eqref{eq:lwp:2} and choosing $\eta=3C\epsilon$, $R=2CE$, for $\epsilon\leq\epsilon_0(E)$ with, say $\epsilon_0(E)\sim E^{-1}$ when $E\gg1$, we obtain 
\[
\norm{\Phi(v)}_{X^\sigma(J)}\leq 2CE,\quad \norm{\Phi(v)}_{L_{t,x}^5(J)}\leq 3C\epsilon\,.
\]
for any $v\in\mathcal{B}_{E,\epsilon}$. Thus , $\Phi:\  \mathcal{B}_E\mapsto \mathcal{B}_E$. Similarly, we prove that 
\[
\norm{\Phi(v_1)-\Phi(v_2)}_{X^\sigma(J)}\leq C\parent{\norm{v_1}_{L_{t,x}^5(J)}\norm{v_1}_{X^\sigma(J)}+\norm{v_2}_{L_{t,x}^5(J)}\norm{v_2}_{X^\sigma(J)}}  \norm{v_1-v_2}_{X^\sigma(J)}\,.
\]
Choosing $\epsilon_0(E)\ll E^{-1}$, we see that $\Phi$ is a contraction mapping, and it admits a unique fixed point $v$, solution to~\eqref{eq:nls:f}. Finally, we deduce the continuity of $t\to v(t,\cdot)$ from the embedding $X^\sigma(J)\xhookrightarrow{}C\parent{J;H_x^\sigma(\R^3)}\,.$
\end{proof}
\begin{proposition}[Blow-up criterion]
\label{prop:blow-up}
Let $v$ be the maximal lifespan solution to~\eqref{eq:nls:f} on $J^*\times\R^3$ given by the local well-posedness theory. If we have
\begin{equation}
\label{eq:blow-up:ass}
\norm{v}_{L_{t,x}^5(J^*)}<+\infty\,,
\end{equation}
then $\sup J^* = +\infty$ and the solution scatters as $t$ goes to $+\infty$.
\end{proposition}
\begin{proof}
Denote $E=\norm{v(0)}_{H_x^\sigma}$, and $\epsilon>0$ to be chosen later on. We proceed by contraction and assume that $\norm{v}_{L_{t,x}^5}<+\infty$, but $T^*=\sup J^*<+\infty$.  We will raise the contradiction by extending the solution up to $T^*$. By the local Cauchy theory from Proposition~\ref{prop:lwp}, it is enough to prove that there exist a constant $C(E,f_0)$, $t_0\in J^*$ and $\delta>0$ such that 
\begin{align}
\label{eq:blow-up1}
\norm{v(t)}_{L_t^\infty H_x^\sigma(J^*)}\leq C(E,f_0)\,,\\
\label{eq:blow-up2}\norm{\e^{i(t-t_0)\Delta}v(t_0)}_{L_{t,x}^5\intervaloo{t_0-\delta}{T^*+\delta}}\leq \epsilon_0\,,
\end{align}
where $\epsilon_0=\epsilon_0(E)$ is as in the local well-posedness Proposition~\ref{prop:lwp}. To prove the above estimates, we proceed as follows. First, we use the global assumption~\eqref{eq:blow-up:ass} to decompose $J^*$ into a finite number $L(\epsilon)$ of intervals $\{I_l\}_{1\leq l\leq L}$ such that for $j\in\{1,\dots,L\}$,
\[
\norm{v}_{L_{t,x}^5(I_j)}\leq \epsilon\,.
\]
Using the assumption~\eqref{eq:ass:data} that $\norm{f}_{L_{t,x}^5(\R)}<+\infty$, and up to an extra decomposition of $J^*$ into $\bigo{L}$ intervals, we may assume that for each $l\in\set{1,\dots,L}$, 
\[
F(J_l)\leq \epsilon\,.
\]
Next, we take $\widetilde{J}\Subset J^*$ a compact sub-interval of $J^*$, and denoting $\widetilde{J}\cap J_l$ by $\widetilde{J}_l$, we prove by induction on $l\leq L$ that there exists $C(\norm{f_0}_{H^s})$ such that 
\begin{equation}
    \norm{v}_{X^\sigma(\widetilde{J}_l)}\leq C2^lE\,.
\end{equation}
For this purpose, we apply the trilinear estimates~\eqref{eq:trilinear-error} to get
\[
\norm{v}_{X^\sigma\widetilde{J}_1)}\leq \norm{v(0)}_{H^\sigma} + C\epsilon +2C\epsilon\norm{v}_{X^\sigma(\widetilde{J}_1)}\,.
\]
Hence, choosing $\epsilon$ small enough we see that 
\[
\norm{v}_{X^\sigma(\widetilde{J}_1)}\leq 2E\,.
\]
Iterating this on $\widetilde{J}_2,\dots, \widetilde{J}_L$ with the same $\epsilon$, we obtain that 
for $l\in\set{1,\dots,L}$,  
\[
\norm{v}_{X^\sigma(\widetilde{J}_l)}\leq 2^lE\,.
\]
Since $\epsilon$ does not depend on $\widetilde{J}\Subset J^*$, we have from sub-additivity of the norm $X^\sigma$ (see Lemma~\ref{lemma:U2-continuity}) that  
\begin{equation}
    \label{eq:uniform-Xsigma}
\norm{v}_{L_t^\infty H^\sigma(J^*)} \leq \underset{\widetilde{J}\Subset J^*}{\sup}\ \norm{v}_{X^\sigma(\widetilde{J})}\leq \parent{\sum_{l=1}^L 2^l} E \leq C(L,E)\,.
\end{equation}
Similarly, the multilinear estimates from Proposition~\ref{prop:trilinear:5} yield for all $t_0\in J^*$,
\[
\norm{\mathcal{I}\parent{\intervaloo{t_0}{\cdot},\mathcal{N}(f+v)}}_{L_{t,x}^5 (\widetilde{J})}\leq \widetilde{C}(E,f_0)\,.
\]
Therefore, taking the sup over all the possible $\widetilde{J}$ we obtain a uniform bound with respect to $t_0$ on the $L_{t,x}^5(J^*)$-norm of the Duhamel integral. Consequently, we have from Duhamel's formula that 
\begin{multline*}
\norm{\e^{i(t-t_0)\Delta}v(t_0)}_{L_{t,x}^5\intervalco{t_0}{T^*}} = \norm{v-\mathcal{I}\parent{\intervaloo{t_0}{\cdot},\mathcal{N}(f+v)}}_{L_{t,x}^5\intervalco{t_0}{T^*}}\\
\leq \norm{v}_{L_{t,x}^5(J^*)} +\norm{\mathcal{I}\parent{\intervaloo{t_0}{\cdot},\mathcal{N}(f+v)}}_{L_{t,x}^5(t_0,T*)}\,,
\end{multline*}
and we conclude the proof of~\eqref{eq:blow-up2} by monotone convergence
\[
\underset{t_0\to T^*}{\lim}\norm{\e^{i(t-t_0)\Delta}v(t_0)}_{L_{t,x}^5\intervalco{t_0}{T^*}}=0\,.
\]
As explained above, this finishes the proof of the global existence. At that point, it is standard to deduce scattering from these global bounds. Indeed, when $J^*=\intervalco{0}{+\infty}$ and $\norm{v}_{L_{t,x}^5(J^*)}<+\infty$, we use the uniform bound~\eqref{eq:uniform-Xsigma}, the nonlinear estimates from Proposition~\ref{prop:trilinear} and monotone convergence to deduce from the Duhamel integral formulation that $\set{\e^{-it_n\Delta}v}_n$ is a Cauchy sequence in $H_x^\sigma$ as $t_n$ goes to $+\infty$. This ends the proof of Proposition~\ref{prop:blow-up}.
\end{proof}
\subsection{Stability theory}
\label{sec:stability}
In this subsection, we prove that global existence and scattering for $v$ can be deduced from a priori uniform estimate on the $H^\sigma$ norm of $v(t)$. Note that this is true when $v$ is solution to~\eqref{eq:nls} without forcing terms, and when $\sigma>\frac{2}{3}$ (see Theorem~\ref{theorem:deterministic-gwp}). Hence, since $v$ is actually solution to the forced equation~\eqref{eq:nls}, we will settle a stability theory at regularity $H^\sigma$. In particular, on some spacetime slabs where the perturbation satisfies a smallness condition, we shall be able to approach $v$ by a solution to~\eqref{eq:nls} and to infer local spacetime bound on $v$. Then, provided we have a global a priori estimate for the $H^\sigma$ norm of $v$, we perform a bootstrap argument to extend the spacetime bound to the whole maximal lifespan of $v$, and to deduce conditional scattering. 
For now on, we fix $\sigma>2/3$, which corresponds to the lowest regularity where the global Cauchy theory for~\eqref{eq:nls} is known.
\begin{lemma}[Short-time stability in $H_x^\sigma$]
\label{lemma:short-stability}
Let $t_0\in\R$, $E\geq 0$ and $v_0\in H_x^\sigma$ with $\norm{v_0}_{H_x^\sigma}\leq E$. Let $v$ be the local solution to~\eqref{eq:nls:f} as in Proposition~\ref{prop:lwp} associated with $v(t_0)=v_0$. Then, take $u_0\in H_x^\sigma$ and denote by $u$ the global solution to~\eqref{eq:nls} in $C^\infty(\R;H_x^\sigma)$, with data $u(t_0)=u_0$. There exists $\epsilon_1=\epsilon_1(E)$ and $C_0(E)$ such that if 
\begin{equation}
\label{eq:ass-small-short-time-stability}
\norm{v_0-u_0}_{H_x^\sigma}\leq \epsilon\,,\quad F^\omega(J)\leq \epsilon\,,\quad \norm{u}_{L_{t,x}^5(J)}\leq \epsilon\,,
\end{equation}
for some $\epsilon\leq\epsilon_1$, then $v$ stays close to $u$ on any compact interval $J\Subset J^*$ that contains $t_0$
\begin{equation}
    \label{eq:short-stability}
    \norm{v-u}_{X^\sigma(J)}\leq C_0(E)\epsilon\,.
\end{equation}
\end{lemma}
\begin{proof}
Let us consider the difference $w=v-u$~\footnote{\ Note that $u$, hence $w$, is in $X^\sigma(J)$. This follows from the uniqueness part of Theorem~\ref{theorem:deterministic-gwp} combined with a contraction mapping argument in $X^{\sigma}(J)$. Moreover, $\norm{u}_{X^\sigma(J)}\leq CE$.},  solution to the equation
\begin{equation}
    \label{eq:eq-stability-w}
    \begin{cases}
    i\partial_tw+\Delta w = \mathcal{N}(v+f)-\mathcal{N}(v)+\mathcal{N}(v)-\mathcal{N}(u)\,,\quad  (t,x)\in\R\times\R^3\,.\\
    w_{t=0} = v_0-u_0\,,
    \end{cases}
\end{equation}
We introduce the function $g: t\in J\mapsto \norm{w}_{X^\sigma(t_0,t)}$, and we perform a continuity argument. Note that as a consequence of the smallness condition $\norm{w(0)}_{H_x^\sigma}\leq\epsilon$ and from Lemma~\ref{lemma:U2-continuity}, we have that $g$ is continuous and $\underset{t\to t_0}{\limsup}\ g(t)\lesssim\epsilon$. Moreover, it follows from Duhamel's formulation that
\[
\norm{w}_{X^\sigma(J)}\leq C\norm{v_0-u_0}_{H_x^\sigma}+\norm{\mathcal{I}\parent{\cdot,\mathcal{N}(v+f)-\mathcal{N}(v)}}_{X^\sigma(J)}+\norm{\mathcal{I}\parent{\cdot,\mathcal{N}(u)-\mathcal{N}(v)}}_{X^\sigma(J)}\,.
\]
From the trilinear estimate~\eqref{eq:trilinear-error} of Proposition~\ref{prop:trilinear} and the smallness assumptions, we get
\begin{equation*}
\begin{split}
\norm{\mathcal{I}\parent{\cdot,\mathcal{N}(v+f)-\mathcal{N}(v)}}_{X^\sigma\intervaloo{t_0}{t}}&\leq
CF^\omega(J)+C\norm{v}_{X^\sigma\intervaloo{t_0}{t}}\parent{F^\omega(J)+\norm{v}_{L_{t,x}^5(\intervaloo{t_0}t{})}F^\omega(J)^\frac{1}{2}}\\
&\hspace{-38pt}\leq C\epsilon + C(\norm{u}_{X^\sigma\intervaloo{t_0}{t}}+\norm{w}_{X^\sigma(\intervaloo{t_0}{t})})\parent{\epsilon+\epsilon^\frac{1}{2}\norm{w}_{L_{t,x}^5\intervaloo{t_0}{t}}+\epsilon^\frac{1}{2}\norm{u}_{L_{t,x}^5(J)}}\\
&\hspace{-38pt}\leq C\epsilon + C\epsilon^\frac{1}{2}\parent{CE+g(t)}\parent{2\epsilon+Cg(t)}\,.
\end{split}
\end{equation*}
Moreover, 
\begin{multline*}\norm{\mathcal{I}\parent{\cdot,\mathcal{N}(u)-\mathcal{N}(v)}}_{X^\sigma\intervaloo{t_0}{t}}\leq C\norm{w}_{X^\sigma(\intervaloo{t_0}{t})}\parent{\norm{v}_{L_{t,x}^5(\intervaloo{t_0}{t})}^2+\norm{u}_{L_{t,x}^5(J)}^2}\\
\leq C\norm{w}_{X^\sigma\intervaloo{t_0}{t}}\parent{\norm{w}_{L_{t,x}^5\intervaloo{t_0}{t}}^2+\norm{u}_{L_{t,x}^5(J)}^2} 
\leq Cg(t)\parent{g(t)^2+\epsilon^2})\,.
\end{multline*}
Hence, by combining the above estimates, we prove that for all $t\in J$ with $t_0\leq t$,  
\[g(t)\leq 2C\epsilon+2C\epsilon^\frac{3}{2}E+g(t)\parent{\epsilon^\frac{1}{2}CE+\epsilon^\frac{3}{2}+g(t)^2+\epsilon^2}\,,\]
and the result follows from a continuity argument, by choosing $\epsilon_1\ll E^{-\frac{1}{2}}$.
\end{proof}
Next, we combine the local well-posedness result with the blow-up criterion to turn the short-time stability result into a long time stability statement. 
\begin{lemma}[Long-time stability]
\label{lemma:long:stability}
Let $\Lambda>0$, $t_0\in\R$, $J=\intervalco{t_0}{T}$, $v_0\in H_x^\sigma$ with $\norm{v_0}_{H_x^\sigma}\leq E$. Let $u$ be the global-in-time solution to~\eqref{eq:nls} in $H_x^\sigma$ with data $u(t_0)=v_0$, and assume that 
\[
\norm{u}_{L_{t,x}^5(J)}\leq \Lambda\,.
\]
There exists $\epsilon_2=\epsilon_2(E,\Lambda)>0$ such that under the smallness condition
\[
F^\omega(J)\leq \epsilon_2\,,
\]
there exists a unique solution $v$ to~\eqref{eq:nls:f} and a constant $C(E,\Lambda)>0$ such that for every compact interval  $\widetilde{J}\Subset J$, we have $v\in X^\sigma(\widetilde{J})$, and
\begin{equation}
\label{eq:long-time-stability}
\norm{v}_{L_{t,x}^5(J)}\leq C(E,\Lambda)\,.
\end{equation}
\end{lemma}
\begin{proof}
Let $v$ be the maximal lifespan solution to~\eqref{eq:nls:f} with $v(t_0)=v_0$, on an interval $J^*\times\R^3$. The whole statement reduces to the estimate~\eqref{eq:long-time-stability}. Indeed, we would be in position to apply the blow-up criterion, and to deduce that $J\subset J^*$. The strategy to prove such a global estimate is to break $J$ into a finite number $L(E,\Lambda)$ of intervals where the $L_{t,x}^5$ norm of $u$ is small, so that we can apply the short-time stability Lemma~\ref{lemma:short-stability}. Then, we sum over the different spacetime slabs. Since their number does only depend on $E$ and $\Lambda$, we obtain~\eqref{eq:long-time-stability}.

We proceed as follows. First, we use the deterministic global well-posedness Theorem~\ref{theorem:deterministic-gwp} and we get a constant $C_0(E)$, such that for any $\widetilde{J}\Subset J^*$, we have
\[
\norm{u}_{X^\sigma(\widetilde{J})}\leq C_0(E)\,.
\]
In addition, we take $\epsilon_1=\epsilon_1(2C_0(E))$ as in Lemma~\ref{lemma:short-stability}, and $\epsilon_2<\epsilon_1$, to be determined shortly. Next, we decompose $J$ into $L(\epsilon,\Lambda,E)$ intervals $\{J_i\}_{1\leq i\leq L}$, with $J_i=\intervalco{t_i}{t_{i+1}}$, such that for $i\in\{0,\dots, L-1\}$ and for some $\epsilon<\epsilon_2$,
\[
\norm{u}_{L_{t,x}^5(J_i)}\leq \epsilon\,. 
\]
Now, we write $w=v-u$, we fix $\widetilde{J}^*\Subset J^*$ and we denote $\widetilde{J}_i=\widetilde{J}^*\cap J_i$. We would like to apply Lemma~\ref{lemma:short-stability} on each $\widetilde{J}_i$. To this end, we need to make sure that for all $i\in\set{1,\dots,L-1}$, we have
\[
\norm{v(t_i)}_{H_x^\sigma}\leq 2C_0(E)\,,\quad \norm{w(t_i)}_{H^\sigma}\leq \epsilon_1(2C_0(E))\,.
\]
Indeed, we already know by assumption and by construction that
\[
F^\omega(\widetilde{J}_i)\leq \epsilon_1\,,\quad \norm{u}_{L_{t,x}^5(\widetilde{J}_i)}\leq \epsilon_1\,.
\]
Hence, we shall prove by induction on $i$ that for some  $\epsilon_2<\epsilon_1$, for every $\epsilon<\epsilon_2$ and for every $i\in\set{0,\dots,L_1}$, there exists $C_i=C_i(E,\Lambda)$ such that 
\begin{equation}
\label{eq:induction:long-time-stability}
\norm{w}_{X^\sigma(\widetilde{J}_i)}\leq C_i\epsilon\,,\quad (\underset{1\leq i\leq L}{\sup} C_i)\epsilon_2\leq\epsilon_1\,,\quad \norm{v(t_i)}_{H_x^\sigma}\leq 2C_0(E)\,.
\end{equation}
In the case when $i=0$, we have $w(t_0)=0$ and it follows from the short-time stability Lemma~\ref{lemma:short-stability} applied on $J_1$, where $v$ and $f$ satisfy the smallness condition~\eqref{eq:ass-small-short-time-stability}, that there exists $C_1(E)$ with 
\[
\norm{w}_{X^\sigma(\widetilde{J}_1)}\leq C_1(E)\epsilon\,.
\]
Hence, we chose $\epsilon_2$ such that $C_1(E)\epsilon_2<\epsilon_1(2C_0(E))$. Next, we assume that~\eqref{eq:induction:long-time-stability} holds up to time $t_i$. By Duhamel's integral formulation, and since $w(t_0)=0$, we have 
\[
\norm{w(t_i)}_{H_x^\sigma} \leq \norm{\mathcal{I}\parent{t_0,t_i,\mathcal{N}(v+f)-\mathcal{N}(u)}}_{H_x^\sigma}\,.
\]
Then, we proceed as in the proof of the short-time stability Lemma~\ref{lemma:short-stability}, and we use the multilinear estimates from Propositions~\ref{prop:trilinear},~\ref{prop:trilinear:5}, as well as the induction assumption, to see that 
\[
\norm{\mathcal{I}\parent{t_0,t_i,\mathcal{N}(v+f)-\mathcal{N}(u)}}_{H_x^\sigma}
\leq C\sum_{l=0}^{i-1}\norm{\mathcal{I}\parent{\cdot,\mathcal{N}(v+f)-\mathcal{N}(u)}}_{X^\sigma(\widetilde{J}_l)}\leq C \sum_{l=0}^{i-1}C_l(E)\epsilon\,.
\]
We define $C_i=C(E)C\sum_{1\leq l\leq i-1}C_l$ and choose $\epsilon_2$ such that $C_{i}\epsilon_2\leq\epsilon_1(2C_0(E))$. Subsequently, the smallness condition ~\eqref{eq:ass-small-short-time-stability} is satisfied on the interval $J_{i}=\intervalco{t_i}{t_{i+1}}$, so that we can apply Lemma~\ref{lemma:short-stability} and obtain
\[
\norm{w}_{X^\sigma(\widetilde{J}_i)}\leq C(E)\sum_{0\leq l\leq i-1}C_j\epsilon\,.
\]
In addition, we get from the triangle inequality that 
\[
\norm{v(t_{i+1})}_{H_x^\sigma}\leq \norm{u}_{L^\infty(J_{i+1};H_x^\sigma)}+\norm{w}_{L^\infty(J_{i+1};H_x^\sigma)}\leq C_0(E)+C_{i+1}\epsilon\leq 2C_0(E)\,.
\]
This finishes the proof of the induction result~\eqref{eq:induction:long-time-stability}. Moreover, it follows from the sub-additivity of the $X^\sigma$ norm (see Lemma~\ref{lemma:U2-continuity}) that 
\[
\norm{v}_{X^\sigma(\widetilde{J})}\leq \norm{u}_{X^\sigma(J)}+\sum_{j=1}^L\norm{w}_{X^\sigma(J_j)}\leq C(E,\Xi)\,.
\]
The proof of~\eqref{eq:long-time-stability} follows by using the embedding $X^\sigma(\widetilde{J})\xhookrightarrow{} L_{t,x}^5(\widetilde{J})$, and the monotone convergence theorem which implies that $\norm{v}_{L_{t,x}^5(J)}=\underset{\widetilde{J}\Subset J}{\sup}\norm{v}_{L_{t,x}^5(\widetilde{J})}$.
\end{proof}
\begin{proposition}[Uniform bound in $H_x^\sigma$ implies scattering]
\label{prop:uniform:scattering}
Let $v_0\in H_x^\sigma$, and let $v$ be the maximal-lifespan solution in $H_x^\sigma$ to~\eqref{eq:nls:f} from initial data $v(0)=v_0$. Suppose that we have the uniform a priori bound
\[
\underset{t\in J^*}{\sup}\norm{v(t)}_{H_x^\sigma}\leq E\,.
\]
Then, the solution $v$ is global and scatters as $t$ goes to $\infty$.
\end{proposition}
\begin{proof}
We get from the blow-up criterion of Lemma~\ref{prop:blow-up} that it suffices to prove the global spacetime bound 
\begin{equation}
\label{eq:goal}
\norm{v}_{L_{t,x}^5(J^*)}<+\infty\,.
\end{equation}
Recall that the \textit{deterministic} Cauchy theory in $H^\sigma$, for $2/3<\sigma\leq1$, provides such a bound. Indeed, Theorem~\ref{theorem:deterministic-gwp} claims that there exists $\Lambda(E)$ such that for any solution $u$ to~\eqref{eq:nls} starting from a data $u(t_0)$ with $\norm{u(t_0)}_{H_x^\sigma}\leq E$, we have
\[
\norm{u}_{L_{t,x}^5(\R)}\leq \Lambda(E)\,.
\]
Subsequently, we fix $\epsilon_2(E,\Lambda(E))=\epsilon_2(E)$ as in Lemma~\ref{lemma:long:stability}, and we divide $J^*$ into a finite number $L(E)$ of intervals $J_i=\intervalco{t_i}{t_{i+1}}$, with $t_0=0$, such that 
\[
F^\omega(J_i)\leq \epsilon_1,\quad j\in\{1,\dots,L\}\,.
\]
Next, for each $i\in\set{0,\dots,L-1}$, we consider the global solution $u$ to~\eqref{eq:nls} with initial data $u(t_i)=v(t_i)$. Since $\norm{v(t_i)}\leq E$ by asumption, we have
\[
\norm{v}_{}\leq \Lambda(E)\,.
\]
Hence, we can apply Lemma~\ref{lemma:long:stability} on $J_i$, and obtain
\[
\norm{v}_{L_{t,x}^5(J_i)}\leq C_i(E)\,.
\]
By summing over the spacetime slabs, we conclude that 
\[
\norm{v}_{L_{t,x}^5(J^*)}\leq C(E,\Lambda(E))\leq C(E)\,.
\]
This finishes the proof of Proposition~\ref{prop:uniform:scattering}.
\end{proof}
\section{Almost conservation laws}
In this section, we prove some almost conservation laws in order to obtain the uniform bound on the $H^\sigma$ norm of $v$. Recall that throughout this section, $\epsilon>0$ and $\omega\in\widetilde{\Omega}_\alpha$ are fixed, such that for $N$ large enough, 
\[
F^\omega(\R)+F_\infty^\omega(\R)\leq C_\alpha,\quad F_2^\omega(\R)\leq C_\alpha N^{\frac{1-\sigma}{2}-\gamma}\,.
\]
Since there is no coercive conservation law at the level of $H^\sigma$ for~\eqref{eq:nls}, we need to use a modified energy. More precisely, we consider the energy of $Iv$, which is finite but not preserved, in order to damp the frequencies of size larger than $N$. However, we expect the time-derivative of the modified energy to be small when $N$ is large, at least on a spacetime slab where there hold some smallness conditions for the $L_{t,x}^4$ norm of $v$ and spacetime bounds for the stochastic forcing term. 
\label{sec:conservation}
\subsection{Setting up the I-method}
Let $v\in H_x^\sigma(\R^3)$ be the maximal lifespan solution to~\eqref{eq:nls:f} on $\intervalco{0}{T^*}$. There is no difficulty to estimate the low frequencies of $v$ since the $L^2$-norm is almost-conserved, and the forcing term is uniformly in $L^2$.
\begin{lemma}[Almost conservation of the mass (see Lemma 7.1 in~\cite{oh-okamoto-19})]
\label{lemma:mass}
Let $v$ be solution to~\eqref{eq:nls:f} on a time interval $J$. Then, 
\[
\int_{\R^3}\abs{v(t,x)}^2dx\lesssim\int_{\R^3}\abs{v_0(x)}^2dx+\int_{\R^3}\abs{f_0(x)}^2dx\,.
\]
\end{lemma}
\begin{proof}
The proof is a straightforward consequence of the conservation of mass for~\eqref{eq:nls} and the unitarity of the linear Schrödinger propagator $\e^{it\Delta}$, from which we deduce that 
\[
\int_{\R^3}\abs{v(t)}^2dx=-2\re\parent{\int v(t)\overline{f^\omega(t)}dx}\leq \frac{1}{2}\int\abs{v(t)}^2dx + 2\int\abs{f^\omega(t)}^2dx\,.
\]
We refer to Lemma 7.1 in~\cite{oh-okamoto-19} for further details.
\end{proof}
To obtain some estimates on the high frequencies of $v$, which has infinite energy, we need to smooth it by applying the I-operator, and by considering $Iv$, which has finite energy, or sometimes $Iu=I(v+f)$. The aim is to prove the uniform bound 
\begin{equation}
\label{eq:ue}
    \underset{t\in\intervalco{0}{T^*}}{\sup}\norm{Iv(t)}_{H^1(\R^3)}<+\infty\,.
\end{equation}
Indeed, we deduce from this that 
\begin{equation}
\label{eq:ue-sigma}
    \underset{t\in\intervalco{0}{T^*}}{\sup}\norm{v(t)}_{H_x^\sigma(\R^3)}<+\infty\,.
\end{equation}
With this uniform estimate~\eqref{eq:ue-sigma} on $\norm{v}_{H_x^\sigma}$ at hand, we will be in position to apply Proposition~\ref{prop:uniform:scattering} to show that the solution $v$ to~\eqref{eq:nls:f} is global-in-time and scatters. Now, observe that the truncated solution $Iu$, as well as $Iv$, satisfies the perturbed cubic Schrödinger equation
\begin{equation}
    \tag{$I-$NLS}
    \label{eq:nls:I}
    i\partial_tIu+\Delta Iu = \mathcal{N}(Iu) + H\,,
    \quad (t,x)\in\R\times\R^3\,,
\end{equation}
where we denoted by $H$ the commutator
\begin{equation}
\label{eq:def:H}
    H \coloneqq I\mathcal{N}(u)-\mathcal{N}(Iu)\,.
\end{equation}
This perturbation $H$ can be seen as a trilinear operator acting on functions $u^{(j)}$ of type $v$ or $f$, and some appropriate spacetime norms of $H$ enjoy some decay with respect to $N$ thanks to frequency cancellations. These cancellations come from the gauge-invariant structure of the equation, and from the definition of the I-operator.
\begin{definition}
Given a time-interval $J=\intervaloo{t_0}{b}\subset\intervalco{0}{T^*}$, we define the  Strichartz norm of a frequency truncated function $v$ by
\[
Z_{I}(J) \coloneqq\norm{Iv}_{S^1(J)} =\sup_{s(q,r)=0}\norm{\japbrak{\nabla}Iv}_{L_t^q(J;L_x^r(\R^3))}\,,
\]
where we recall that $s(q,r)=\frac{2}{q}-d\parent{\frac{1}{2}-\frac{1}{r}}$ encodes the admissibility condition for the Strichartz estimate to hold.
\end{definition}
We shall prove that $Z_I(J)$ is of size $\sim N^{1-\sigma}$ provided $E(Iv(t_0))\leq N^{2(1-\sigma)}$ and $\norm{Iv}_{L_{t,x}^4(J)}^4\ll 1$. This is the matter of Lemma~\ref{lemma:Z}. Before this, let us state and prove the main result of this section, where we estimate the spacetime Lebesgue norm of $wH$ for a given function $w$ that lies uniformly in $L_x^2$.
\begin{lemma}
\label{lemma:H}
Let $J\subset\intervalco{0}{T^*}$ such that 
$Z_I(J)<+\infty$, and let $w\in L_t^\infty(J;L_x^2(\R^3))$. For all $\delta>0$, there exists a constant $C$ such that we have
\begin{multline}
\label{eq:lemma:H}
\int_J\absbig{\int_{\R^3}w(t,x)\overline{H}(t,x)dx}dt\leq CN^{-1+\delta}\underset{K\in2^\N}{\sup}\parent{K^{-1}\norm{P_Kw}_{L_t^\infty L_x^2(J)}}\\ \parent{Z_I(J)^2+F_2^\omega(J)2}\parent{Z_I(J)+F_\infty^\omega(J)}\,.
\end{multline}
\end{lemma}
The idea is to use the frequency cancellations from the $I$ operator in the error term $H$ to gain a negative power of $N$. As we will see, we shall often come across terms as in the left-hand-side of~\eqref{eq:lemma:H}, especially when we estimate the modified energy and the modified Morawetz interaction. Hence,  obtaining estimates like~\eqref{eq:lemma:H} turns out to be the crucial part of our analysis.
\begin{proof}
The strategy of the proof follows the sames lines as the one of Proposition 3 from~\cite{ckstt-2004}, but with some differences. At a fixed time $t\in J$, we use Plancherel's formula, and we develop the nonlinearity in Fourier. We obtain
\begin{equation}
\label{eq:IH}
   \int_{\R^3}w(t,x)\overline{H}(t,x)dx =\int_{\Sigma\xi_i=0}\parentbig{1-\frac{m(\xi_2+\xi_3+\xi_4)}{m(\xi_2)m(\xi_3)m(\xi_4)}}\widehat{w}(t,\xi_1)\widehat{\overline{Iu}}(t,\xi_2)\widehat{Iu}(t,\xi_3)\widehat{\overline{Iu}}(t,\xi_4)\,.
\end{equation}
We perform a Littlewood-Paley decomposition of each term appearing in the above integral, so that each term in the above multilinear integral is localized around a dyadic frequency $N_i$, and we denote by $B$ the point wise bound of the Fourier multiplier 
\[
\underset{\xi_i\sim N_i}{\sup}\absbig{1-\frac{m(\xi_2+\xi_3+\xi_4)}{m(\xi_2)m(\xi_3)m(\xi_4)}}\leq B(N_1,N_2,N_3,N_4)\,.
\]
Next, we factorize by the point wise bound $B$ benefit from the cancellations of the Fourier multiplier appearing in~\eqref{eq:IH}, and we are left to estimate
\[
B(N_1,N_2,N_3,N_4)\absbig{\int_{\R^3}\widehat{\Lambda(w,Iu,Iu,Iu)}(\xi)\widehat{\overline{Iu}}(\xi)d\xi}\,,
\]
for a given multilinear operator $\Lambda$ as in equality (3.22) from~\cite{ckstt-2004}. To estimate such a multilinear integral, we apply Plancherel, Hölder's inequality and the Coifman-Meyer estimate (see Page 179 in~\cite{coifman-meyer-78}). We also integrate in time and use Hölder's inequality in time, to see that the left-hand-side of~\eqref{eq:lemma:H} is bounded from above by 
\[
\sum_{N_i}B(N_1,N_2,N_3,N_4)\norm{P_{N_1}w}_{L_t^\infty L_x^2(J)}\norm{P_{N_2}Iu}_{L_t^2L_x^6(J)}\norm{P_{N_3}Iu}_{L_t^2L_x^6(J)}\norm{P_{N_4}Iu}_{L_t^\infty L_x^6(J)}\,.
\]
It turns out from the case by case analysis conducted in the proof of Proposition 3.1 in~\cite{ckstt-2004} that for all $\delta>0$, there exits $C>0$ such that for all $N_1,N_2,N_3,N_4,N\in2^\N$ we have
\[
B(N_1,N_2,N_3,N_4)N^{-1+\delta}\frac{N_1}{N_2N_3}\leq CN^{-1+\delta}\max(N_1,N_2,N_3,N_4)^{-\delta}\,.
\]
Hence, we are able to sum over the $N_i$'s and to gain a power of $N$. We have therefore
\begin{multline}
\label{eq:proof:H:lemma}
\int_J\absbig{\int_{\R^3}w(t,x)\overline{H}(t,x)dx}dt \\
\leq C\parent{N_1^{-1}\norm{P_{N_1}w}_{L_t^\infty L_x^2}}\norm{\nabla Iu}_{L_t^2L_x^6(J)}\norm{\nabla Iu}_{L_t^2L_x^6(J)}\norm{Iu}_{L_t^\infty L_x^6(J)}\,.
\end{multline}
It remains to estimate the contributions of the terms on the right hand side in~\eqref{eq:proof:H:lemma}. Each term $Iu$ is decomposed into $Iu=If+Iv$. To estimate the terms $Iv$ in $L_t^2L_x^6$, we use the endpoint Strichartz estimate, whereas We need Sobolev embedding to estimate $Iv$ in $L_t^\infty L_x^6$. As for the terms $If$, we only use the definitions from Section~\ref{section:random}. This concludes the proof of Lemma~\ref{lemma:H}.
\end{proof}
Let us now address the a priori estimate for $Z_I(J)$ on a spacetime slab, where we assume a bound for the modified energy of $v$, as well as a smallness condition on spacetime norms of $Iv$ and $f$. More precisely, provided that the energy of $Iv$ is initially of size $N^{2(1-\sigma)}$, we prove that $Z_I(J)$ is of size $\sim N^{1-\sigma}$ on an interval $J$ where there holds~\eqref{eq:ass:4} and~\eqref{eq:ass:4f}.
\begin{lemma}[Modified local existence theory]
\label{lemma:Z}
Let $t_0\in J^*$ be such that 
\[
\mathcal{E}(v)(t_0)\leq N^{2(1-\sigma)}\,.
\]
There  $C>0$ and $\eta>0$ that do not depend on $N$, $\epsilon_N= N^{-2(1-\sigma)-\delta}$ for some arbitrarily small $\delta$, and $N_0=N_0(C_\alpha,f_0)$ such that for any interval $J=\intervalco{t_0}{b}\subset\intervalco{0}{T^*}$ on which the following smallness conditions~\footnote{Note that $(4,3)$ is an admissible pair. Hence, $\norm{f}_{L_t^4L_x^3(\R)}\leq C\norm{f_0}_{L_x^2}$ for every $\omega$.} are satisfied
\begin{align}
    \label{eq:ass:4}
    \norm{Iv}_{L_{t,x}^4(J)}^4&\leq \epsilon_N\,,\\
    \label{eq:ass:4f}
    \norm{f}_{L_t^4L_x^3(J)}&\leq\eta\,,
\end{align}
we have that for all $N\geq N_0$,
\begin{equation}
\label{eq:Z:apriori}
Z_I(J)\leq CN^{1-\sigma}\,.
\end{equation}
\end{lemma}
For now on, we assume that $N>N_0$, such that~\eqref{eq:Z:apriori} holds whenever the assumptions of Lemma~\ref{lemma:Z} are satisfied. 
\begin{proof}
Recall that $Iv$ is solution to~\eqref{eq:nls:I}. By the Duhamel's integral formulation for equation~\eqref{eq:nls:I}, we have that
\begin{equation*}
\begin{split}
Iv(t) &= \e^{i(t_0)\Delta}Iv(t_0)-\int_{t_0}^t\e^{i(t-\tau)\Delta}\parent{\mathcal{N}(Iu)+H}d\tau \\
&= \e^{i(t_0)\Delta}Iv(t_0) + \mathcal{I}\parent{\intervaloo{t_0}{t};\mathcal{N}_1}+\mathcal{I}\parent{\intervaloo{t_0}{t};\mathcal{N}_1}+\mathcal{I}\parent{\intervaloo{t_0}{t};H} \,,
\end{split}
\end{equation*}
where~\footnote{\ We also have similar terms, but the analysis does not depend on the complex conjugates at this stage.}
\[
\mathcal{N}_1 \coloneqq \mathcal{N}(v,v,v)+\mathcal{N}(f,f,f)+\mathcal{N}(f,v,v)\,,\quad \mathcal{N}_2=\mathcal{N}(f,f,v)\,.
\]
It follows from the dual Strichartz estimate of Proposition~\ref{proposition:strichartz} that
\[
    Z_I(J)\leq C\parentbig{ \norm{Iv(t_0)}_{H^1}+\underbrace{\norm{\japbrak{\nabla}I\mathcal{N}_1}_{L_t^{6/5}L_x^{18/11}(J)}}_{I_1}+\underbrace{\norm{\japbrak{\nabla}I\mathcal{N}_2}_{L_t^1L_x^2(J)}}_{I_2}+\underbrace{\norm{\mathcal{I}(\intervaloo{t_0}{t},\japbrak{\nabla}IH)}_{U_\Delta^2(J)}}_{I_3}}\,.
\]
\textbf{Term} $\operatorname{I_1}$: We deduce from the Leibniz rule that
\[\norm{\japbrak{\nabla}I(\mathcal{N}(Iv))}_{L_t^{6/5}L_x^{18/11}}\leq \norm{\japbrak{\nabla}Iv}_{L_t^2L_x^6}\norm{Iv}_{L_t^6L_x^{9/2}}^2\,.\]
By interpolation and Sobolev embedding, we have
\[\norm{Iv}_{L_t^6L_x^{9/2}}\leq\norm{Iv}_{L_{t,x}^4(J)}^\frac{2}{3}\norm{Iv}_{L_t^\infty L_x^6(J)}^\frac{1}{3}\leq C\epsilon_N^\frac{1}{6}Z_I(J)^\frac{1}{3}\,.\]
Consequently, the contribution of this term is less than
\[\norm{\japbrak{\nabla}I(\mathcal{N}(Iv))}_{L_t^{6/5}L_x^{18/11}} \leq C\epsilon_N^\frac{1}{3}Z_I(J)^{1+\frac{2}{3}}\,.\]
Similarly, we have
\begin{equation*}
\begin{split}
\norm{\japbrak{\nabla}I(\mathcal{N}(Iv,Iv,If))}_{L_t^{6/5}L_x^{18/11}}& \\
&\leq \norm{\japbrak{\nabla}If}_{L_t^2L_x^6}\norm{Iv}_{L_t^6L_x^{9/2}}^2+\norm{\japbrak{\nabla}Iv}_{L_t^2L_x^6}\norm{Iv}_{L_t^6L_x^{9/2}}\norm{If}_{L_t^6L_x^{9/2}} \\
&\leq CZ_I(J)\parent{F_2(J)\epsilon_N^\frac{1}{3}Z_I(J)^{-\frac{1}{3}}+\epsilon_N^\frac{1}{6}Z_I(J)^\frac{1}{3}\norm{If}_{L_{t,x}^4}^{2/3}\norm{If}_{L_t^\infty L_x^6}^\frac{1}{3}} \\
&\leq CZ_I(J)\parent{F_2(J)\epsilon_N^\frac{1}{3}Z_I(J)^{-\frac{1}{3}}+\epsilon_N^\frac{1}{6}Z_I(J)^\frac{1}{3}{F_\infty^\omega(J)}^\frac{1}{3}{F^\omega(J)}^\frac{2}{3}}\,.
\end{split}
\end{equation*}
Additionally, 
\[
\norm{\japbrak{\nabla}(If)(If)^2}_{L_t^{6/5}L_x^{18/11}}\leq \norm{\japbrak{\nabla}If}_{L_t^2L_x^6}\norm{If}_{L_t^6L_x^{9/2}}^2\leq C\norm{f_0}_{L_x^2}^2F_2(J)\,.
\]
Hence, the contribution of the first term is less than 
\[
\operatorname{I}_1 \leq CZ_I(J)\parent{F_2(J)Z_I(J)^{-1}+\epsilon_N^\frac{1}{3}Z_I(J)^\frac{2}{3}+\epsilon_N^\frac{1}{3} N^\frac{1-\sigma}{2}Z_I(J)^{-\frac{1}{3}}C_\alpha+\epsilon_N^\frac{1}{6}Z_I(J)^\frac{1}{3}C_\alpha}\,.
\]
\item\textbf{Term} $\operatorname{I_2}$: Once again, when a derivative falls on the term $If$, we need to place it inside $L_t^2$, and use the improved Sobolev embedding for randomized radial initial data. It follows from the Leibniz rule that 
\[
\norm{\japbrak{\nabla}I\mathcal{N}_2}_{L_t^1L_x^2}\leq C\norm{(\japbrak{\nabla}If)(If)(Iv) }_{L_t^1L_x^2(J)}+\norm{(\japbrak{\nabla}Iv)(If)^2 }_{L_t^1L_x^2(J)}\,.
\]
We use Hölder's inequality, Sobolev embedding and the assumption that $s\geq 3/7$ to get
\[
\norm{(\japbrak{\nabla}If)(If)(Iv)}_{L_t^1L_x^2(J)}\leq C\norm{\japbrak{\nabla}If}_{L_t^2L_x^6(J)}\norm{If}_{L_tL_x^{12}(J)}\norm{Iv}_{L_{t,x}^4(J)}
\leq CC_\alpha^2\epsilon_N^\frac{1}{4}N^\frac{1-\sigma}{2}\,.
\]
Finally, 
\[
\norm{\japbrak{\nabla}Iv(If)^2}_{L_t^1L_x^2(J)}\lesssim\norm{\japbrak{\nabla}Iv}_{L_t^2L_x^6(J)}\norm{If}_{L_t^4L_x^3(J)}^2
\lesssim\eta^2 Z_I(J) \,.
\]
Collecting the above estimates, we see that the contribution of this term is less than
\[
\operatorname{I}_2\leq CZ_I(J)\parent{C_\alpha^3N^{\frac{1-\sigma}{2}}Z_I(J)^{-1}+C_\alpha^2\epsilon_N^\frac{1}{4}N^{-\frac{1-\sigma}{2}}+\eta^2}\,.
\]
\textbf{Term} $\operatorname{I_3}$: It follows from the duality estimate in $U_\Delta^2$ as stated in Proposition~\ref{proposition:dual:U2} and from Lemma~\ref{lemma:Z} that
\begin{equation*}
\begin{split}
&\norm{\mathcal{I}((t_0,t),\japbrak{\nabla}IH)}_{U_\Delta^2}\leq C\underset{\norm{w}_{V_\Delta^2}\leq 1}{\sup}\abs{\iint_{J\times\R^3}\overline{w}(t,x)\japbrak{\nabla}IHdxdt}\\
&\leq CN^{-1+\delta}\underset{\norm{w}_{V_\Delta^2}\leq1}{\sup}\,\underset{K\in2^\N}{\sup}\,\parent{K^{-1}\norm{P_K(\japbrak{\nabla}Iw)}_{L_t^\infty L_x^2(J)}}\parent{Z_I(J)^2+F_2^\omega(J)2}\parent{Z_I(J)+F_\infty^\omega(J)}
\\
&\leq CN^{-1+\delta}\parent{Z_I(J)^2+F_2^\omega(J)2}\parent{Z_I(J)+F_\infty^\omega(J)} \,.
\end{split}
\end{equation*}
Next, we take $\gamma>0$ arbitrarilly small, and we fix 
\[
\epsilon_N = N^{-2(1-\sigma)-\gamma}\,.
\]
We have
\[
\operatorname{I}_1\leq CZ_I(J)\parentbig{N^{-\frac{2(1-\sigma)}{3}-\frac{\gamma}{3}}Z_I(J)^\frac{2}{3}+C_\alpha N^{-\frac{2(1-\sigma)}{3}-\frac{\gamma}{3}}N^\frac{1-\sigma}{2}Z_I(J)^{-\frac{1}{3}}+C_\alpha N^{-\frac{1-\sigma}{3}-\frac{\gamma}{3}}Z_I(J)^\frac{1}{3}}
\,.\]
If $Z_I(J)\sim N^{1-\sigma}$, then the term in the parenthesis on the right-hand-side of the above estimate is smaller than a negative power of $N$. Moreover, we have 
\[
\underset{\abs{J}\to 0}{\lim}Z_I(J)=0\,.
\]
Hence, we shall be able to apply a continuity argument in this case, and prove that $Z_I(J)$ is indeed of size $\sim N^{1-\sigma}$. Similarly, we have
\[
\operatorname{I}_2\leq CZ_I(J)\parent{C_\alpha^3N^\frac{1-\sigma}{2}Z_I(J)^{-1}+C_\alpha^2Z_I(J)^{-1}+\eta^2}\,.
\]
Here again, we shall be able use the continuity method. We emphasize that $\eta$ does not have to depend on $N$. The last contribution is estimated as follows: 
\[
\operatorname{I}_3\leq CN^{-1+\delta}\parent{Z_I(J)^2+N^{1-\sigma}}\parent{Z_I(J)+C_\alpha}\,.
\]
In this case, we can see that when $Z_I(J)\sim N^{1-\sigma}$, the parenthesis in the right-hand-side of the above estimate is bounded by a negative power of $N$ provided $1/2<\sigma$. Finally, we conclude by collecting the above contributions and by using the continuity method for $1/2<\sigma$, $N>N_0$ large enough and $\eta$ small enough. We stress out that $\eta$ and $N_0$ only depend on some universal constants and on $C_\alpha$, which is fixed throuhout the analysis. This concludes the proof of Lemma~\ref{lemma:Z}.
\end{proof}
\subsection{Modified energy increment}
In this section, we revisit in a probabilistic setting the analysis of modified-energy increment performed when using the I-method (see for instance~\cite{dodson-13,dodson19,ckstt-2004}). A similar strategy was developed in~\cite{fan-mendelson-20}. Here, we propose a different approach that strongly relies on the structure of the nonlinearity. Subsequently, we follow~\cite{oh-okamoto-19} and we include the probabilistic forcing temr in the potential part of the energy
\begin{equation*}
    \mathcal{E}(v(t))\coloneqq \frac{1}{2}\int_{\R^3}\abs{\nabla Iv(t,x)}^2dx+\frac{1}{4}\int_{\R^3}\abs{Iv+If}^4dx\,.
\end{equation*}
The idea is to benefit from the frequency cancellations appearing in $H$. More precisely, when we handle the increment of the energy, we aim to come up with terms under the form $wH$, for which we already obtained some spacetime bounds in Lemma~\ref{lemma:H}. In the next Propositon, we estimate the energy increments on some intervals where the $L_{t,x}^4$ norm of $Iv$ satisfies a smallness condition. 
\begin{proposition}
\label{prop:energy:loc-incr}
Let $t_0\in J^*$, such that 
\[
\mathcal{E}(v)(t_0)\leq N^{2(1-\sigma)}\,.
\]
Then, for all interval $J=\intervalcc{t_0}{t}\subset\intervalco{0}{T^*}$ where $Iv$ satisfies the smallness conditions~\eqref{eq:ass:4} and~\eqref{eq:ass:4f}, that lead to the modified local well-posedness statement of Lemma~\ref{lemma:Z}, there exists $C(\alpha)$ such that
\begin{equation}
    \label{eq:E:loc-incr}
    \int_{J}\absbig{\frac{d}{dt}\mathcal{E}(v)}dtdx\leq C N^{0+}N^{6(1-\sigma)-1}+CN^{1-\sigma}F_2^\omega(J)2+CN^{1-\sigma}F_2(J)\norm{Iv}_{L_{t,x}^4(J)}^2\,.
\end{equation}
\end{proposition}
Observe that we gain a negative power of $N$ from the perturbative term that contains $H$.
However, as explained in the introduction, we keep  track of some sub-additive quantities when we estimate the other terms. This will be useful when we will sum over the intervals where we have the modified local well-posedness result of Lemma~\ref{lemma:Z}. Still, we are forced to lose some powers of $N$ when we estimates the terms where the gradient hits $If$ by $F_2$. This indicates that we have no hope to prevent the energy from extending $\bigo{1}$, even if we rescale the initial data at the begining of the analysis.
\begin{proof}
We deduce from~\eqref{eq:nls:I} that 
\begin{align*}
    \frac{d}{dt}E(Iv(\tau))&= \re\int_{\R^3}\partial_t(Iv)\parent{\mathcal{N}(Iu)-I\mathcal{N}(u)}dx+\re\int_{\R^3}\partial_t(If)\mathcal{N}(Iu)dx\\
    &=\re\int_{\R^3}\partial_t(Iv)Hdx+\-\re \int_{\R^3}i\Delta(If)\mathcal{N}(Iu)dx\,.
\end{align*}
Consequently,
\begin{equation}
\int_J\absbig{\frac{d}{dt}E(Iv(\tau))}d\tau\leq C \underbrace{\int_J\absbig{\int_{\R^3}\partial_t(Iv)\overline{H}dx}dt}_{\operatorname{I}}+\underbrace{\norm{\Delta (If)\mathcal{N}(Iu)}_{L_{t,x}^1(J)}}_{\operatorname{II}}\,.
\end{equation}
\paragraph{\textbf{Estimate of term I}} 
In light of Lemma~\ref{lemma:H}, it suffices to prove that
\begin{equation}
\label{eq:dtIu}
\underset{K\in2^\N}{\sup}\parent{K^{-1}\norm{(\partial_tIv}_{L_t^\infty L_x^2}}\leq N^{3(1-\sigma)}\,.
\end{equation}
Indeed, we have
\[
Z_I(J)\leq CN^{1-\sigma}\,,\quad F_\infty^\omega(\R)\leq C_\alpha\,,\quad F_2(J)\leq C_\alpha N^{\frac{1-\sigma}{2}-\gamma}\,.
\]
The first estimate on $Z_I(J)$ follows from the assumptions~\eqref{eq:ass:4},~\eqref{eq:ass:4f}, and Lemma~\ref{lemma:Z}. Subsequently, we use Lemma~\ref{lemma:H} to get that
\begin{multline*}
\int_J\absbig{\int_{\R^3} \partial_t(Iu)\overline{H} dx}dt \\
\leq C N^{-1+\delta}\underset{K\in 2^\N}{\sup}\parent{K^{-1}\norm{P_K\parent{\partial_t(Iu)}}_{L_t^\infty L_x^2(J)}}\parent{Z_I(J)^2+F_2^\omega(J)2}\parent{Z_I(J)+F_\infty^\omega(J)}\\
\leq C(C_\alpha)N^{-1+\delta}\underset{K\in 2^\N}{\sup}\parent{K^{-1}\norm{P_K\parent{\partial_t(Iu)}}_{L_t^\infty L_x^2(J)}}N^{3(1-\sigma)}\,.
\end{multline*}
Let us turn to the proof of~\eqref{eq:dtIu}, and fix $K\in2^\N$. Recall that $Iv$ is solution to
\[\partial_tIv = i\Delta Iv - iI\mathcal{N}(u)\,.\]
First, we have from the definition of $Z_I(J)$ and form~\eqref{eq:Z:apriori} that
\[\norm{P_K\Delta Iv}_{L_t^\infty L_x^2(J)}\leq CKZ_I(J)\leq CKN^{1-\sigma}\,.\]
It remains to handle the nonlinear term $ I\mathcal{N}(u)\sim I\mathcal{N}(u^{(1)},u^{(2)},u^{(3)})$ with $u^{(i)}\in\set{v,f}$. For this purpose, we shall lose some derivatives and apply the fractional Leibniz rule with the operator $\nabla I$, especially when all the three terms are of type $v$ since we only have some controls on norms of $\nabla Iv$, and not on $\nabla v$.  One way of loosing some derivatives is to use Sobolev embedding. We eventually prove that 
\begin{equation}
\label{eq:Nw}
\norm{P_KI\mathcal{N}(u)}_{L_t^\infty L_x^2(J)}\leq C K\parent{ Z_I(J)^3+F_{\infty}(J)^3}\,.
\end{equation}
In what follows, all the spacetime norms will be taken over $J\times\R^3$. First, we observe that 
\[
\norm{v}_{L_t^\infty L_x^3}\leq C \norm{\japbrak{\nabla}^{1/2}v}_{L_t^\infty L_x^2}\leq C \norm{\japbrak{\nabla}Iv}_{L_t^\infty L_x^2}\leq CZ_{I}(J)\,.
\]
This follows from the Sobolev embedding and from the definition of the I-operator (see~\eqref{eq:def:I}). Next, we perform a case by case analysis depending on the type of terms that occur in the trilinear form. 

\paragraph{\textbf{Case 1: Three random terms}} In this case, we use Hölder inequality and the probabilistic estimates~\eqref{eq:ass:data} to get 
    \[
    \norm{P_KI(f^3)}_{L_t^\infty L_x^2(J)}\leq \norm{f}_{L_t^\infty L_x^6(J)}^3\leq F_{\infty}(J)^3\leq C_\alpha^3\,.
    \]
\paragraph{\textbf{Case 2: Two random terms}} In this case, we use Sobolev embedding, Hölder's inequality and the almost conservation of the mass to obtain 
\[
\norm{P_KI(f^2v)}_{L_t^\infty L_x^2}\leq C \norm{\japbrak{\nabla}P_KI(f^2v)}_{L_t^\infty L_x^{6/5}}
\leq C K\norm{f}_{L_t^\infty L_x^6}^2\norm{v}_{L_t^\infty L_x^2}
\leq C KF_{\infty}(J)^2\norm{f_0}_{L_x^2}
\,.
\]
\paragraph{\textbf{Case 3: One random term}}  Similarly, 
\[\norm{P_KI(fv^2)}_{L_t^\infty L_x^2}\leq C \norm{\japbrak{\nabla}P_KI(fv^2)}_{L_t^\infty L_x^{6/5}}
\leq C K\norm{f}_{L_t^\infty L_x^6}\norm{v}_{L_t^\infty L_x^3}\leq C KF_{\infty}(J)Z_I(J)^2\,.\]
\paragraph{\textbf{Case 4: Three deterministic terms}} Here, we need to use the Leibniz rule for $I\japbrak{\nabla}$. We obtain
\begin{equation*}
\begin{split}
\norm{P_KI(v^3)}_{L_t^\infty L_x^2}&\leq C \norm{\japbrak{\nabla}^{5/4}P_KI(v^3)}_{L_{t,x}^{12/11}}
\leq C K\norm{\japbrak{\nabla}^{1/4}P_KI(v^3)}_{L_{t,x}^{12/11}}
\\&\leq C K\norm{\japbrak{\nabla}^{1/4}Iv}_{L_t^\infty L_x^4}\norm{v}_{L_t^\infty L_x^3}^2
\leq C K\norm{\japbrak{\nabla}Iv}_{L_t^\infty L_x^2}\norm{v}_{L_t^\infty L_x^3}^2\leq CK Z_I(J)^3\,.
\end{split}
\end{equation*}
This addresses the contribution for the term $\operatorname{I}$ in the energy increment.
\paragraph{\textbf{Estimate of term II}}
After integrating by parts and applying the Leibniz rule, we write 
\begin{equation}
\label{eq:incr-II}
\norm{\Delta (If)\mathcal{N}(Iu)}_{L_{t,x}^1(J)}\leq C\iint_{J\times\R^3}\abs{\nabla(If)\mathcal{N}\parent{\nabla Iu^{(1)},Iu^{(2)},Iu^{(3)}}}dtdx
\end{equation}
with $u^{(i)}\in\set{f,v}$ for $i\in\set{1,2,3}$. Next, we perform a case by case analysis to estimate the contribution of the different terms on the right-hand side of~\eqref{eq:incr-II}. 
\paragraph{\textbf{No random term.}} In this case, we have $u^{(i)}=v$ for $i\in\set{1,2,3}$. We apply Cauchy-Schwarz inequality to obtain
\[
~\eqref{eq:incr-II} \leq C \norm{\nabla If}_{L_t^2L_x^\infty(J)}\norm{\nabla Iv}_{L_t^\infty L_x^2(J)}\norm{Iv}_{L_{t,x}^4}^2
\leq C F_2(J)Z_I(J)\norm{Iv}_{L_{t,x}^4(J)}^2\,.
\]
To handle the other cases, we apply Cauchy-Schwarz and we use the estimate~\eqref{eq:ass:data:2}.
\paragraph{\textbf{One random term.}}
\begin{enumerate}[]
\item \textbf{Case 1:} $u^{(1)}=f,u^{(2)}=u^{(3)}=v$. It follows from the almost conservation from Lemma~\eqref{lemma:mass} that
\[
~\eqref{eq:incr-II} \leq C \norm{\nabla If}_{L_t^2L_x^\infty(J)}^2\norm{Iv}_{L_t^\infty L_x^2}^2
\leq C F_2^\omega(J)2\,.
\]
\item\textbf{Case 2:} $u^{(1)}=u^{(2)}=v,u^{(3)}=f$. Similarly,
\[
~\eqref{eq:incr-II}\leq C \norm{\nabla If}_{L_t^2L_x^\infty(J)}\norm{ If}_{L_t^2L_x^\infty(J)}\norm{\nabla Iv}_{L_t^\infty L_x^2(J)}\norm{ Iv}_{L_t^\infty L_x^2(J)}
\leq CF_2^\omega(J)2Z_I(J)\,.
\]
\end{enumerate}
\paragraph{\textbf{Two random terms.}}
\begin{enumerate}[]
\item \textbf{Case 1:} $u^{(1)}=u^{(2)}=f,u^{(3)}=v$. 
We deduce from Cauchy-Schwarz and from the almost conservation of the mass that ~\ref{lemma:mass} that 
\[
~\eqref{eq:incr-II}\leq \norm{\nabla If}_{L_t^2L_x^\infty(J)}^2\norm{ If}_{L_t^\infty L_x^2(J)}\norm{ Iv}_{L_t^\infty L_x^2(J)}
        \leq C F_2^\omega(J)2\,.
\]
\item \textbf{Case 2:} $u^{(1)}=v, u^{(2)}=u^{(3)}=f$, $u^{(2)}=u^{(3)}=f$. In this case, we have
\begin{multline*}
~\eqref{eq:incr-II} \leq C \norm{\nabla If}_{L_t^2L_x^\infty(J)}\norm{\nabla Iv}_{L_t^\infty L_x^2(J)}\norm{ If}_{L_t^\infty L_x^2(J)}\norm{ Iv}_{L_t^\infty L_x^2(J)}\\
\leq C F_2(J)\norm{ If}_{L_t^\infty L_x^2(J)}Z_I(J)\leq C F_2(J)Z_I(J)\,.
\end{multline*}
\end{enumerate}
\paragraph{\textbf{Three random terms.}} Here, we have $u^{(i)}=f$ for $i\in\set{1,2,3}$, so that
\[
~\eqref{eq:incr-II} \leq C\norm{\nabla If}_{L_t^2L_x^\infty(J)}^2\norm{If}_{L_t^\infty L_x^2(J)}^2
\leq C F_2^\omega(J)2\,.
\]
Collecting all the contributions, we obtain the desired estimate on the modified energy increment.
\end{proof}

\subsection{Modified interaction Morawetz}
We revisit the standard interaction Morawetz inequality in a perturbed setting, for the equation~\eqref{eq:nls:I} satisfied by $Iu$. First, we recall the formal computations that leads to the so-called Lin-Strauss inequality. Given a convex weight $a$, we define the action~\footnote{\ we sum over the indices, and $\partial_j f = \sum_{j=1}^3\partial_{x_j}f$.} 
\[
M_0(t)=2\im\int_{\R^3}\partial_ja(x)(\partial_j Iu)(x)\overline{Iu}(x)dx = 2\im\int_{\R^3}\partial_jaT_{0,j}dx\,.
\]
It corresponds to the contraction of the current density $T_{0,j}=T_{j,0}=2\im\parent{\partial_j(Iu)\overline{Iu}}$ against the vector field  $\nabla a$. Thanks to the convexity of $a$, we hope to get some monotonicity for $M_0$. The general virial identity reads 
\[
\frac{d}{dt} M_y(t)=\int_{\R^3}-\Delta^2a\abs{Iu}^2+4\partial_{kj}^2a\re(\partial_k(\overline{Iu})\partial_j(Iu))+\Delta a \abs{Iu}^4+2\re\parent{2\partial_ja\partial_j(Iu)H+\Delta a(Iu)H}dx\,.
\]
Given a point $y\in\R^3$, the recentered Morawetz action, denoted $M_y$, corresponds to the particular weight $a_y(x)=\abs{x-y}$, that satisfies 
\[
\partial_ja = \frac{x_j-y_j}{\abs{x-y}},\quad \Delta a = \frac{2}{\abs{x-y}},\quad \Delta^2a=-2\pi\delta_y,\quad \partial_{jk}^2a = \operatorname{I} - \frac{x-y}{\abs{x-y}}\otimes\frac{x-y}{\abs{x-y}} \eqqcolon \operatorname{P}\,.
\]
Hence, 
\begin{multline}
\label{eq:morawetz:action}
    -\frac{d}{dt} M_y(t) = \pi\abs{Iu(t,y)}^2+2\int_{\R^3}\abs{\operatorname{P}\nabla Iu(x)}^2dx+\int_{\R^3}\frac{1}{\abs{x-y}}\abs{Iu(x)}^4dx\\
    +4\re\int\frac{x_j-y_j}{\abs{x-y}}\partial_j(Iu)\overline{H}dx+4\re\int_{\R^3}\frac{1}{\abs{x-y}}Iu\overline{H}dx\,.
\end{multline}
Observe the presence of two extra terms, that come from the presence of the forcing term $H$. Next, we average the recentered Morawetz action against the mass density $\abs{Iu(y)}^2$ to obtain the Morawetz interaction, 
\[
\int_{\R^3}M_y(t)\abs{Iu(y)}^2dy\,.
\]
Note that this quantity is bounded by the conserved $L^2$-norm times the critical $\dot{H}^{1/2}$-norm of $Iu$.
\begin{equation}
\label{eq:bound:morawetz}
\abs{\int_{\R^3}M_y(t)\abs{Iu(y)}^2dy}\leq C\norm{Iu(t)}_{L_x^2}^2\norm{Iu(t)}_{\dot{H}^{1/2}}^2\,.
\end{equation}
We can now state and prove the modified Morawetz interaction, that will provide a spacetime control on $Iu$.
\begin{proposition}[Modified interaction Morawetz]
Assume that $J\subset J^*$ is an interval such that $Z_I(J_k)<\infty$. Then, for any $\delta>0$, there exists $C>0$ such that for any partition $J=\bigcup_{k=0}^{L-1}J_k$, there holds
\label{prop:morawetz}
\begin{multline}
    \label{eq:modified:morawetz}
    \iint_{J\times\R^3}\abs{Iu}^4dtdx
    \leq C\norm{Iu}_{L^\infty \dot{H}^{1/2}(J)}^2\norm{Iu}_{L^\infty L^2(J)}^2\\
    +CN^{-1+\delta}\norm{Iu}_{L_t^\infty L_x^2(J)}\norm{Iu}_{L_t^\infty H^{1/2}(J)}^2\parentbig{\sum_{k=0}^{L-1}\parent{Z_I(J_k)^2+F_2(J_k)^2}\parent{Z_I(J_k)+F_{\infty}(J_k)}}\,.
\end{multline}
\end{proposition}
\begin{proof}
By using the Leibniz rule, and then integrating over time, we obtain   
\begin{multline*}
\left[\int_{\R^3}M_y(t)\abs{Iu(y)}^2dy\right]_{t=t_0}^{t=t_1} = \underbrace{\iint_{J\times\R^3}\abs{Iu(t,y)}^2\parent{\frac{d}{dt}\int_{\R^3}\frac{x_j-y_j}{\abs{x-y}}T_{0,j}(x)dx}dydt}_{\operatorname{I}}\\
+\underbrace{\iiint_{J\times\R^3\times\R^3}\frac{x_j-y_j}{\abs{x-y}}T_{0,j}(x)2\re(\partial_tIu(y)\overline{Iu}(y))dxdydt}_{\operatorname{II}}\,.
\end{multline*}
Next, we plug~\eqref{eq:morawetz:action} into I and II. We have 
\begin{multline}
    \label{eq:morawetz:interactionI}
    \operatorname{I} = \pi\iint_{J\times\R^3}\abs{Iu(t,y)}^4dydt\\+2\iiint_{J\times\R^3\times\R^3}\abs{P_{x,y}\nabla u(t,x)}^2\abs{Iu(y)}^2dxdydt +\iiint_{J\times\R^3\times\R^3} \frac{2}{\abs{x-y}}\abs{Iu(x)}^4\abs{Iu(y)}^2dxdydt \\
    +\underbrace{4\re\iiint_{\R\times\R^3\times\R^3}\abs{Iu(y)}^2\frac{x_j-y_j}{\abs{x-y}}\partial_j(Iu)(x)\overline{H}(x)dxdydt}_{\operatorname{I}_1}\\
    +\underbrace{4\re\iint_{J\times\R^3\times\R^3}\frac{\abs{Iu(y)}^2}{\abs{x-y}}Iu(x)\overline{H(x)}dxdydt}_{\operatorname{I}_2}\,.
\end{multline}
To expand the term II we use equation~\eqref{eq:nls:I} and we write the time derivative of $Iu$ in terms of the perturbation $H$. 
\begin{multline}
    \label{eq:morawetz:interactionII}
    \operatorname{II} = -\iiint_{J\times\R^3\times\R^3}\frac{x_j-y_j}{\abs{x-y}}T_{0,j}(x)2\re(i\Delta Iu(y)\overline{Iu}(y))dxdydt\\
    +\underbrace{\iiint_{J\times\R^3\times\R^3}\frac{x_j-y_j}{\abs{x-y}}I_{0,j}(x)2\re(iIu(y)\overline{H}(y))dxdydt}_{\operatorname{II}_1}\,.
\end{multline}
Next, we observe that the terms which does not come from the perturbation have a sign. Hence, we can use the bound~\eqref{eq:bound:morawetz} and see that 
\[
\pi\iint_{J\times\R^3}\abs{Iu(t,y)}^4dydt\leq 2C\norm{Iu}_{L_t^\infty L_x^2(J)}^2\norm{Iu}_{L_t^\infty\dot{H}^{1/2}(J)}^2 + \abs{\operatorname{I}_1+\operatorname{I}_2+\operatorname{II}_1}\,.
\]
We are left to estimate three terms $\operatorname{I}_1\,, \operatorname{I}_2$ and $\operatorname{II}_1$, that contain a spacetime norm of the commutator $H$. Hence, we will use Lemma~\ref{lemma:H} to address these three terms. Up to a partition, and using sub-additivity, we may assume that $L=1$.

\textbf{Term} $\operatorname{I}_1$: We use the triangular inequality and Lemma~\ref{lemma:H} to get
\begin{multline*}
\abs{\operatorname{I}_1}\leq\norm{Iu}_{L_t^\infty L_y^2(J)}^2\int_J\underset{y}{\sup}\absbig{\int_{\R^3}\frac{x_j-y_j}{\abs{x-y}}\partial_j(Iu)(x)\overline{H}(x)dx}dt\\
\leq CN^{-1+\delta}\norm{Iu}_{L_t^\infty L_y^2(J)}^2\underset{y\in\R^3\,K\in2^\N}{\sup}\parent{K^{-1}{\norm{P_K\parent{\frac{x_j-y_j}{\abs{x-y}}\partial_j(Iu)(x)}}_{L_t^\infty L_x^2(J)}}}\\
\parent{Z_I(J)^2+F_2^\omega(J)2}\parent{Z_I(J)+F_{\infty}(J)}\,.
\end{multline*}
We fix $y\in\R^3$, $t\in J$ and $K\in2^\N$, and we proceed as follows.
\[
\norm{P_K\parent{\frac{x_j-y_j}{\abs{x-y}}\partial_j(Iu)(x)}}_{L_x^2} = \underset{\norm{g}_{L_2}\leq1}{\sup}\absbig{\int_{\R^3}\parent{\frac{x_j-y_j}{\abs{x-y}}\partial_j(Iu)(t,x)} \overline{P_Kg}(x) dx}
\]
We deduce from the Leibniz rule that for each $x\in\R^3$,
\[
\frac{x-y}{\abs{x-y}}\cdot \nabla \overline{Iu}(x) = \operatorname{div} \parent{\frac{x-y}{\abs{x-y}}\overline{Iu}(x)} - \frac{2}{\abs{x-y}}\overline{Iu}(x)\,.
\]
Next, we apply Hardy's inequality to handle the singular term with $\abs{x-y}^{-1}$, and we obtain
\begin{equation*}
\begin{split}
\norm{P_K\parent{\frac{x_j-y_j}{\abs{x-y}}\partial_j(Iu)(x)}}_{L_x^2} &\leq 
\underset{\norm{g}_{L_2}\leq1}{\sup}\absbig{\int_{\R^3}\frac{x-y}{\abs{x-y}}\cdot Iu(t,x)\overline{\nabla P_Kg}(x)+ Iu(t,x)\frac{2}{\abs{x-y}}\overline{P_Kg}(x)dx} \\
&\leq C\norm{Iu}_{L_x^2} \underset{\norm{g}_{L_2}\leq1}{\sup}\parent{K\norm{g}_{L_x^2}+\norm{\frac{1}{\abs{x-y}}P_Kg(x)}_{L_x^2}} \\
&\leq CK\norm{Iu}_{L_x^2}\,.
\end{split}
\end{equation*}
This addresses the estimate for the term $\operatorname{I}_1$: 
\[
\abs{\operatorname{I}_1}\leq CN^{-1+\delta}\norm{Iu}_{L_t^\infty L_x^2(J)}^3 \parent{Z_I(J)^2+F_2^\omega(J)2}\parent{Z_I(J)+F_{\infty}(J)}\,.
\]
\textbf{Term} $\operatorname{I}_2$: We proceed similarly. In this case, we are left to estimate 
\

\
\begin{equation*}
\begin{split}
\norm{P_K\parent{\frac{Iu(x)}{\abs{x-y}}}}_{L_x^2(J)} &= \underset{\norm{g}_{L_2}\leq1}{\sup}\absbig{\int_{\R^3}Iu(t,x)\frac{1}{\abs{x-y}}\overline{P_Kg}(x) dx}\\
&\leq \norm{Iu}_{L_x^2}\norm{\frac{1}{\abs{x-y}}g}_{L_x^2} \\
&\leq CK\norm{Iu}_{L_x^2}\norm{g}_{L_x^2}\,,
\end{split}
\end{equation*}
where we used Hardy's inequality in the last line. Hence, 
\[
\abs{\operatorname{I}_2}\leq CN^{-1+\delta}\norm{Iu}_{L_t^\infty L_x^2(J)}^3 \parent{Z_I(J)^2+F_2^\omega(J)2}\parent{Z_I(J)+F_{\infty}(J)}\,.
\]
\textbf{Term} $\operatorname{II}_1$:
    \begin{equation*}
        \begin{split} \abs{\operatorname{II_1}}&\leq \underset{t,y}{\sup}\absbig{\int_{\R^3}\frac{x_j-y_j}{\abs{x-y}}I_{0,j}(x)dx}\int_J\absbig{\int_{\R^3}Iu(y)H(y)dy}dt\\
            &\leq CN^{-1+\delta}\norm{Iu}_{L_t^\infty \cdot{H}^{1/2}(J)}^2\norm{Iu}_{L_t^\infty L_x^2(J)}\parent{Z_I(J)^2+F_2^\omega(J)2}\parent{Z_I(J)+F_\infty^\omega(J)}\,.
        \end{split}
    \end{equation*}
This concludes the proof of Proposition~\ref{prop:morawetz}.
\end{proof}
\section{Almost-sure global well-posedness and scattering}
In this section, we glue together the pieces collected in the previous sections to prove the main theorem. First, we take $\omega\in\Omega_\alpha$ such that
\[
F^\omega(\R)+F_\infty^\omega(\R)\leq C_\alpha\,.
\]
 Then, we apply the local probabilistic Cauchy theory developed in Section~\ref{section:lwp}. For such an $\omega$, we have a local solution $u = \e^{it\Delta}f_0^\omega+v$ to~\eqref{eq:nls} with initial data $u_{\lvert t=0}=f_0^\omega$, and $v$ satisfies~\eqref{eq:nls:f} with zero initial condition up to a maximal lifetime $T^*$. Next, assuming that $f_0$ is radially-symmetric, and taking $\omega\in\widetilde{\Omega}_\alpha$, we obtain in this section a uniform a priori estimates on $\norm{v(t)}_{H^\sigma}$, and we deduce from the stability theory from Section~\ref{sec:stability} that $T^*=+\infty$, and that $v$ scatters at infinity. For this purpose, we perform a double bootstrap argument with a modified interaction Morawetz, in order to establish global estimates on the modified energy $\mathcal{E}(v)$. 
\subsection{Double bootstrap argument}
\label{sec:double-bootstrap}
Recall that we fixed $\alpha>0$, $N_0(\alpha)$ large enough and $\widetilde{\Omega}_\alpha$ with $\mathbb{P}(\Omega_\alpha\setminus\Omega_\alpha)\leq\alpha$, such that for every $\omega\in\Omega_\alpha$ and $N\geq N_0(\alpha)$, there holds~\eqref{eq:ass:data:2}. Furthermore, we can assume from the large deviation estimate on the $L_x^p$ norm  of the randomized initial data (see~\cite{benyi2015-local}, Lemma 2.3), and from the fact that the I-operator commutes with the randomization and that it is bounded on $L^2$, that for all $\omega\in\widetilde{\Omega}_\alpha$ and all $N$,
\begin{equation}
\label{eq:ass:L4}
\norm{If_0^\omega}_{L_x^4}\leq C_\alpha\,.
\end{equation}

Then, we write $C$ the constant that appears in the modified Morawetz inequality~\eqref{eq:modified:morawetz}, and we define
\begin{equation}
    \label{eq:choice:M}
\mathcal{M}_0\coloneqq \norm{u}_{L^\infty L^2(\intervalco{0}{T^*})}\,,\quad M = 4C\mathcal{M}_0^3\,.
\end{equation}
It follows from the almost conservation of the mass from Lemma~\ref{lemma:mass} that $\mathcal{M}_0$ does only depend on $\norm{f_0}_{L_x^2}$. Afterwards, we consider the set
\[\Theta_N=\set{T\in\intervalco{0}{T^*}\mid \underset{t\in\intervalcc{0}{T}}{\sup} \mathcal{E}(v(t))\leq N^{2(1-\delta)},\ \norm{Iv}_{L_{t,x}^4\intervaloo{0}{T}}^4\leq MN^{1-\sigma}}.\]
We claim that for every $6/7<\sigma<2s\leq1$, there exists  $N_0'(\epsilon,\norm{f_0}_{L_x^2(\R^3)},s,\sigma)$ to be defined later such that for all $N\geq \max(N_0,N_0')$, we have 
\begin{equation}
\label{eq:thetaM}
    \Theta_N = \intervalco{0}{T^*}\,.
\end{equation} 
Before diving into the proof, we note that $\Theta_N$ is nonempty. Indeed, it follows from~\eqref{eq:ass:L4} and from $v(0)=0$ that $\mathcal{E}(v)(0)=\frac{1}{4}\norm{If_0^\omega}_{L_x^4}^4\leq C_\alpha \ll N^{2(1-\sigma)}$, and hence $T=0\in\Theta_N$. Moreover, we see from a continuity argument that it is also closed. Hence, it remains to prove that $\Theta_N$ is open. Let us fix $T\in\Theta_N$, with $T<T^*$. By the local well-posedness theory for~\eqref{eq:nls:f} presented in Proposition~\ref{prop:lwp}, there exists a small $T_2>T$ such that 
\begin{equation}
\label{eq:2M}
    \norm{Iv}_{L_{t,x}^4\intervaloo{0}{T+\delta}}^4\leq 2MN^{1-\sigma}\,.
\end{equation}
The double bootstrap argument is the following. First, by using the a priori estimate~\eqref{eq:2M}, by tuning the parameter and by summing over the spacetime slabs where
\[\disp\norm{Iv}_{L_{t,x}^4\intervaloo{T_1}{T}}^4\leq\epsilon_N\,,\] we observe from the local-in-time energy increment of Proposition~\ref{prop:energy:loc-incr} that if a time $T_1$ is such that $\mathcal{E}(v(T_1))\leq N^{2(1-\sigma)}$, where $\epsilon_N~\sim N^{-2(1-\sigma)}$ is as in Lemma~\ref{lemma:Z}, then the energy remains less than $N^{2(1-\sigma)}$ as long as we control the $L_{t,x}^4$ norm of $Iv$. Next, we deduce from the modified interaction Morawetz~\eqref{eq:modified:morawetz} that we eventually have~\eqref{eq:2M} with $MN^{1-\sigma}$ on the right-hand side, instead of $2MN^{1-\sigma}$.
\begin{proposition}[Double bootstrap argument]
\label{prop:E:global:increment}
Let $6/7<\sigma<2s$. There exists $\widetilde{N}_0(\epsilon,\sigma,s,\norm{f_0}_{L^2})$ such that for all $N\geq \max(N_0,\widetilde{N}_0)$, if $T_1$ is in $\Theta_N$ and $T_1<T_2<T^*$ is such that 
\[
\norm{Iv}_{L_{t,x}^4(0,T_2)}^4\leq 2MN^{1-\sigma}\,,
\]
then $T_2$ is also in $\Theta_N$.
\end{proposition}
\begin{proof}
Given a fixed $N\geq N_0$, we consider $\epsilon_N= N^{-4(1-\sigma)-\delta}$ as in Lemma~\ref{lemma:Z}, where $\delta>0$ is arbitrarily small. As a consequence of the assumption $\norm{Iv}_{L_{t,x}^4\intervalcc{0}{T_2}}\leq 2MN^{1-\sigma}$, there are $L=2MN^{1-\sigma}\epsilon_N^{-1}$ spacetime-slabs on which the smallness condition $\norm{Iv}_{L_{t,x}^4}^4\leq \epsilon_N$ is satisfied. We decompose therefore $\intervalco{0}{T_2}$ into $L$ intervals $J_k=\intervalco{t_k}{t_{k+1}}$ such that for $\disp k\in\set{0,\dots,L-1}$
\[\norm{Iv}_{L_{t,x}^4(J_k)}^4\leq\epsilon_N\,,\]
with 
\[0=t_0<t_1<\dots<t_L =T_2,\quad \text{and}\ L=\frac{2MN^{1-\sigma}}{\epsilon_N}=2MN^{3(1-\sigma)+\delta}\,.\]
Up to some extra decomposition of these intervals into a finite number of intervals that does only depend on $\norm{f_0}_{H^s}$ and on $\eta$, a universal constant that controls the  $L_t^4L_x^3$ norm of $f$ on $J_k$ used in the modified well-posedness result from Lemma~\ref{lemma:Z}. Hence, the number of spacetime-slabs on which we have all the smallness condition we need is bounded from above by
\[
L\leq C(f_0,\eta)\frac{2M}{\epsilon_N}\leq C(f_0,\eta)2MN^{3(1-\sigma)+\delta}\,.
\]
Now, by iteration on $k$, let us deduce from Proposition~\ref{prop:energy:loc-incr} that 
\begin{equation}
\label{eq:bootstrap:E}
    \underset{t\in\intervalcc{0}{t_k}}{\sup}\mathcal{E}(v(t))\leq N^{2(1-\sigma)}\,.
\end{equation}
We note from the definition of $\Theta_N$ that~\eqref{eq:bootstrap:E} is already known when $t_k\leq T_1$, since we assumed that $T_1\in\Theta_N$. After that, we consider $k$ such that~\eqref{eq:bootstrap:E} holds for $t_k$. We prove that it still holds for any $t\in J_k$. To do so, we fix $t\in J_k$ and we use sub-additivity:
\[
\mathcal{E}(v(t))\leq \mathcal{E}(v(0)))+\int_0^{t_{k+1}}\absbig{\frac{d}{dt}\mathcal{E}(Iv(\tau))}d\tau= \mathcal{E}(v(0))) + \sum_{0\leq i\leq k-1}\int_{J_i}\absbig{\frac{d}{dt}\mathcal{E}(v(\tau))}d\tau\,.
\]
Initially, we know from~\eqref{eq:ass:L4} that $\mathcal{E}(v(0))\leq C_\alpha$. Moreover, we have by construction that $\mathcal{E}(Iv(t_i))\leq N^{2(1-\sigma)}$, and $\norm{Iv}_{L_{t,x}^4(J_i)}^4\leq\epsilon_N$ for each $i\in\set{0,\dots,k}$. Hence, we can apply the modified energy increment~\eqref{eq:E:loc-incr} on each interval $J_i$, and get
\[
\mathcal{E}(v(t))\leq C_\alpha + C_0\sum_{0\leq i\leq k-1}\parent{N^{6(1-\sigma)-1}+N^{1-\sigma}F_2(J_i)^2+N^{1-\sigma}F_2(J_i)\norm{Iv}_{L_{t,x}^4(J_i)}^2}\,,
\]
where $C_0=C_0(f_0,s,\sigma)$ is a universal constant that arises in~\eqref{eq:E:loc-incr}. Next, we brutally multiply the first term by the number of spacetime-slabs~\footnote{\ If we prove a long time Strichartz estimate as in~\cite{dodson19}, we might be able to use also some sub-additivity to estimate the first term, and to loosen therefore our restriction on $s$.}, and we use sub-additivity for the second term and the third term. When summing over the spacetime slabs, we also apply Cauchy-Schwarz on the third term, and we obtain
\begin{multline*}
     \mathcal{E}(v(t))\lesssim N^{\gamma}+LN^{6(1-\sigma)-1}\\
     +N^{1-\sigma}\biggl\{\sum_{0\leq i\leq k-1}F_{2}(J_j)^2+\biggl(\sum_{0\leq i\leq k-1}F_2(J_i)^2\biggr)^{1/2}\biggl(\sum_{0\leq i\leq k-1}\norm{Iv}_{L_{t,x}^4(J_j)}^4\biggr)^{1/2}\biggr\}\,.
\end{multline*}
We conclude by using the upper bound on $L$ and the estimate~\eqref{eq:ass:data:2} on $F_2$  that
\begin{equation*}
\begin{split}
\mathcal{E}(v(t))&\leq C_\alpha+C_0\parent{LN^{6(1-\sigma)-1}+N^{1-\sigma}\parent{F_2^\omega(\R)^2+F_2^\omega(\R)(2MN^{1-\sigma})^{1/2}}}\\
&\leq C_\alpha+ C_0\parent{2MN^{9(1-\sigma)-1+\delta}+N^{1-\sigma}\parent{N^{-\gamma}N^{1-\sigma}+ N^{0-}N^\frac{1-\sigma}{2}N^\frac{1-\sigma}{2}}}\\
&\leq N^{2(1-\sigma)}\parent{C_\alpha N^{-2(1-\sigma)}+2C_0MN^{7(1-\sigma)-1+\delta}+N^{-2\gamma}+N^{-\gamma}}\\
&\leq N^{2(1-\sigma)}\,,
\end{split}
\end{equation*}
provided that $\disp 6/7<\sigma$, and $N\geq N_0'=N_0'(C_0,\gamma,M)$. This finishes the first part of the proof of Proposition~\ref{prop:E:global:increment}. To show that $T_2\in\Theta_N$, it remains to prove that
\[
\iint_{\intervalco{0}{T_2}\times\R^3}\abs{Iv}^4dtdx \leq MN^{1-\sigma}\,.
\]
For this purpose, we perform the same analysis as in the first part of the proof to handle the remainders that contain the commutator $H$. Namely, we decompose $\intervalcc{0}{T_2}$ into $L$ intervals $\set{J_k}_{0\leq k\leq L-1}$ such that~\eqref{eq:ass:4} and ~\eqref{eq:ass:4f} holds on each $J_k$. Since we already know from Proposition~\eqref{prop:E:global:increment} that 
\[\underset{t\leq T_2}{\sup}\mathcal{E}(v(t))\leq N^{2(1-\sigma)}\,, \]
we can deduce from Lemma~\ref{lemma:Z} that for each $k$
\[Z_{N,J_k}(v)\leq CN^{1-\sigma}\,.\]
Therefore, it follows from the modified interaction Morawetz that ~\eqref{eq:modified:morawetz}
\begin{multline*}
\iint_{\intervalco{0}{T_2}\times\R^3}\abs{Iu}^4dtdx\leq C\parent{\norm{Iv}_{L_t^\infty(\intervalcc{0}{T_2}; \dot{H}^{1/2})}^2+\norm{If}_{L_t^\infty(\intervalcc{0}{T_2}; \dot{H}^{1/2})}^2}\norm{Iu}_{L_t^\infty(\intervalcc{0}{T_2}; L_x^2)}^2\\
+CN^{-1+\delta}L\norm{Iu}_{L_t^\infty L_x^2\intervalcc{0}{T_2}}\norm{Iu}_{L_t^\infty \dot{H}_x^\frac{1}{2}\intervalcc{0}{T_2}}^2\parent{N^{2(1-\sigma)}+C_\alpha^2N^{1-\alpha}}\parent{N^{1-\alpha}+C_\alpha}\,.
\end{multline*}
As a consequence of the global energy estimate on $\intervalcc{0}{T_2}$ shown in the first part of the proof, we deduce from interpolation and from the coercivity of $\mathcal{E}(v)$ that 
\begin{equation}
\label{eq:H1/2IV}
\norm{Iv}_{L_t^\infty\dot{H}^\frac{1}{2}\intervalcc{0}{T_2}}^2\leq\norm{Iv}_{L_t^\infty\dot{H}^{1}\intervalcc{0}{T_2}}\norm{Iv}_{L_t^\infty L_x^2\intervalcc{0}{T_2}}\leq \sqrt{2}\mathcal{E}(v)^\frac{1}{2}\mathcal{M}_0 \leq \sqrt{2}\mathcal{M}_0N^{1-\sigma}\,.
\end{equation}
Moreover, 
\begin{equation}
\label{eq:H1/2If}
\norm{If}_{L^\infty \dot{H}^{1/2}}^2\leq C N^{1-2s}\norm{f_0}_{H^s}^2\leq C(f_0)N^{0-} N^{1-\sigma}\,,
\end{equation}
for some irrelevant constant $C(f_0)=C(\norm{f_0}_{H^s},s,\sigma)$.
We deduce from the almost conservation of the mass, and from estimates~\eqref{eq:H1/2IV},~\eqref{eq:H1/2If} that
\begin{multline}
C\parent{\norm{Iv}_{L_t^\infty(\intervalcc{0}{T_2}; \dot{H}^{1/2})}^2+\norm{If}_{L_t^\infty(\intervalcc{0}{T_2}; \dot{H}^{1/2})}^2}\norm{Iu}_{L_t^\infty(\intervalcc{0}{T_2}; L_x^2)}^2\\
\leq
C\mathcal{M}_0^3\parent{\sqrt{2}N^{1-\sigma}+C(f_0)^{1-\sigma-\delta}}\leq 2C\mathcal{M}_0^3N^{1-\sigma}\leq \frac{M}{2}\mathcal{M}_0^3N^{1-\sigma}\,,
\end{multline}
provided that $6/7<\sigma$, and that $\delta$ is taken sufficiently small. We conclude that
\[
\iint_{\intervalco{0}{T_2}\times\R^3}\abs{Iu}^4dtdx\leq\frac{M}{2}N^{1-\sigma}+N^{1-\sigma}N^{0-}\parent{C_0+C_0'}\,.
\]
Moreover, using the bound~\footnote{We actually control $F$ defined with $f$ instead of $If$. But we can easily replace $f$ by $If$ in the proof, if we replace $f_0$ by $If_0$ using the fact $I$ is a Fourier multiplier and commutes therefore with the randomization. Moreover, the operator norm of $I$ acting on $L^2$ is less than 1. Hence, we do not lose any power of $N$.} on $F$ written in ~\eqref{eq:ass:data:def},
\begin{multline*}
\iint_{\intervalco{0}{T_2}\times\R^3}\abs{Iu}^4dtdx\lesssim\iint_{\intervalco{0}{T_2}\times\R^3}\abs{If}^4dtdx+\iint_{\intervalco{0}{T_2}\times\R^3}\abs{Iu}^4dtdx \\
\leq 4F_{\infty}^\omega(\intervaloo{0}{T_2})^4+\frac{M}{2}N^{1-\sigma}+4N^{0-}(C_0+C_0')N^{1-\sigma}\leq MN^{1-\sigma}\,.
\end{multline*}
In the last estimate, we took $N\geq N_0"=N_0"(f_0)$ large enough. This concludes the proof of Proposition~\ref{prop:E:global:increment} with $\widetilde{N}_0=\max(N_0',N_0")$.
\end{proof}
\subsection{Proof of Theorem~\ref{theorem:main}}
We proved that for every $\alpha>0$ and $\omega\in\Omega_\alpha$ as in Lemma~\ref{eq:ass:data:2}, the Cauchy problem~\eqref{eq:nls:f} admits a unique local solution $v$, with a maximal lifespan $\intervalco{0}{T^*}$. Moreover, we find a set $\widetilde{\Omega}_\alpha$ such that $\mathbb{P}(\Omega_\alpha\setminus\widetilde{\Omega}_\alpha)\leq\alpha$, where we proved in Proposition~\ref{prop:E:global:increment} by a double bootstrap argument on $\mathcal{E}(v)$ and a modified Morawetz estimate that if $\omega\in\widetilde{\Omega}_\alpha$, there exists $N_0=N_0(\epsilon,\norm{f_0}_{H^s},s,\sigma)$ such that for all $N\geq N_0$, we have $\Theta_N = \intervalco{0}{T^*}$. In particular, 
\[\underset{0\leq t\leq T^*}{\sup}\mathcal{E}(v(t))\leq N^{2(1-\sigma)}\,.\]
Hence, we deduce from the above estimate on the coercive modified energy, from the operator bound~\eqref{eq:norm:I} for the I-operator and from the almost conservation of the mass, that there exists $N$ large enough and a constant $C_N$ such that  
\[\underset{0\leq t\leq T^*}{\sup}\norm{v}_{H_x^\sigma}\leq \norm{v}_{L^\infty(\intervalcc{0}{T^*};L^2)}+ CN^{1-\sigma}\norm{Iv}_{L^\infty(\intervalcc{0}{T^*};\dot{H}^1)}\leq CN^{3(1-\sigma)}\eqqcolon C_N.\]
In what follow we drop the dependence on $N$. Since we have a global a priori estimate on the $H^\sigma$ norm of $v$, we can apply the stability theory presented in Section~\ref{section:lwp}, and deduce from Proposition~\ref{prop:uniform:scattering} that $T^*=+\infty$, and that $v$ scatters at infinity. Finally, we take
\[
\Sigma = \bigcap_{\epsilon>0}\widetilde{\Omega}_\alpha\,.
\]
Since $\mathbb{P}(\Omega\setminus\widetilde{\Omega}_\alpha)\leq2\alpha$ for every $\alpha>0$, we have that $\mathbb{P}(\Sigma)=1$ and, by construction, a unique global scattering solution to~\eqref{eq:nls} with initial data $u(0)=f_0^\omega\in H^s$ for all $\omega\in\Sigma$.
\printbibliography[
title={References}
]

@article{bourgain94,
  author  = "Bourgain, Jean",
  title   = "Periodic nonlinear Schrödinger equation and invariant measures",
  year    = "1994",
  journal = "Commun. Math. Phys.",
  volume  = "166",
  pages   = "1--26"
}

@article{bourgain96-gibbs,
 author = " Bourgain, Jean",
 title = "Invariant measures for the 2{D}-defocusing nonlinear {S}chr{\"o}dinger equation",
 journal = "Commun. Math. Phys.",
 issn = "0010-3616; 1432-0916/e",
 volume = "176",
 number = "2",
 pages = "421-445",
 year = "1996",
 publisher = "Springer, Berlin/Heidelberg",
}

@article{bourgain-98,
author = "Bourgain, Jean",
title = "Refinements of Strichartz' inequality and applications to 2D-NLS with critical nonlinearity",
Fjournal = "International Mathematics Research Notices",
journal = "Internat. Math. Res. Notices",
volume = "1998",
number = "5",
year = "1998",
pages = "253-283",
}

@article{bourgain-98-3D,
author = "Bourgain, Jean",
title = "Scattering in the energy space and below for 3D NLS ",
Fjournal = "Journal d’Analyse Math{\'e}matique ",
journal = "J. Anal. Math.",
year = "1998",
volume = "75",
pages = "267 - 297",
}

@article{benyi2015,
 author = " B{\'e}nyi, {\'A}rp{\'a}d and  Oh, Tadahiro and  Pocovnicu, Oana",
 Title = "On the probabilistic {C}auchy theory of the cubic nonlinear {S}chr{\"o}dinger equation on \(\mathbb {R}^d\), \(d \geq 3\)",
 FJournal = "{Trans. Am. Math. Soc., Ser B}",
 journal = "Trans. Am. Math. Soc., Ser. B",
 volume = "2",
 pages = "1-50",
 year = "2015",
 publisher = "American Mathematical Society (AMS), Providence, RI",
}

@inbook{benyi2015-local,
   author = " B\'enyi, \'Arp\'ad and  Oh, Tadahiro and  Pocovnicu, Oana",
   booktitle="Excursions in Harmonic Analysis, Volume 4: The February Fourier Talks at the Norbert Wiener Center",
   title = "Wiener randomization on unbounded domains and an application to almost sure well-posedness of {NLS}",
   publisher = "Springer International Publishing",
   year = "2015",
   pages = "3-25",
   isbn="978-3-319-20188-7",
   doi="10.1007/978-3-319-20188-7_1",
}

@inbook{benyi2019,
author = " B\'enyi, \'Arp\'ad and  Oh, Tadahiro and  Pocovnicu, Oana",
title = "On the Probabilistic {C}auchy Theory for Nonlinear Dispersive {PDE}s",
publisher = "Springer International Publishing",
year = "2019",
pages = "1-32",
booktitle = "Landscapes of Time-Frequency Analysis",
isbn="978-3-030-05210-2",
doi="10.1007/978-3-030-05210-2_1"
}

@article{benyi-oh-pocovnicu-2019,
title = "Higher order expansions for the probabilistic local Cauchy theory of the cubic nonlinear Schr{\"o}dinger equation on {$\mathbb{R}^3$}",
author = "{\'A}rp{\'a}d B{\'e}nyi and Tadahiro Oh and Oana Pocovnicu",
year = "2019",
doi = "10.1090/btran/29",
volume = "6",
pages = "114-160",
journal = "Trans. Amer. Math. Soc. Ser. B",
issn = "2330-0000",
}

@article{burq-thomann2020,
title="Almost sure scattering for the one dimensional nonlinear Schr{\"o}dinger equation", 
author=" Burq, Nicolas and  Thomann, Laurent",
year="2020",
journal = "arXiv e-prints, 2012.13571",
}

@article{burq-tzvetkov-2008I,
author=" Burq, Nicolas
and Tzvetkov, Nikolay",
title="Random data {C}auchy theory for supercritical wave equations {I}: local theory",
journal="Invent. Math.",
year="2008",
volume="173",
number="3",
pages="449-475",
publisher ="Springer",
issn="1432-1297",
doi="10.1007/s00222-008-0124-z",
}

@article{burq-tzvetkov-2008II,
author=" Burq, Nicolas
and Tzvetkov, Nikolay",
title="Random data {C}auchy theory for supercritical wave equations {II}: a global existence result",
journal="Invent. Math.",
year="2008",
volume="173",
number="3",
pages="477--496",
publisher = "Springer",
issn="1432-1297",
doi="10.1007/s00222-008-0123-0",
}

@article{christ-colliander-tao,
title = "Ill-posedness for nonlinear Schr{\"o}dinger and wave equations",
author = "Christ, Michael and Colliander, James and Tao, Terence",
journal = "arXiv-eprint, 0311048 ",
year = "2003",
% eprint = "0311048",
% archivePrefix ="arXiv",
% primaryClass="math.AP",
}

@article{ckstt-2004,
author = {Colliander, J. and Keel, M. and Staffilani, G. and Takaoka, H. and Tao, T.},
title = {Global existence and scattering for rough solutions of a nonlinear Schrödinger equation on {$\R^3$} },
journal = {Commun. Partial Differ. Equ.},
volume = {57},
number = {8},
pages = {987-1014},
% doi = {https://doi.org/10.1002/cpa.20029},
year = {2004}
}

@article{coifman-meyer-78,
author = "Meyer, Yves and Coifman, R.",
title = "Commutateurs d’int{\'e}grales singuli{\`e}eres et op{\'e}rateurs multilin{\'e}eaires",
journal = "Ann. Inst. Fourier (Grenoble)",
year = "1978",
volume = "28",
pages = "177-202",
}

@article{da-prato02,
author = "Da Prato, G. and Debussche, A",
title = "Two-dimensional Navier-Stokes equations driven by a space-time white
noise", 
journal = "J. Funct. Anal.",
volume = "196",
year = "2002",
number = " 1",
pages = "180-210",
}

@article{dodson-13,
 Author = {Benjamin {Dodson}},
 Title = {{Global well-posedness and scattering for the defocusing, cubic nonlinear Schr\"odinger equation when \(n = 3\) via a linear-nonlinear decomposition}},
 FJournal = {{Discrete and Continuous Dynamical Systems}},
 Journal = {{Discrete Contin. Dyn. Syst.}},
%  ISSN = {1078-0947; 0133-0189; 1553-5231/e},
 Volume = {33},
 Number = {5},
 Pages = {1905--1926},
 Year = {2013},
 Publisher = {American Institute of Mathematical Sciences (AIMS), Springfield, MO},
%  Language = {English},
%  MSC2010 = {35Q55},
%  Zbl = {1277.35316}
}

@article{colliander-oh-2012,
title = "Almost sure well-posedness of the cubic nonlinear Schrödinger equation below {$L^2(\mathbb{T})$}",
author = "Colliander, James and Oh, Tadahiro", 
year = "2012",
journal = "Duke Math. J.",
number = "3",
volume = "161",
page = "367–414", 
}

@article{local-smoothing,
author ="Constantin, P and Saut, J.-C.",
title = "Local smoothing properties of dispersive equations",
journal = "J. Amer. Math. Soc.",
year = "1988",
volume = "413-439.", 
volume = "1",
number = "2",
pages = "413–439.",
}

@article{dodson19,
 Author = {Benjamin {Dodson}},
 Title = {{Global well-posedness and scattering for nonlinear Schr\"odinger equations with algebraic nonlinearity when \(d = 2,3\) and \(u_0\) is radial}},
 FJournal = {{Cambridge Journal of Mathematics}},
 Journal = {{Camb. J. Math.}},
%  ISSN = {2168-0930; 2168-0949/e},
 Volume = {7},
 Number = {3},
 Pages = {283--318},
 Year = {2019},
 Publisher = {International Press of Boston, Somerville, MA},
%  Language = {English},
%  MSC2010 = {35Q55 35Q41 35A01 35A02 35P25},
%  Zbl = {1428.35502}
}

@Article{dodson-luhrmann-mendelson-19,
 Author = {Benjamin {Dodson} and Jonas {L\"uhrmann} and Dana {Mendelson}},
 Title = {{Almost sure local well-posedness and scattering for the 4D cubic nonlinear Schr\"odinger equation}},
 FJournal = {{Advances in Mathematics}},
 Journal = {{Adv. Math.}},
 ISSN = {0001-8708},
 Volume = {347},
 Pages = {619--676},
 Year = {2019},
 Publisher = {Elsevier (Academic Press), San Diego, CA},
 Language = {English},
%  MSC2010 = {35Q55 35B40 35B25 35P25 35L05 35R60},
%  Zbl = {1428.35503}
}

@article{dodson-luhrmann-mendelson-20,
Author = {Benjamin {Dodson} and Jonas {L\"uhrmann} and Dana {Mendelson}},
Title = {{Almost sure scattering for the {$4D$} energy-critical defocusing nonlinear wave equation with radial data}},
FJournal = {{American Journal of Mathematics}},
Journal = {{Am. J. Math.}},
% ISSN = {0002-9327; 1080-6377/e},
Volume = {142},
Number = {2},
Pages = {475--504},
Year = {2020},
Publisher = {Johns Hopkins University Press, Baltimore, MD},
Language = {English},
% MSC2010 = {35L71 37L15},
% Zbl = {1440.35217}
}

@article{fan-mendelson-20,
      title={Construction of {$L^2$} log-log blowup solutions for the mass critical nonlinear Schr\"odinger equation}, 
      author={Fan, Chenjie  and Mendelson, Dana},
      year={2020},
      journal = {arXiv-eprint, 2010.07821},
      eprint={2010.07821},
      archivePrefix={arXiv},
      primaryClass={math.AP},
}

@article{gubinelli-koch-oh2021,
   title={Global Dynamics for the Two-dimensional Stochastic Nonlinear Wave Equations},
   journal={Internat. Math. Res. Notices},
   author={Gubinelli, Massimiliano and Koch, Herbert and Oh, Tadahiro and Tolomeo, Leonardo},
   year={2021},
}

@article{hadac2009,
title = {Well-posedness and scattering for the {KP-II} equation in a critical space},
author = {Hadac, Martin and Herr, Sebastian and Koch, Herbert},
journal = {Ann. Inst. H. Poincar{\'e} Anal. Non Lin{\'e}aire},
volume = {26},
number = {3},
pages = {917-941},
year = {2009},
issn = {0294-1449},
doi = {https://doi.org/10.1016/j.anihpc.2008.04.002},
}

@article{herr2011,
author = "Herr, Sebastian and Tataru, Daniel and Tzvetkov, Nikolay",
fjournal = "Duke Mathematical Journal",
journal = "Duke Math. J.",
doi = "10.1215/00127094-1415889",
number = "2",
pages = "329-349",
publisher = "Duke University Press",
title = {Global well-posedness of the energy-critical nonlinear {S}chrödinger equation with small initial data in {$H^1(\mathbb{T}^3)$}},
url = "https://doi.org/10.1215/00127094-1415889",
volume = "159",
year = "2011"
}

@article{killip-murphy-visan-2019,
author = {Rowan Killip and Jason Murphy and Monica Visan},
title = {Almost sure scattering for the energy-critical NLS with radial data below {$H^1(\mathbb{R}^4)$}},
journal = {Commun. Partial Differ. Equ.},
volume = {44},
number = {1},
pages = {51-71},
year  = {2019},
publisher = {Taylor & Francis},
doi = {10.1080/03605302.2018.1541904},
}

@article{kenig-merle-10,
author = "Kenig, Carlos E and Merle, Frank",
title = "Scattering for {$\dot{H}^{1/2}$} bounded solutions to the cubic, defocusing {NLS} in 3 dimensions",
journal = "Trans. Amer. Math. Soc.",
year = "2010",
volume = "362",
number = "4",
pages = "1937-1962",
}

@article{kenig-ponce-vega-00,
author = "Kenig, Carlos E and Ponce, Gustavo and Vega, Luis",
title = "Global well-posedness for semi-linear wave equations ",
journal = "Comm. Partial Diff. Equ.",
year = "2000",
volume = "25",
number = "9-10",
pages = "1741--1752 ",
}

@article{luhrmann-mendelson-14,
author = "L{\"u}hrmann, Jonas and Mendelson, Dana ",
title = "Random data Cauchy theory for nonlinear wave equations of power-type on {$\R^3$} ",
journal = "Comm. Partial Differ. Equ.",
year = "2014",
volume = "39",
number = "12",
pages = "2262-2283",
}

@article{oh-okamoto-19,
title = "On the probabilistic well-posedness of the nonlinear Schr{\"o}dinger equations with non-algebraic nonlinearities",
author = "Tadahiro Oh and Mamoru Okamoto and Oana Pocovnicu",
year = "2019",
volume = "39",
pages = "3479-3520",
journal = "Discrete and Contin. Dyn. Syst.-Series A",
number = "6",
}

@article{qingtang-11,
author = "Su, Qingtang",
year = "2012",
title = "Global well-posedness and scattering for defocusing, cubic {NLS} in
{$\mathbb{R}^3$}",
volume = "19",
number = "2", 
journal = "Math. Res. Lett.",
pages = "431-451"
}

@article{prt2014,
title = "Probabilistic well-posedness for the supercritical nonlinear harmonic oscillator",
author = "Poiret, Aur\'elien and Robert, Didier and Thomann, Laurent",
year = "2014",
journal = "Anal. PDE",
volume = "7", 
number = "4", 
pages = "997-1026",
}

@article{poiret2012,
title = "Solutions globales pour l’{\'e}quation de Schr{\"o}dinger cubique en dimension 3 ",
author = "Poiret, Aur\'elien",
year = "2012",
journal = "arXiv-eprint, 1207.1578",
eprint = "1207.1578",
archivePrefix="arXiv",
primaryClass = "math.AP",
}

@article{sun-2015,
author = "Sun, Chenmin and Xia, Bo ",
title = "Probabilistic well-posedness for supercritical wave equations with periodic boundary condition on dimension three",
year = "2016",
journal = "Illinois J. Math.",
volume = "60",
number = "2",
pages = "481–503",
}

\end{document}